\documentclass[11pt]{article}
\usepackage{graphicx} 

\usepackage[a4paper,left=0.85in,right=0.85in,top=0.85in,bottom=1.15in]{geometry}
\usepackage{graphicx}

\graphicspath{{figures/}}

\usepackage{caption}
\usepackage[english]{babel}
\usepackage[utf8]{inputenc}

\usepackage{amssymb}
\usepackage{amsmath}
\usepackage{amsthm}
\usepackage{amsfonts}

\usepackage[scr=dutchcal]{mathalpha}

\usepackage{hyperref}
\usepackage{cleveref}
\usepackage{algorithm,algorithmic}

\usepackage[normalem]{ulem}
\usepackage{cancel}

\usepackage{tikz-cd}
\usepackage{quiver}

\usepackage{color}
\usepackage{xcolor}
\definecolor{FTChampagnePink}{HTML}{F2DFCE}
\definecolor{FTOldLace}{HTML}{FFF1E0}
\definecolor{FTFloralWhite}{HTML}{FFF9F5}
\definecolor{FTMetallicSeaweed}{HTML}{0D7680}
\definecolor{FTMaroon}{HTML}{8F223A}
\definecolor{BBGreen}{HTML}{18453B}

\makeatletter
\newcommand{\globalcolor}[1]{%
  \color{#1}\global\let\default@color\current@color
}
\makeatother

\usepackage{booktabs}
\usepackage{diagbox}

\usepackage{cancel}
\usepackage[textsize=footnotesize]{todonotes}

\usepackage{enumitem}
\usepackage{comment}
\usepackage{thmtools, thm-restate}

\usepackage{blkarray, bigstrut}
\usepackage[square,sort,comma,numbers]{natbib}

\DeclareMathOperator{\iden}{\mathrm{id}} 
\DeclareMathOperator{\dm}{\mathrm{d}}
\DeclareMathOperator{\Z}{\mathbb{Z}}
\DeclareMathOperator{\R}{\mathbb{R}}
\DeclareMathOperator{\RP}{\mathbb{R}P}

\DeclareMathOperator{\X}{\mathbb{X}}

\DeclareMathOperator{\Sbb}{\mathbb{S}}

\DeclareMathOperator{\mathstar}{St}

\DeclareMathOperator{\SCpx}{\text{\textbf{SCpx}}}
\DeclareMathOperator{\TopS}{\text{\textbf{Top}}}

\DeclareMathOperator{\Grpd}{\text{\textbf{Grpd}}}

\DeclareMathOperator{\AbGrp}{\text{\textbf{Ab}}}

\newcommand{\ChGrpd}[2]{\fun{\mathcal{C}_{#1}}{#2}}
\newcommand{\HGrpd}[2]{\fun{\mathcal{H}_{#1}}{#2}}

\newcommand{\inv}[1]{#1^{-1}}
\newcommand{\fibre}[2]{\inv{#1}{\left({#2}\right)}}

\newcommand{\openball}[2]{B_{#1}{\left({#2}\right)}}

\newcommand{\rbracket}[1]{\mathopen{}\left({#1}\right)\mathclose{}}

\newcommand{\fun}[2]{{#1}\rbracket{#2}}
\newcommand{\sett}[1]{\left\{ {#1} \right\}}
\newcommand{\settt}[2]{\left\{{#1} \ \mid \ {#2}\right\}}

\newcommand{\dmet}[2]{\fun{\dm}{{#1},{#2}}}

\newcommand{\fungroup}[1]{{\pi_1} ({#1})}
\newcommand{\fungroupbased}[2]{\fungroup{{#1}, {#2}}}

\newcommand{\decktransf}[1]{\mathrm{Aut}({#1})}

\DeclareMathOperator{\imag}{im}

\newcommand{\cU}{\mathcal{U}}
\newcommand{\cV}{\mathcal{V}}
\newcommand{\cN}{\mathcal{N}}
\newcommand{\cE}{\mathcal{E}}
\newcommand{\cI}{\mathcal{I}}

\newcommand{\sN}{\mathsf{N}}
\newcommand{\nerve}[1]{\sN\rbracket{#1}}

\newcommand{\conv}[1]{\mathrm{conv}(#1)}

\newcommand{\homol}[2]{\fun{H_{#1}}{#2}}

\newcommand{\abel}[1]{{#1}^{\mathsf{ab}}}

\newcommand{\Egroup}[1]{\fun{e}{{#1}}}
\newcommand{\Egpd}[1]{\cE {{#1}}}
\newcommand{\Fgpd}[1]{\Pi {{#1}}}

\newcommand{\opstar}[1]{\mathstar\rbracket{#1}}

\newcommand{\pr}{\mathrm{pr}}

\newcommand{\grpd}[1]{\mathsf{#1}}
\newcommand{\conf}[2]{\fun{\mathrm{Conf}_{#1}}{#2}}
\newcommand{\uconf}[2]{\fun{\mathrm{UConf}_{#1}}{#2}}

\newcommand{\qtextq}[1]{\quad\text{#1}\quad}

\theoremstyle{plain}
  \newtheorem{theorem}{Theorem}[section]
  \newtheorem{corollary}[theorem]{Corollary}
  \newtheorem{lemma}[theorem]{Lemma}
  
  \newtheorem{proposition}[theorem]{Proposition}

\theoremstyle{definition}
  \newtheorem{definition}[theorem]{Definition}
  
  \newtheorem{ex}[theorem]{Example}
  
  \newtheorem{remark}[theorem]{Remark}
  \newenvironment{example}{\begin{ex}}{\end{ex}}

\title{Inferring Ambient Cycles of Point Samples on Manifolds with Universal Coverings}
\author{Ka Man Yim}
\date{January 2025}

\begin{document}

\maketitle
\begin{abstract}
    A central objective of topological data analysis is to identify topologically significant features in data represented as a finite point cloud. We consider the setting where the ambient space of the point sample is a compact Riemannian manifold. Given a simplicial complex constructed on the point set, we can relate the first homology of the complex with that of the ambient manifold by matching edges in the complex with minimising geodesics between points.  Provided the universal covering of the manifold is known, we give a constructive method for identifying whether a given edge loop (or representative first homology cycle) on the complex corresponds to a non-trivial loop (or first homology class) of the ambient manifold. We show that metric data on the point cloud and its fibre in the covering suffices for the construction, and formalise our approach in the framework of groupoids and monodromy of coverings. 
\end{abstract}
 \tableofcontents 
\section{Introduction}
A recurring theme in topological data analysis is to analyse the `shape' of a subset $A \subseteq X$ in relation to the geometry and topology of the ambient space. Out of the many variations on tackling this problem, we focus on analysing whether $X$ inherits the non-trivial loops and cycles of the ambient space. For example, we wish to infer the induced homomomorphism on the first homology group by the inclusion of $A$ into $M$. If a cycle in $A$ maps onto a non-trivial cycle in $X$, we call those the \emph{ambient cycles} of $A$ in $X$.

Here, we restrict ourselves to the setting where $M$ is a compact Riemannian manifold, and, the shape in question we want to analyse is based on a finite point cloud in $M$. For example, this shape can be the $\epsilon$-thickening of the point cloud. We consider this task in the setting where the manifold is equipped with its \emph{universal covering}. The universal covering is often accessible for model manifolds, such as tori, projective spaces, and compact surfaces with uniform curvature. We show how data inferred from the universal covering of the ambient manifold allows us to infer a covering of the simplicial complex, induced by an embedding of the complex. We show how such data can be integrated to infer the homomorphism on the fundamental group and first homology group induced by the embedding. We first set out the theoretical framework of this approach, which is based on a groupoids description of covering spaces. After setting out the abstract framework, we describe the practical pipeline for processing point clouds on such manifolds, where metric information of the manifolds are leveraged to make the inference.

\paragraph{Coverings and Groupoids}
The theory of covering spaces relates the fundamental group of a space $X$ with a symmetry on another space $\tilde{X}$. An exposition of the theory of coverings is given in~\Cref{sec:covering}. We summarise the key objects of interest in this introduction. In a class of coverings called \emph{regular coverings} (or $\Gamma$-coverings), the base space $X$ is the orbit space of a $\Gamma$-action on $\tilde{X}$, and the quotient $p: \tilde{X} \twoheadrightarrow X$ induces a short exact sequence of fundamental groups
\begin{equation*}
    \begin{tikzcd}
        0 & \fungroup{\tilde{X}, \tilde{x}} & \fungroup{X, x} & \Gamma & 0 
    \arrow["\fungroup{p}",hook, from=1-2, to=1-3]
	\arrow[from=1-4, to=1-5]
	\arrow["{\mu}", two heads, from=1-3, to=1-4]
	\arrow[from=1-1, to=1-2]
    \end{tikzcd}.
\end{equation*}
The homomorphism $\mu$, which we call the \emph{monodromy homomorphism}, relates loops in $X$ with the $\Gamma$-action on $\fibre{p}{x}$ via the homotopy lifting property of coverings. Monodromy classifies $\Gamma$-coverings up to isomorphism. The exactness of the sequence implies loops in $\pi_1(X)$ can be categorised into those that descend from the covering space $\tilde{X}$ (i.e. in the image of $\pi_1(p)$), and those that are generated due to the quotient by the group action (i.e. with non-trivial monodromy).  In the case where $\tilde{X}$ is simply connected, the covering is the \emph{universal covering} of the base space $X$. Since the sequence above is  exactness, the monodromy $\mu$ is an isomorphism between $\fungroup{X}$ and $\Gamma$. If we have a continuous map $f: X \to Y$, the universal covering of $q: \tilde{Y} \twoheadrightarrow Y$ induces a $\Gamma$-covering $p: \tilde{X} \twoheadrightarrow X$, where $G \cong \pi_1(Y)$. The monodromy $\mu_X: \fungroup{X}\to \Gamma$ of the induced covering on $X$ in fact recovers the induced homomorphism $\fungroup{f}: \fungroup{X} \to \fungroup{Y}$ up to isomorphism.

From the point of view of our application to simplicial complexes embedded in an ambient manifold, if we know the universal covering of the ambient manifold, then a recovery of the induced covering on the complex suffices to infer $\fungroup{f}$. In our setting, the induced covering of the complex is encoded as a \emph{transition homomorphism} (here viewing $p$ as a principal $\Gamma$-bundle). We thus focus on how monodromy of a covering can be inferred from its {transition homomorphism}. Given a good cover $\cU$ of $X$, a transition homomorphism prescribes how a $\Gamma$-covering $p: \tilde{X} \twoheadrightarrow X$ can be reconstructed by gluing together fibres $\fibre{p}{U}$ of $U \in \cU$; this consists of associating elements $g_{ij} \in \Gamma$ to pairs of cover elements $U_i$ and $U_j$ with non-empty intersection. In other words, it is an attachment of elements of $\Gamma$ to edges in the \emph{nerve} $\nerve{\cU}$ of the open cover. If the open cover is finite (which can be made so whenever $X$ is compact), then the transition homomorphism is a finite description of the $\Gamma$-covering. We employ the formalism of \emph{groupoids}\footnote{We recall a groupoid is a (small) category where all morphisms are invertible, and a groupoid homomorphism is simply a functor between groupoids. See~\Cref{sec:groupoid}.} to relate the transition homomorphism with monodromy. In more precise terms, the transition homomorphism is a groupoid homomorphism $t: \Egpd{\nerve{\cU}} \to \Gamma$ on the \emph{edge groupoid} $\Egpd{\nerve{\cU}}$ of the nerve. In the way which the nerve of a good cover $\nerve{\cU}$ is a discrete combinatorial topological representation of $X$, the edge groupoid is a discrete representation of the \emph{fundamental groupoid}\footnote{While the fundamental group describes the algebra of homotopy classes of loops based at a single point, the fundamental groupoid describes the algebra of homotopy classes of paths between any pair of points in $X$.} $\Fgpd{X}$ of the space. There is in fact an analogue of homotopy equivalence that relates the two groupoids. We show in~\Cref{ssec:transition} how monodromy and the transition homomorphism encode equivalent data. This equivalence is summarised by the following commutative diagram:
 \begin{equation*}
    \begin{tikzcd}
	\Egpd{\nerve{\cU}} && \Fgpd{X} \\
	& \Gamma
	\arrow["\simeq", curve={height=-6pt}, from=1-1, to=1-3]
	\arrow["t"', from=1-1, to=2-2]
	\arrow["S", curve={height=-6pt}, from=1-3, to=1-1]
	\arrow["{\mu}", from=1-3, to=2-2]
\end{tikzcd}.
\end{equation*}
Here the monodromy homomorphism $\mu$ here is promoted from a group homomorphism on the fundamental group to a groupoid homomorphism on the fundamental groupoid. This equivalence underpins the practical applications or our methods, as it allows us to summarise monodromy on compact spaces using finite data expressed in the transition homomorphism $t$ on the nerve of the covering $\cU$. Returning to the context of simplicial complexes embedded in an ambient space $f: |K| \to X$, we consider the transition homomorphism $t$ of the induced covering on $|K|$ by $f$; because the monodromy $\mu$ of the induced covering recovers the induced homomorphism of the fundamental groups $\fungroup{f}: \fungroup{|K|} \to \fungroup{X}$, we can deduce $\fungroup{f}$ by constructing the transition homomorphism $t$ of the induced covering. 

We also discuss the homology of groupoids. In particular, taking the first homology of the fundamental groupoid recovers the first singular homology of the underlying space. Likewise, the first homology of an edge groupoid of a simplicial complex is isomorphic to its first simplicial homology. Moreover, these relationships are natural isomorphisms. This allows us to translate groupoid homomorphisms between fundamental groupoids, edge groupoids, and groups, to homomorphisms between homology. This is especially important as we translate data from the transition homomorphisms $t$ of induced coverings to recover homomorphisms $\homol{1}{X} \to \homol{1}{Y}$. 

\paragraph{Applications to Topological Data Analysis on Manifolds}
We focus on the geometric setting where we have a continuous map $f: X \to M$, where $M$ is a complete Riemannian manifold, and $p$ is the universal Riemannian covering --- that is, $p$ is also a local isometry. 

Like most methods in standard topological data analysis, we assume the input data is a finite set of points $\X = \sett{x_1, \ldots, x_N}$ in $M$. We first consider $X = \X^\epsilon$, the $\epsilon$-thickening of $\X$, and $f$ being the inclusion of $\X^\epsilon$ into $M$. Here we require $\epsilon$ to be less than the geodesic convexity radius of $M$ to facilitate a good cover of $\X^\epsilon$. We show how we can infer the transition homomorphism $t$ of the induced covering on $X$, and thus infer the induced homomorphism of $f$ on the fundamental group or first homology. The required information to infer $t$ consists simply of the distances between points in $\X$, and the distances between points in their fibres $\fibre{p}{x} \subset \tilde{M}$. 

In a sparse sampling case where distances between points may exceed the geodesic convexity radius, we consider another approach to simulating the shape of the point cloud rather than considering its $\epsilon$-thickening. Here we consider a graph $G$ on the vertex set $\X$, and consider a map $f: |G| \to M$ that sends edges to minimising geodesics on $M$. For example, $G$ can be an $\epsilon$-neighbourhood graph or a $k$-nn graph on $\X$. We show that minimising geodesics can be uniquely assigned to pairs of points in $\X$ for generic configurations of $\X$. In this construction, the same metric data in the case for $\X^\epsilon$ suffices to infer the transition homomorphism $t: \Egpd{G} \to \Gamma$. Using $t$, we follow the same recipe to find the induced homomorphisms $\fungroup{f}$ and $\homol{1}{f}$.

Finally, as an application, we use this computational framework to empirically analyse a geometric signature of measures on manifolds called \emph{principal persistence measures}~\cite{Gomez2024CurvatureDiagrams}. The one-dimensional principal persistence measure is a geometric summary of the distribution of four point configurations sampled from the measure. In particular, it encodes the geometry of cycles that can be formed on four point configurations by minimising geodesic interpolations. We show how we can decompose the support of the principal persistence measure of an empirical sample on a manifold by the homology class of cycles to help interpret the geometric and topological features of the measure detected by its principal persistence measure. 

\paragraph{Related Work}
Within topological data analysis, the relationship between covering spaces, monodromy, and the homology of subcomplexes in such spaces was first explored in~\cite{Onus2022QuantifyingComplexes}. The focus of that study was on detecting loops and one-cycles of complexes in $\Z^n$ periodic spaces that are generated by the periodic symmetry. In~\cite{Onus2023ComputingWindows}  the authors extend their work to higher homology groups for complexes in $\Z$-periodic spaces. Here we only consider one-dimensional homology, but our approach generalises to complexes in universal covering spaces with group action $\Gamma$ not necessarily $\Z^n$.

\section{Groupoids}\label{sec:groupoid}
In this section, we discuss the algebraic and categorical framework that describes homotopy classes of paths on a topological space $X$ -- the \emph{fundamental groupoid} $\Fgpd{X}$. The fundamental groupoid generalises the fundamental group to describe homotopy classes of paths between any two fixed end points, in addition to loops.

By framing our problem in terms of groupoids, we reap some conceptual benefits. First, the groupoid formalism provides a natural framework to relate paths to simplicial approximations of paths. If a space $X$ admits a good cover $\cU$, then we can relate the fundamental groupoid to the \emph{edge groupoid} $\Egpd{\nerve{\cU}}$ of the nerve of the cover via a pair of groupoid homomorphism $R: \Egpd{\nerve{\cU}} \leftrightarrows \Fgpd{X} : S$. In a similar fashion to the fundamental groupoid, the edge groupoid describes equivalence classes of edge paths on a simplicial complex, related by a combinatorial notion of homotopy. In \Cref{prop:equivalence_groupoids} we describe how the pair of groupoid homomorphism $R$ and $S$ describe an equivalence (of categories) between the fundamental and edge groupoids. Regarding $\cE$ and $\Pi$ as functors from simplicial complexes to groupoids, we also show that $R$ is a natural isomorphism between those two functors in \Cref{prop:gr_functorial}.

In \Cref{ssec:homology_groupoid} we review basic facts about the homology of groupoids, and how it relates to the homology of spaces.  \emph{Groupoid homology} $\mathcal{H}_\ast: \Grpd  \to \AbGrp$ is a functor from the category of groupoids to abelian groups. In particular, the groupoid homology of a group is equivalent to its group homology. Furthermore, we recall the first group homology is the abelianisation of a group $\HGrpd{1}{\Gamma} \cong \abel{\Gamma}$. Focussing on the fundamental groupoid $\Fgpd{X}$, we consider  $\HGrpd{1}{\Fgpd{X}}$ , which we view as a functor $\TopS \xrightarrow{\Pi} \Grpd \xrightarrow{\mathcal{H}_1} \AbGrp$. There is a natural isomorphism $\mathscr{s}: H_1 \Rightarrow \mathcal{H} \circ \Pi$~\Cref{prop:Hfgpd_func}. As such, a groupoid homomorphism $\mu: \Fgpd{X} \to \Gamma$ induces a homomorphism on singular homology $\mathscr{s}\circ \HGrpd{1}{\mu}: \homol{1}{X} \to \abel{\Gamma}$. We also derive an analogous natural isomorphism between the first homology of edge groupoids and simplicial homology in \Cref{prop:hegrpd_func}.

The benefit of this framework emerges in the subsequent section about covering spaces. We defer detailed discussion on this to \Cref{sec:covering}, but give a brief account to motivate our study of groupoids. The fundamental groupoid provides a clean framework to describe concepts such as homotopy path-lifting, and coverings induced by continuous maps. In particular, as we focus on $\Gamma$-coverings (coverings generated by quotients of regular group actions), the homotopy lifting of paths are described by a groupoid homomorphism $\mu : \Fgpd{X} \to \Gamma$, which we call the \emph{monodromy homomorphism}. In particular, if the underlying space admits a good cover, we show that $\mu$ can be recovered from a discrete summary. Using the homomorphisms $R$ and $S$, we can pass between $\mu$ and a groupoid homomorphism $t: \Egpd{\nerve{\cU}} \to \Gamma$ on the edge groupoid of the nerve of the cover. 

We first lay out a few basic definitions and properties of groupoids. We refer to the reader to~\cite{Riehl2017CategoryContext,Brown2006TopologyGroupoids,Cohen1989CombinatorialApproach} for a detailed exposition of groupoids and category theory. A \emph{groupoid} $\grpd{C}$ is a small category where all morphisms are isomorphisms. For two objects in $a,b \in \grpd{C}$, we let $\grpd{C}(a,b)$ denote the set of morphisms from $a$ to $b$; morphisms $f \in \grpd{G}(a,b)$ and $g \in \grpd{G}(b,c)$, they compose to form a morphism $fg \in \grpd{C}(a,c)$; and there is an identity morphism $\iden_a \in \grpd{C}(a,a)$ for any $a$. The set of morphisms $\grpd{C}_a := \grpd{C}(a,a)$ form a group which we call the \emph{vertex group} of $\grpd{G}$ at $a$. 

\begin{ex}
    A groupoid with only one object is a group. 
\end{ex}

\begin{ex}\label{ex:simplicial_grpd}
    A {simplicial groupoid} $\Delta^n$ is a groupoid with $n+1$ objects; and between any two objects $i$ and $j$ there is exactly one morphism $\Delta^n(i,j) = \{(ij)\}$. This implies $(ij)(jk)= (ik)$ for all $i,j,k \in \{0, \ldots, n\}$.
\end{ex}

The category of groupoids $\Grpd$ is the category where the objects are groupoids and the morphisms are functors $F: \grpd{C} \to \grpd{D}$ between two groupoids as categories. We call morphisms in $\Grpd$ groupoid homomorphisms, A \emph{homotopy} of groupoid morphisms $F,G: X \to Y$ is a natural isomorphism $\theta: F \Rightarrow G$. Two groupoids are (homotopy) \emph{equivalent} if there is an equivalence of categories $F: \grpd{C} \rightleftarrows \grpd{D}: G$. In other words, there are natural isomorphisms $\iden_{\grpd{C}} \Rightarrow GF$ and $FG \Rightarrow \iden_{\grpd{D}}$.  We say $F: \grpd{C} \to \grpd{D}$ defines an \emph{equivalence of groupoids} if there exists a dual $G: \grpd{D} \to \grpd{C}$ such that $F$ and $G$ define an equivalence. As the nomenclature would suggest, an equivalence of categories is an equivalence relation~\cite[Lemma 1.5.5]{Riehl2017CategoryContext}.

\begin{ex}
    If the groupoid $\grpd{C}$ has only one object (i.e. it is a group), then functors $F$ and $G$ defining an equivalence of categories implies $GF$ is an inner automorphism $g \mapsto hg\inv{h}$ for some $h \in \grpd{C}(a,a)$. 
\end{ex}

Given a groupoid homomorphism $F: \grpd{C}\to \grpd{D}$, we can determine whether it defines an equivalence without explicitly constructing its dual $G$: a functor $F$ is an equivalence of categories iff $F$ full and faithful (the map $\grpd{C}(a,b) \to \grpd{D}(F(a), F(b))$ is a surjective and injective), and essentially surjective (every $d \in \grpd{D}$ is isomorphic to some $F(c)$) (\cite[Thm 1.5.9]{Riehl2017CategoryContext}). Since the condition of essentially surjective is satisfied by definition on connected groupoids, $F: \grpd{C} \to \grpd{D}$ is an equivalence of groupoids iff $F$ is full and faithful.

\subsection{The Fundamental Groupoid and the Edge Groupoid} \label{ssec:fungroupoid}
The fundamental groupoid generalises the fundamental group to a groupoid. Rather than considering homotopy classes of loops based on a single base point of a space $X$, the fundamental groupoid $\Fgpd{X}$ records the set of homotopy classes of paths between any pair of points in the space. For example, the vertex group $\Fgpd{X}(x,x)$ is the fundamental group $\fungroup{X,x}$. We refer the reader to~\cite[\S 6]{Brown2006TopologyGroupoids} for an introductory account to the topic.

\begin{definition} For a topological space $X$, the \emph{fundamental groupoid} $\Fgpd{X}$ is a groupoid where the objects are the points in $X$, and the set of morphisms $\Fgpd{X}(x,y)$ from $x$ to $y$ are the set of homotopy classes of paths from $x$ to $y$ in $X$. 
\end{definition}

Composing morphisms $[\gamma_1] \in \Fgpd{X}(x,y)$ and $[\gamma_2] \in \Fgpd{X}(y,z)$ of the fundamental groupoid corresponds to concatenating homotopy classes of paths $[\gamma_1] \ast [\gamma_2] = [\gamma_1 \ast \gamma_2] \in \Fgpd{X}(x,z)$. The fundamental group $\fungroup{X,x} = \Fgpd{X}(x,x)$ is the group of morphisms from $x$ to $x$ in the fundamental groupoid; that is, the set of homotopy classes of loops based at $x$.

Indeed, the fundamental groupoid can be regarded as a functor $\Pi: \TopS \to \Grpd$. If $f: X \to Y$ is a continuous map, then we have an induced homomorphism of the fundamental groupoids $\Fgpd{f}: \Fgpd{X} \to \Fgpd{Y}$: each homotopy class of paths $[\gamma] \in \Fgpd{X}(x_0,x_1)$  is sent to a homotopy class of paths $[f\gamma] \in \Fgpd{Y}(f(x_0),f(x_1))$. A homotopy $F: X \times I \to Y$ between maps $f = F(\cdot,0)$ and $g = F(\cdot, 1)$ induces a groupoid homotopy $\Pi_1(F): \Pi_1(f) \Rightarrow \Pi_1(g)$ between the induced groupoid morphisms. Naturally, a homotopy equivalence $f: X  \rightleftarrows  Y: g$ induces an equivalence between the fundamental  groupoids $\Fgpd{f}: \Fgpd{X} \rightleftarrows  Y: \Fgpd{g}$~\cite[6.5.10]{Brown2006TopologyGroupoids}.

 The combinatorial analogue of the fundamental groupoid is the \emph{edge groupoid} $\Egpd{K}$ of a simplicial complex $K$. We introduce the concept following the presentation of~\cite{Cohen1989CombinatorialApproach}. We first define the corresponding notions of paths and homotopy on $K$.
 \begin{definition}
     An edge path $\eta: I_m \to K$ on a simplicial complex $K$ is a simplicial map from the path graph $I_m$ with $m$ edges to $K$. An edge loop based at $u$ is an edge path where $\eta(0) = \eta(m) = u$.
 \end{definition}
In other words, an edge path is prescribed by a finite sequence of vertices $v_0 v_1 \cdots v_m$ where each successive pair of vertices $v_i v_{i+1}$ span a simplex in $K$ (i.e. either $v_i = v_{i+1}$, or $v_i v_{i+1}$ is an edge). We now describe a combinatorial analogue of homotopy between paths rel end points. Roughly speaking, two edge paths are homotopy equivalent if one can translate one to another over two-simplices of the complex. This is formally described in the following elementary operation between two edge paths. 

\begin{definition}
    An edge path $\eta: I_m \to K$ is an elementary contraction of a path  $\eta': I_{m+1} \to K$,  there is some $0 \leq i < m$, such that 
    \begin{itemize}
        \item $\eta(j) = \eta'(j)$ for $j \leq i$;
        \item $\eta(j) = \eta(j+1)$ for $j \geq i+1$; and
        \item $\eta'(i)\eta'(i+1)\eta'(i+2)$ spans a simplex in $K$.
    \end{itemize}
    Two edge paths $\eta$ and $\eta'$ are edge homotopic if there is a finite sequence of edge paths $\eta = \eta_0,\ldots, \eta_k = \eta'$ such that adjacent edge paths $\eta_i$ and $\eta_{i+1}$ in the sequence are related by an elementary contraction: either $\eta_i$ is an elementary contraction of $\eta_{i+1}$ or vice versa.
\end{definition}

Edge homotopy is an equivalence relation on the set of edge paths on $K$. Like homotopy classes of paths on a topological space, homotopy classes of edge paths concatenate to form homotopy classes of edge paths. The homotopy class of the constant path (i.e. a sequence consisting of one vertex) functions as the identity, as concatenating an edge path with the constant path can be reduced to the edge path via an elementary contraction. Reversing the order of an edge path is equivalent to taking the inverse of the edge path: if we concatenate an edge path with one where the vertices are sequenced in reverse, elementary contractions can reduce it to the constant path. 

Thus, taking vertices of $K$ as objects and homotopy classes of edge paths between vertices as the set of morphisms, we can check that this gives us a groupoid. 
\begin{definition} The \emph{edge groupoid} $\Egpd{K}$ of a simplicial complex $K$ is a groupoid whose objects are the set of vertices of $K$, and the set of morphisms $\Egpd{K}(v,w)$ from $v$ to $w$ is the set of homotopy classes of edge paths. 
\end{definition}

We call the vertex group $\Egroup{K,u}$ the \emph{edge group} at $u$. Like the fundamental groupoid for the category of topological spaces, the edge groupoid is a functor $\cE: \SCpx \to \Grpd$ from the category of simplicial complexes to groupoids. If $f: L \to K$ is a simplicial map, the we have an induced homomorphism of groupoids $\Egpd{f}: \Egpd{L} \to \Egpd{K}$, since simplicial maps satisfy \cref{eq:cocycle_egpd}. We refer the reader to~\cite[\S 5.2]{Cohen1989CombinatorialApproach} and~\cite[\S 6.4]{Armstrong2013BasicTopology} for details.

Because the set of morphisms of $\Egpd{K}$ are equivalence classes of edge paths, any morphism of $\Egpd{K}(u_0, u_m)$ can be expressed as a composition of the edge classes $[u_0 u_1] \cdots [u_{m-1} u_m]$. It follows that any homomorphism of groupoids $F: \Egpd{K} \to \grpd{C}$ can be fully expressed if we know $F$ on its edges, as
\begin{equation*}
     F([u_0 \cdots u_m]) = F([u_0 u_1]) \cdots F([u_{m-1} u_m]).
\end{equation*}
The following lemma states that any groupoid homomorphism $F$ from $\Egpd{K}$ can be defined by specifying its image on the edges, provided $F$ satisfies the \emph{cocycle condition} \cref{eq:cocycle_egpd} on two-simplices of $K$ so that it is well-defined. 

\begin{lemma}\label{lem:cocycle_egpd}
Let $K$ be a simplicial complex and $\iota: K_1 \hookrightarrow K$ be the inclusion of its one-skeleton into $K$. Suppose we have a groupoid homomorphism $F_1: \Egpd{K_1} \to \grpd{C}$. Then there is a groupoid homomorphism $F: \Egpd{K} \to \grpd{C}$ such that $F_1 = F \circ \Egpd{\iota}$, iff
    \begin{align}
    F_1(uv)F_1(vw)F_1(wu) = \iden_{\grpd{C}_{F_1(u)}}. \label{eq:cocycle_egpd}
    \end{align}
\begin{proof}
     Suppose we have $F: \Egpd{K} \to \grpd{C}$. For $uvw$ spanning a simplex in $K$, since $[uv]  [vw]  [wu] = [uvw] = [u]$ is edge-homotopic to the constant path by elementary collapses, $F$ being a functor implies $ F_1(uv)F_1(vw)F_1(wu) = F([uv])F([vw])F([wu])  = F([uv]  [vw]  [wu] ) = F([u]) = \iden $.

    Conversely, suppose \cref{eq:cocycle_egpd} are satisfied. We now verify that edge-homotopic paths from $x$ and $y$ are sent to the same morphism in  $\grpd{C}(F_1(x), F_1(y))$. Let the sequence of edge paths $a_0, a_1, \ldots a_n$ be an edge homotopy. In other words, for adjacent edge paths in the sequence, $a_i$ and $a_{i+1}$, either $a_i$ contracts onto $a_{i+1}$ or vice versa. In other words, we can write $a_i = b\cdot uvw \cdot c$ and $a_{i+1} = b\cdot uw \cdot c$, or vice versa the collapse goes the other way. Applying $F_1$ to the edge paths, and the identities in \cref{eq:cocycle_egpd}, we have $F_1(a_i) = F_1(b) F_1(uv) F_1(vw) F_1(c) = F_1(b) F_1(uw) F_1(c)= F_1(a_{i+1})$. Thus, we have a well-defined map between the set of morphisms  $\Egpd{K}(x,y) \to \grpd{C}(F_1(x),F_1(y))$. Defining $F(x) = F_1(x)$, and $F([uv]) = F_1(uv)$, we thus have a functor $F: \Egpd{K} \to \grpd{C}$ such that $F_1 = F \circ \Egpd{\iota}$.
\end{proof}
\end{lemma}

\subsection{Equivalence of the Fundamental and Edge Groupoids}
We now show that the geometric realisation of a simplicial complex $K$ induces a groupoid homomorphism $R : \Egpd{K} \to \Fgpd{K}$. The homomorphism $R$ formally relates homotopy classes of edge paths to homotopy classes of paths on the geometric realisation of a simplicial complex. We set up this result by considering a slightly more general category than that of simplicial complexes.

Recall a \emph{good cover} $\cU = \sett{U_i}_{i \in \cI}$ of a topological space is a collection of open sets $U_i$ where nonempty finite intersections of cover elements are contractible. We further recall the concept of the \emph{nerve} of an open cover. The nerve $\nerve{\cU}$  of an open cover $\cU$ is the abstract simplicial complex on a vertex set indexed by $\cI$, where 
\begin{equation}
    \nerve{\cU} = \settt{J \subseteq \cI }{\cap_{j \in J} U_j \neq \emptyset \quad \& \quad |J| < \infty}.
\end{equation}
Since we have restricted ourselves to good covers, the original nerve lemma of Leray~\cite{Leray1945SurRepresentations} states that the geometric realisation of the nerve of a good cover is homotopy equivalent $X \simeq |\nerve{\cU}|$ to the underlying space. We refer the reader to~\cite{Bjorner2003NervesGroups,Bauer2023AVariations} for contemporary results where conditions on the cover are relaxed.
\begin{ex}
   Finite simplicial complexes dimension admit good covers. If we take the cover to be the collection of open vertex stars $\cV = \settt{\opstar{u}}{u \in K_0}$, the nerve of the vertex stars is in fact isomorphic to $K$ itself: an intersection of vertex stars $\opstar{u_0},\ldots, \opstar{u_m}$  is non-empty precisely because their intersection contains the simplex $u_0, \ldots, u_m$ and any of its cofaces (i.e. its upper set in the face poset of $K$). 
\end{ex}
We now show that there is an equivalence of groupoids $\Fgpd{X} \simeq \Egpd{\nerve{\cU}}$ if $\cU$ is a good cover. This is directly implied by the nerve lemma which proves that there exists a homotopy equivalence between $X \simeq |\nerve{\cU}|$. Here we explicitly construct the groupoid homomorphisms $R: \Fgpd{X} \to \Egpd{\nerve{\cU}}$ and its adjoint $S: \Egpd{\nerve{\cU}} \to \Fgpd{X}$, such that $R: \Fgpd{X} \rightleftarrows \Egpd{\nerve{\cU}: S}$ is an equivalence of groupoids. We refer to $R$ as the \emph{realisation} homomorphism and $S$ as the \emph{snapping} homomorphism. These homomorphisms translate between homotopy classes of paths and edge paths. For paths in $X$, we can use the notion of \emph{carriers}~\cite{Bjorner2003NervesGroups} to discretise or `snap' them to paths on the complex.

\begin{restatable}{definition}{defpathcarrier}
    \label{def:pathcarrier} Given topological space $X$ with a good cover $\cU = \sett{U_i}_{i\in \cI}$, we say an edge path $\eta: I_m \to \nerve{\cU}$ \emph{carries} a path $\gamma: I \to X$, if the path can be parametrised such that for $k= 0,\ldots, m$
\begin{align}\textstyle
    \fun{\gamma}{\frac{k}{m}} &\in U_{\eta(k)} \\
    \textstyle \fun{\gamma}{\left(\frac{k}{m},\frac{k+1}{m}\right)} &\subset U_{\eta(k)} \cup U_{\eta({k+1})},\quad k<m.
\end{align}
\end{restatable}

The proof of the following statements, which is a generalisation of the proofs of similar statements relating the edge group to the fundamental group in \cite[Theorem 6.10]{Armstrong2013BasicTopology}, is given in \Cref{app:carrier}. Using the fact that any path in $X$ is carried by some edge path on the nerve of a good cover, we can define the snapping morphism in \Cref{prop:equivalence_groupoids}. We defer the proofs to~\Cref{app:carrier}.

\begin{restatable}[Path Carrier]{lemma}{pathcarrier}
\label{lem:pathcarrier}%
 Let $X$ be a space that admits an open cover $\cU = \sett{U_i}_{i \in \cI}$. If $x \in U_i$ and $y \in U_j$, any path $\gamma: I \to X$ from $x$ to $y$ is carried by some edge path on $\nerve{\cU}$ from $i$ to $j$.
\end{restatable}

\begin{restatable}{proposition}{equivalenceEFG}
    \label{prop:equivalence_groupoids} Let $X$ be a space with a good cover $\cU = \sett{U_i}_{i \in \cI}$. Consider a pair of maps $r:\cI \leftrightarrows X: s$, satisfying $ r(i) \in U_i$  and $x \in U_{s(x)}$. Then $r$ and $s$ \emph{uniquely} define a pair of groupoid homomorphisms $R: \Egpd{\nerve{\cU}} \leftrightarrows \Fgpd{X}: S$, where on objects $r(u) = R(u)$ and $S(x) = s(x)$; and on classes of paths:
\begin{itemize}
    \item $R([\eta]) \in \Fgpd{X}(r(i), r(j))$ is the unique class of paths represented by a path $\gamma \in R([\eta])$ carried by the edge path $\eta$; and 
    \item $S([\gamma]) \in \Egpd{\nerve{\cU}}(s(x), s(y))$ is the unique class of edge paths, one of which carries $\gamma \in [\gamma] \in \Fgpd{X}(x,y)$.
\end{itemize}
Furthermore, the pair $(R,S)$ is an equivalence of categories between $\Egpd{\nerve{\cU}}$ and $\Fgpd{X}$. Explicitly, we have natural isomorphisms $\varsigma: \iden_{\Fgpd{X}} \Rightarrow RS$, and $\varrho: SR \Rightarrow \iden_{\Egpd{\nerve{\cU}}}$. Furthermore,  $\varsigma_x \in \Fgpd{X}(x, rs(x))$ for any $x \in X$ is a unique homotopy class of paths carried by $U_{S(x)}$, and $\varrho_i \in \Egpd{\nerve{\cU}}(i, sr(i))$ is the edge homotopy class of the edge from $i$ to $sr(i)$ on the nerve.
\end{restatable}

By considering $(|K|, \sett{\opstar{u}}_{u \in K_0})$ as a space with good cover, \Cref{prop:equivalence_groupoids} specialises to the following corollary. Note that we have a stronger condition on the choice of representatives of cover elements: a vertex star is represented by the vertex itself. Since the nerve of the vertex star cover is the simplicial complex itself, the edge groupoid here is the edge groupoid of the complex.

\begin{corollary} \label{cor:equivalence_groupoids_scpx} Let $K$ be a simplicial complex. We have an equivalence of categories $R: \Egpd{K} \leftrightarrows \Fgpd{K}: S$, defined as follows:
\begin{itemize}
    \item On objects: $R$ sends vertices on the abstract simplicial complex $K$ to their corresponding points in its geometric realisation; and $S$ snaps points in $x$ to some vertex such that $x \in \opstar{S(x)}$. In particular, we can choose $SR(u) = u$;
    \item On classes of paths, $R([\eta]) = [|\eta|] \in \Fgpd{X}(r(i), r(j))$, and $S([\gamma]) \in \Egpd{K}(S(x), S(y))$ is the class of edge paths edge-homotopic to one whose vertex stars carry $\gamma$.
\end{itemize}
Explicitly, we have natural isomorphisms $\varsigma: \iden_{\Fgpd{K}} \Rightarrow RS$, and $SR = \iden_{\Egpd{K}}$. Furthermore, $\varsigma$ is unique, and carried by $\opstar{S(x)}$. 
\end{corollary}
Note that because each vertex is only contained in its own star and no others, the snapping map $s$ must send $u \in |K|$ to $u \in K$. Thus, if we choose $r$ to send $u \in K$ to $u \in |K|$, we must have $SR = \iden_{\Egpd{K}}$. 

We now show that the realisation functor $R: \Egpd{K}\to \Fgpd{K}$ is in fact a natural transformation of functors from the category of simplicial complexes to groupoids. Recall $\cE: \SCpx \to \Grpd$ is a functor from the category of simplicial complexes to groupoids. There is another functor, $\Pi: \SCpx \to \Grpd$, which is given by sending simplicial complexes to the fundamental groupoid of their geometric realisation. The following proposition shows that $R$ is a natural transformation between these two functors.  

\begin{proposition} \label{prop:gr_functorial} Consider the functors $\cE: \SCpx \to \Grpd$ and $\Pi: \SCpx \to \Grpd$ from the category of simplicial complexes to groupoids. For any simplicial map $f: K \to L$, the realisation homomorphisms of $K$ and $L$ commute with the groupoid homomorphisms induced by $f$:
\begin{equation} \label{dgm:gr_functorial}
    \begin{tikzcd}
        \Egpd{K} \arrow[r,"R_K"]\arrow[d,"\Egpd{f}"'] & \Fgpd{K}  \arrow[d,"\Fgpd{f}"] \\
        \Egpd{L} \arrow[r,"R_L"] & \Fgpd{L}
    \end{tikzcd}
\end{equation}
In other words,  the realisation homomorphism is a natural transformation $R: \cE \Rightarrow \Pi$. 
\end{proposition}
\begin{proof} Recall the continuous map $|f|: |K| \to |L|$ induced by a simplicial map $f$ satisfies $|f|(|\sigma|) = |f(\sigma)|$. Furthermore, recall simplicial maps sends edge homotopic paths to homotopy classes of paths, and continuous maps sends homotopy classes of paths to homotopy classes of paths.

Consider an edge $vw \in K$, and the class $[vw] \in \Egpd{K}(v,w)$. Following the top path of composition of maps in the square above,
\begin{equation*}
    [vw] \overset{R_K}{\longmapsto} [|vw|] \overset{\Fgpd{f}}{\longmapsto} [|f|(|vw|)] = [|f(v)f(w)|],
\end{equation*}
where the last equality is due to the property of the continuous maps induced by simplicial maps. Following the top path of composition of maps in the square above,
\begin{equation*}
    [vw]\overset{\Egpd{f}}{\longmapsto} [f(v)f(w)] \overset{R_L}{\longmapsto} [|f(v)f(w)|].
\end{equation*}
Thus on the level of homotopy classes of edge paths represented by edges, the square commutes. Since any class of edge path is composed by homotopy classes of edge paths represented by edges, and $R_\bullet$, $\Pi$, and $\cE$ respects composition, the commutativity of the square for edges extends to commutativity of the square for any class of edge paths in $\Egpd{K}$.
\end{proof}

While realisation is functorial, the snapping homomorphism $S$ is not natural and depends on which vertices are chosen to represent points in $|K|$. However, the lemma below shows that snapping admits pullbacks: if $f: K \to L$ is a simplicial map, for a given snapping $S_L: \Fgpd{L} \to \Egpd{L}$, we can construct snapping $S_K: \Fgpd{K} \to \Egpd{K}$, such that \cref{dgm:snapping_functorial} commutes.

\begin{lemma}\label{lem:compatible_snapping} Let $f: K \to L$ be a simplicial map. For any snapping homomorphism $S_L: 
\Fgpd{L} \to \Egpd{L}$, we can find a snapping homomorphism $S_K:\Fgpd{K} \to \Egpd{K}$ such that the following diagram commutes:
\begin{equation} \label{dgm:snapping_functorial}
    \begin{tikzcd}[cramped]
	\Egpd{K} & \Fgpd{K} \\
	\Egpd{L} & \Fgpd{L}
	\arrow["\Egpd{f}"', from=1-1, to=2-1]
	\arrow["{\exists S_K}"', from=1-2, to=1-1]
	\arrow["\Fgpd{f}", from=1-2, to=2-2]
	\arrow["{S_L}"', from=2-2, to=2-1]
\end{tikzcd}.\end{equation}
\end{lemma} 

\begin{proof} We first proceed on the level of objects. Let $s_L: |L| \to L_0$ be the snapping map that defines $S_L$. We show that we can find $s_K: |K| \to K_0$ such that for all $x \in |K|$,
    \begin{equation} \label{eq:s_a_quasi_functorial_object}
       s_L \circ f(x) = f \circ  s_K(x) .
    \end{equation}
Let $\tau$ be the unique simplex in $L$ containing $f(x)$ in its interior, and $s_L \circ f(x) = v$. Because $s_L$ defines a snapping, $\tau$ contains $v$ in its boundary. If we examine preimages in $|K|$, $\fibre{f}{\tau} \ni x$, and the unique simplex $\sigma$ in $K$ containing $x$ in its interior is in $\fibre{f}{\tau}$. Since every simplex in $\fibre{f}{\tau}$ must have some vertex $u \in \fibre{f}{v}$ in its boundary, $\sigma$ must have some $u \in \fibre{f}{v}$ in its boundary. In other words, there is some $u \in K_0$ such that $x \in \opstar{u}$, and $f(u) = s_L \circ f(x)$. We define $s_K(x)$ by picking one such vertex $u = s_K(x)$ to be the image of $x$. From \Cref{prop:equivalence_groupoids} $s_K$ uniquely defines a snapping homomorphism $S: \Fgpd{X} \to \Egpd{K}$.

Having established \cref{eq:s_a_quasi_functorial_object} which shows that \Cref{dgm:snapping_functorial} commutes on the level of objects, we now compare how $S_L \circ \Fgpd{f}$ and $\Egpd{f} \circ S_K$ sends homotopy classes of paths. Since $SR$ is the identity functor, and $R_L \circ \Egpd{f} = \Fgpd{f} \circ R_K$, we have 
\begin{equation*}
     \Egpd{f} \circ S_K= S_L R_L \circ  \Egpd{f} \circ S_K = S_L \circ  \Fgpd{f} \circ R_K S_K .
\end{equation*}
Thus \cref{dgm:snapping_functorial} commutes iff $S_L \circ \Fgpd{f} \circ  R_K S_K = S_L \circ  \Fgpd{f}$.

From \Cref{prop:equivalence_groupoids}, we have a unique natural transformation $\varsigma^K: \iden_{\Fgpd{K}} \Rightarrow R_KS_K$. Consider $[\gamma] \in \Fgpd{K}(x,y)$. Then 
\begin{equation}
    S_L \circ \Fgpd{f} \circ  R_K S_K([\gamma]) = \inv{(S_L \circ \Fgpd{f}(\varsigma_x^K))} \ast (S_L \circ \Fgpd{f}([\gamma])) \ast (S_L \circ \Fgpd{f}(\varsigma_y^K)).
\end{equation}
Consider $S_L \circ \Fgpd{f}(\varsigma_x^K)$. Recall $\varsigma_x^K$ is the unique class of paths from $x$ to $S_K(x)$ carried by $\opstar{S_K(x)}$. Since $f(\opstar{S_K(x)}) \subseteq \opstar{f \circ S_K(x)} = \opstar{S_L \circ f(x)}$, the class of paths $\varsigma_x^K$ is mapped by $\Fgpd{f}$ to a class of paths $\Fgpd{X}(f(x), R_L S_L(f(x)))$ carried by $\opstar{S_L \circ f(x)}$. In other words, $\Fgpd{f}(\varsigma_x^K) = \varsigma_{f(x)}^L$. Thus, $S_L \circ \Fgpd{f}(\varsigma_x^K) = \inv{(\varsigma_{f(x)}^L)} \varsigma_{f(x)}^L = e \in \fungroup{L, S_L \circ f(x)}$. Substituting into the equation above, we thus have $S_L \circ \Fgpd{f} \circ R_K S_K = S_L \circ \Fgpd{f}$.
\end{proof}

\begin{remark}
  We emphasise that snappings are not natural transformations. Since the snapping $S_K$ is constructed so that it is compatible with the snapping $S_L$ in the codomain, if we have multiple simplicial maps $f_i : K \to L_i$, there may not be a snapping $S_K$ that is compatible with all of $S_{L_i}$. The example in \Cref{fig:snapping_contradiction} illustrates how there is no choice of snappings $S_{L_i}$ that admit a compatible snapping $S_K$. 
\end{remark}
\begin{figure}
    \centering
    \includegraphics[width = 0.45\textwidth]{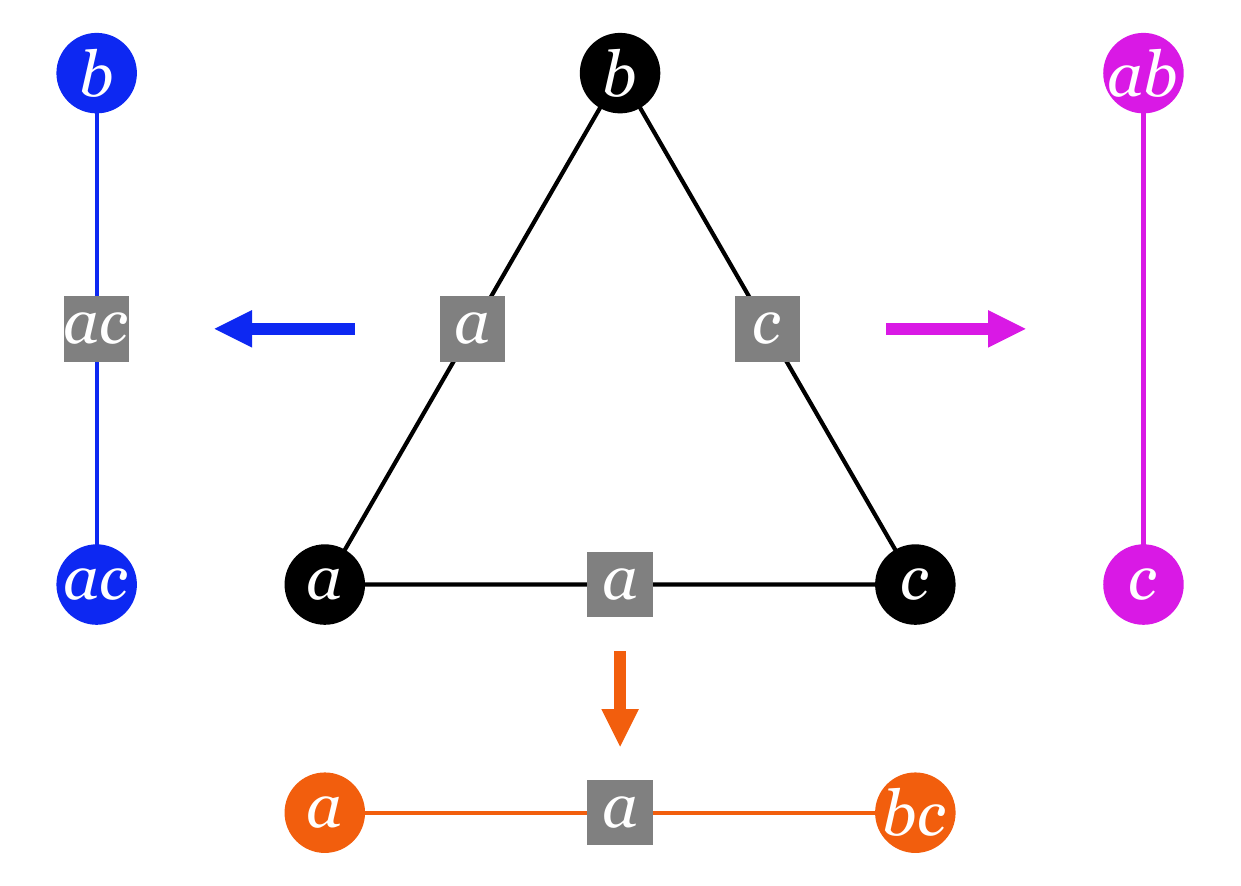}
    \caption{Consider the boundary of a two-simplex $K$ and the three possible surjective simplicial maps from $K$ to the closure of edge. The images of the maps on vertices is marked out in the diagram. For the (arbitrarily chosen) snapping on the left hand side edge, we illustrate the only snapping on the bottom edge, such that there is a snapping on $K$ that is simultaneously compatible with snappings on both the left and bottom edge. There is no snapping on the remaining edge that is compatible with this snapping. In other words, there is no choice of snappings on the three codomains, where there is a snapping on the pre-image such that  \cref{dgm:snapping_functorial} commutes for all three maps.}
    \label{fig:snapping_contradiction}
\end{figure}

\subsection{Homology} \label{ssec:homology_groupoid}
The concept of group homology extends to that of groupoids. We refer the reader to~\cite[Ch.6]{Weibel1994AnAlgebra} for a comprehensive account of group homology, and give an elementary account of the basics of groupoid homology, deferring details to~\cite[Ch.16]{Higgins1971CategoriesGroupoids}. 

We first define the notion of a \emph{simplex} of the groupoid $\grpd{C}$. Recall the simplicial groupoid $\Delta^n$ \Cref{ex:simplicial_grpd}. An $n$-simplex of $\grpd{C}$ is a groupoid homomorphism $\sigma: \Delta^n \to \grpd{C}$; we denote the set of $n$-simplices of $\grpd{C}$ as $\Sigma_n(G)$. In particular, $\Sigma_0(\grpd{C})$ is the set of objects of $\grpd{C}$; $\Sigma_1(\grpd{C})$ is the set of morphisms of $\grpd{C}$; and $\Sigma_2(\grpd{C})$ is the set of triplets of morphisms $(f,g,h)$ of $\grpd{C}$ satisfying $h = fg$. Denoting $\phi_k^n: \Delta^{n-1} \hookrightarrow \Delta^n$ to be the groupoid monomorphism including the $k$\textsuperscript{th} face of $\Delta^n$ into $\Delta^n$ (that is, the image is the simplex with the $k$\textsuperscript{th} object deleted), we induce a function $\psi_k^n: \Sigma_n \to \Sigma_{n-1}$ where $\sigma \mapsto \sigma \circ \phi_k^n$. A groupoid homomorphism $F: \grpd{C} \to \grpd{D}$ induces a function $\Sigma_k(\grpd{C}) \to \Sigma_k(\grpd{D})$ where $\sigma: \Delta^n \to \grpd{C}$ is sent to a simplex $F \circ \sigma: \Delta^n \to \grpd{D}$.

We fix an abelian group $A$ to be the homology coefficient group. We define the $k$-chains of $\grpd{C}$ with $A$-coefficients to be the direct sum over the abelian groups $A$ indexed over the $k$-simplices $\sigma \in \Sigma_k(\grpd{C})$:
\begin{equation}
    \ChGrpd{k}{\grpd{C};M} := \bigoplus_{\sigma \in \Sigma_k(\grpd{C})} A.
\end{equation}
The face relations induce the boundary maps $\partial_k: \ChGrpd{k}{\grpd{C};M} \to \ChGrpd{k-1}{\grpd{C};M}$ that is the linear extension of 
\begin{equation}
    \partial \sigma=\sum_{k=0}^n (-1)^k \psi_k^n \sigma.
\end{equation}
We can check via standard arguments that $\partial^2= 0$. We then define the $k$\textsuperscript{th} homology of the groupoid $\HGrpd{k}{\grpd{C};M}$ with $A$ coefficients as the chain complex of $(\ChGrpd{k}{\grpd{C};M}, \partial)$.

Passing from groupoids to homology is functorial. A groupoid homomorphism $F: \grpd{C} \to \grpd{D}$ induces morphisms of complexes $\ChGrpd{\bullet}{\grpd{C};A} \to\ChGrpd{\bullet}{\grpd{D};A}$, and homomorphism of homology groups $\HGrpd{\bullet}{F;A}: \HGrpd{\bullet}{\grpd{C};A} \to \HGrpd{\bullet}{\grpd{D};A}$. 
Furthermore, if $\theta: F \Rightarrow G$ is a homotopy of groupoid morphisms, then we induce identical homomorphisms $\HGrpd{\bullet}{F;A} = \HGrpd{\bullet}{G;A}$~\cite[Theorem 15]{Higgins1971CategoriesGroupoids}. As a consequence, an equivalence of groupoids $F: \grpd{C} \to \grpd{D}$ induces an isomorphism $\HGrpd{\bullet}{\grpd{C};A} \xrightarrow{\cong} \HGrpd{\bullet}{\grpd{D};A}$. 

We collect some further facts about the homology of groupoids. Here on we suppress the coefficient group $A$ in our notation. First, if $\grpd{C}$ has connected components $\grpd{C}^{(i)}$, then $\HGrpd{\bullet}{\grpd{C}} = \bigoplus_i \HGrpd{\bullet}{\grpd{C}^{(i)}}$; if $\grpd{C}$ is connected, and $\grpd{C}_v$ is a vertex group of $\grpd{C}$, then the equivalence of groupoids between a connected groupoid and any of its vertex groups implies $\HGrpd{\bullet}{\grpd{C}_v} \xrightarrow{\cong} \HGrpd{\bullet}{\grpd{C}}$. Since the groupoid homology of a group $G$ is its group homology, for a connected groupoid $\grpd{C}$ we thus have $\HGrpd{0}{\grpd{C};A} \cong A$ and $\HGrpd{1}{\grpd{C};A} \cong {\grpd{C}_v}/[\grpd{C}_v, \grpd{C}_v] \otimes_{\Z} A$ due to~\cite[Theorem 6.1.12]{Weibel1994AnAlgebra}. We reproduce an elementary proof in \Cref{app:grouphomology} for completeness.
\begin{restatable}{proposition}{grouphomology}
    \label{prop:grp_homology}
    If $G$ is a group, and $A$ is an abelian group, then
    \begin{align}
        \HGrpd{0}{G; A} & \cong A \\
        \eta: \HGrpd{1}{G;A} & \xrightarrow{\cong} \abel{G} \otimes A.
    \end{align}
    The isomorphism $\eta$ is evaluated on a representative element of a class of $\HGrpd{1}{G;A}$ as 
    \begin{equation}
        \textstyle \eta: \left[\sum_i \langle  g_i \rangle \otimes a_i \right] \mapsto \sum_i g_i \otimes a_i,
    \end{equation}
    where $\langle g \rangle$ to denote elements of $G$ as a basis for the free abelian group $\ChGrpd{1}{G} = \Z[G]$; the left hand side sum is a formal sum, while the right hand side sum denotes addition in the abelian group $\abel{G} \otimes A$.
\end{restatable}

In the context of this paper, we are mainly concerned with relating the homology of the fundamental groupoid of a space with the singular homology of the underlying space, we restrict our treatment here to homology up to degree one, as the fundamental groupoid retains only the one-dimensional topology of the underlying space. Due to the Hurewicz homomorphism~\cite[\S 2.A]{HatcherAlgebraicTopology} $\abel{\fungroup{X,x}} \cong H_1(X)$ for the singular homology of path-connected spaces, and $\abel{\Egroup{K,u}} \cong H_1(K)$~\cite[\S 5.1.3]{Stillwell2012ClassicalTheory} for the first simplicial homology group of a path-connected complex, we obtain the following result by considering the vertex groups of the fundamental and edge complexes: for $k=0,1$, 
    \begin{align}
        \HGrpd{k}{\Fgpd{X};A} &\cong H_k(X;A) \\
        \HGrpd{k}{\Egpd{K};A} &\cong H_k(K;A). 
    \end{align}
Because we have an equivalence of groupoids between the edge groupoid of a simplicial complex and the fundamental groupoid of its geometric realisation, this also implies the equivalence of singular and simplicial homology in the special cases of degree = 0,1. 

We show below that regarding $\mathscr{H}_\bullet 
\circ \Pi$ and $H_\bullet$ as functors from topological spaces to abelian groups, we have a natural isomorphism between them.

\begin{lemma}\label{lem:fun_grpd_chains} Let $X$ be a topological space. The function $\gamma \mapsto [\gamma]$ sending paths to their homotopy class rel end points, induces for $k= 0,1$ isomorphisms $\mathscr{s}_k: H_k(X;A) \xrightarrow{\cong} \HGrpd{k}{\Fgpd{X};A}$, given on representative cycles by
\begin{align}
    \mathscr{s}_0: \sum_i x_i \otimes a_i + B_0(X;A) &\mapsto \sum_i x_i \otimes a_i + \mathcal{B}_0({\Fgpd{X};A}); \\
    \mathscr{s}_1: \sum_i \sigma_i \otimes a_i + B_1(X;A) &\mapsto \sum_i [\sigma_i] \otimes a_i + \mathcal{B}_1({\Fgpd{X};A}).
\end{align}
\end{lemma}
\begin{proof} 
We derive the isomorphism $\mathscr{s}$ from a chain homomorphism 
\begin{equation}
    \begin{tikzcd}
	{B_1(X;A)} & {C_1(X;A)} & {C_0(X;A)} \\
	\mathcal{B}_0({\Fgpd{X};A})& {\ChGrpd{1}{\Fgpd{X};A}} & {\ChGrpd{0}{\Fgpd{X};A}}
	\arrow[hook, from=1-1, to=1-2]
	\arrow[two heads, from=1-1, to=2-1]
	\arrow["{\partial_1}", from=1-2, to=1-3]
	\arrow["q", two heads, from=1-2, to=2-2]
	\arrow["\cong", from=1-3, to=2-3]
	\arrow[hook, from=2-1, to=2-2]
	\arrow["{\partial_1}", from=2-2, to=2-3]
\end{tikzcd}.
\end{equation}
We first consider the case where $A = \Z$. As a short hand, we let $B_1 = B_1(X)$, $C_1 = C_1(X)$, and $\mathcal{B}_1 = \ChGrpd{1}{\Fgpd{X}}$, and $\mathcal{B}_1 = \mathcal{B}_1({\Fgpd{X}})$. 
    
    We first show that $C_1/B_1 \cong \mathcal{B}_1/\mathcal{B}_1$
    Let $q: C_1 \twoheadrightarrow \mathcal{B}_1$ be the unique homomorphism between the free abelian groups induced by the function $\gamma \mapsto [\gamma]$. Explicitly, 
    \[
    q: \sum_i n_i \sigma_i \mapsto \sum_i n_i [\sigma_i]
    \]
    where $\sigma_i$ are singular one-simplices. We show that $\ker q \subset B$. Let us write any element of $C_1$ as $c = (-1)^{k_1}\sigma_1 + \cdots + (-1)^{k_m} \sigma_m$, where $\sigma_i$ are allowed to repeat. If $q(c) = 0$, then terms in the formal sum can be paired and indexed such that $\sigma_{2k} \simeq \sigma_{2k+1}$ and they have oppositely signed coefficients. Let us then write $c$ as $c = (-1)^{k_1}(\sigma_1 - \sigma_2) + \cdots + (-1)^{k_{m-1}}(\sigma_{m-1} - \sigma_{m})$. Since $\sigma_{2k} \simeq \sigma_{2k+1}$, each $(\sigma_{2j} - \sigma_{2j+1})$ term in the sum is a boundary of a singular two-simplex (see~\cite{HatcherAlgebraicTopology}, theorem 2A.1, points \emph{(i-ii)} of proof). Thus, $\ker q \subset B$.

    We first consider the following homomorphisms between such groups
    \[\begin{tikzcd}[cramped]
	B_1 & C_1 & {C_1/B_1} \\
	{\mathcal{B}_1} & {\mathcal{B}_1} & {\mathcal{B}_1/\mathcal{B}_1}
	\arrow[hook, from=1-1, to=1-2]
	\arrow["{q\rvert_B}"', dashed, from=1-1, to=2-1]
	\arrow[two heads, from=1-2, to=1-3]
	\arrow["q"', two heads, from=1-2, to=2-2]
	\arrow["q^\ast", dashed, from=1-3, to=2-3]
	\arrow[hook, from=2-1, to=2-2]
	\arrow[two heads, from=2-2, to=2-3]
\end{tikzcd},\]
    where the homomorphism $q$ sends simplicial chains $\sum_i c_i \sigma_i$ to the formal sum of homotopy classes of paths $\sum_i c_i [\sigma_i]$. We proceed to verify that the homomorphisms $q\rvert_B$ and $q^\ast$ -- the induced homomorphisms on quotients -- exist, and commute with the other homomorphisms in the diagram. 

    Let us first consider $q\rvert_B$ and check that $q(B_1) \subseteq \mathcal{B}_1$. Consider the boundary of singular two-simplex $a + b - c$. Since $c \simeq a \cdot b$, we have $q(a+ b - c)= [a] + [b] - [a\cdot b] = [a] + [b] - [a] \cdot [b] \in \mathcal{B}$. Extending linearly to boundary chains shows that $q(B_1) \subseteq \mathcal{B}_1$. As such $q$ induces a homomorphism $q^\ast: C_1/B_1 \to \mathscr{C_1/B_1}$. Moreover, $q\rvert_B$ is a surjection onto $\mathcal{B}_1$, since every $[a]+[b]-[a \cdot b]$ is the image of a boundary of a two-simplex (see~\cite{HatcherAlgebraicTopology}, proof of theorem 2A.1, item \emph{(iii)}). As such, $q(B_1) = \mathcal{B}_1$, and the diagram above commutes. This implies $q^\ast$ is injective. Because $q$ is surjective, we thus conclude that $q^\ast$ is an isomorphism.

    If we consider the zero chain groups, we have an isomorphism $C_0(X) = \ChGrpd{0}{\Fgpd{X}}$, as they are free abelian groups over the same basis. From the definition of the first boundary homomorphism of either chains, which sends (the class of) path $\gamma$ to $\gamma(1) - \gamma(0)$, the following commutes and thus we have a chain homomorphism 
    \begin{equation*}
        \begin{tikzcd}
	{B_1} & {C_1} & {C_0} \\
	{\mathcal{B}_1} & {\mathcal{B}_1} & {\mathcal{B}_0}
	\arrow[hook, from=1-1, to=1-2]
	\arrow[two heads, from=1-1, to=2-1]
	\arrow["{\partial_1}", from=1-2, to=1-3]
	\arrow["q", two heads, from=1-2, to=2-2]
	\arrow["\cong", from=1-3, to=2-3]
	\arrow[hook, from=2-1, to=2-2]
	\arrow["{\partial_1}", from=2-2, to=2-3]
\end{tikzcd}.
    \end{equation*}
    Because the right hand side square commutes and $q$ is surjective, we have an isomorphism $\HGrpd{0}{X} \cong H_0(X)$. Furthermore, since $q^\ast: C_1/B_1 \to \mathscr{C_1/B_1}$ is an isomorphism, the diagram above implies an isomorphism $\mathscr{s}: H_1(X) \xrightarrow{\cong} \HGrpd{1}{\Fgpd{X}}$ between the kernels of $\partial_1$ restricted to the quotient groups. 

    Having taken care of the case $A = \Z$, we consider coefficients in arbitrary abelian groups. Taking tensor products of the diagram above with $\otimes A$ yield the corresponding homomorphisms for the chains in $A$-coefficients. Because the universal coefficient theorem in degree $ k \leq 1$ implies there are natural homomorphisms $H_k(C_\bullet) \otimes A \xrightarrow{\cong} H_k(C_\bullet; A)$ for $C_\bullet$ a chain complex of free abelian groups, the induced homomorphisms $\mathscr{s}_k: H_k(X;A) \to \HGrpd{k}{\Fgpd{X};A}$ is an isomorphism for all abelian groups $A$. 
\end{proof}

\begin{proposition}\label{prop:Hfgpd_func}
    For $k =0,1$, there is a natural isomorphism between the singular homology functor $H_k: \TopS \to \AbGrp$ and $\mathscr{H}_k \circ \Pi: \TopS \xrightarrow{\Pi} \Grpd \xrightarrow{\mathscr{H}_k} \AbGrp$. Concretely, for a continuous map $f: X \to Y$ between topological spaces, the isomorphisms $\mathscr{s}_k$ given in \Cref{lem:fun_grpd_chains} commute with the induced homomorphisms below:
    \begin{equation}
        \begin{tikzcd}
	{H_k(X;A)} & {H_k(Y;A)} \\
	{\HGrpd{k}{\Fgpd{X};A}} & {\HGrpd{k}{\Fgpd{Y};A}}
	\arrow["{H_k(f)}", from=1-1, to=1-2]
	\arrow["\cong", "{\mathscr{s}^X_k}"', from=1-1, to=2-1]
	\arrow["{\mathscr{s}^Y_k}", "\cong"', from=1-2, to=2-2]
	\arrow["\HGrpd{k}{\Fgpd{f}}", from=2-1, to=2-2]
\end{tikzcd}.
    \end{equation}
\end{proposition}

\begin{proof}
    We show this from direct evaluation. For $k = 1$, and a representative cycle $h = \sum_i \sigma_i \otimes a_i + B_1(X;A)$, 
    \begin{align*}
       \mathscr{s}_1^Y \circ H_1(f) (h) &= \mathscr{s}_1^Y (\sum_i (f \circ \sigma_i) \otimes a_i + B_1(Y;A)) = \sum_i [f \circ \sigma_i] \otimes a_i + B_1(\Fgpd{Y};A) \\
       \HGrpd{1}{\Fgpd{f}} \circ \mathscr{s}_1^X (h) &= \HGrpd{1}{\Fgpd{f}}(\sum_i [\sigma_i] \otimes a_i + B_1(
       \Fgpd{X};A)) = \sum_i [f \circ \sigma_i] \otimes a_i + B_1(\Fgpd{Y};A) \\
       &=  \mathscr{s}_1^Y \circ H_1(f) (h).
    \end{align*}
    The case for $k = 0$ follows along similar lines upon direct evaluation of such homomorphisms. 
\end{proof}

We show similar results for the edge groupoid.

\begin{lemma}\label{lem:egrpd_chains} Let $K$ be a simplicial complex. The function $uv \mapsto [uv]$ sending edges to their homotopy class of edge paths rel end points, , induces for $k= 0,1$ isomorphisms from simplicial homology groups $\mathscr{s}_k: H_k(K;A) \xrightarrow{\cong} \HGrpd{k}{\Egpd{K};A}$; they are given on representative cycles by
\begin{align}
    \mathscr{s}_0: \sum_i v_i \otimes a_i + B_0(K;A) &\mapsto \sum_i v_i \otimes a_i + \mathcal{B}_0({\Egpd{K};A}); \\
    \mathscr{s}_1: \sum_i \sigma_i \otimes a_i + B_1(K;A) &\mapsto \sum_i [\sigma_i] \otimes a_i + \mathcal{B}_1(\Egpd{K};A).
\end{align}
\end{lemma} 

\begin{proof} 
In a similar vein to \Cref{lem:fun_grpd_chains}, we derive the isomorphisms $\mathscr{s}$ from a chain homomorphism. Note however that there are subtle differences in the chain homomorphisms: the singular chains surject onto $\ChGrpd{\bullet}{\Fgpd{X}}$, while the simplicial chains inject into $\ChGrpd{\bullet}{\Egpd{K}}$.

\begin{equation}\label{eq:dgm_egrpd_chain}
    \begin{tikzcd}
	{B_1(K;A)} & {C_1(K;A)} & {C_0(K;A)} \\
	{\mathcal{B}_1(\Egpd{K};A)} & {\ChGrpd{1}{\Egpd{K};A}} & {\ChGrpd{0}{\Egpd{K};A}}
	\arrow[hook, from=1-1, to=1-2]
	\arrow[hook, from=1-1, to=2-1]
	\arrow["{\partial_1}", from=1-2, to=1-3]
	\arrow["q", hook, from=1-2, to=2-2]
	\arrow["\cong", from=1-3, to=2-3]
	\arrow[hook, from=2-1, to=2-2]
	\arrow["{\partial_1}", from=2-2, to=2-3]
\end{tikzcd}.
\end{equation}
We first consider the case where $A = \Z$. As a short hand, we let $B = B_1(K)$, $C = C_1(K)$, and $\mathcal{C} = \mathcal{C}(\Egpd{X})$, and $\mathcal{B} = \mathcal{B}(\Egpd{K})$. Let $q: C \to \mathcal{C}$ be the unique homomorphism between the free abelian groups induced by the function $uv \mapsto [uv]$ sending edges to their homotopy class of paths. We first show that $q(B) \subseteq \mathcal{B}$, and thus we induce a homomorphism $q^\ast: C/B \to \mathcal{C}/\mathcal{B}$. Each element in $B$ is a linear combination over elements $uv + vw - uw$ where $uvw$ is a two-simplex in $K$. Because 
\[q(uv + vw - uw) = \langle[uv] \rangle + \langle[vw] \rangle - \langle[uw] \rangle  = \langle[uv] \rangle + \langle[vw] \rangle - \langle[uv]\cdot [vw] \rangle \in \mathcal{B},\]
we thus have $q(B) \subseteq \mathcal{B}$. In addition, $q^\ast$ is surjective: because $\langle [f]\cdot [g] \rangle \sim \langle [f] \rangle  + \langle [g] \rangle $ in $\mathcal{C}/\mathcal{B}$, and every homotopy class of edge path can be expressed as a composition of classes of edges, any element of $\mathcal{C}/\mathcal{B}$ is represented by a formal sum over edge classes $\langle [uv] \rangle$. Suppose the orientation of the edge class $\langle [vu] \rangle $ is opposite to the fixed orientation of the edge $uv$ as a basis element of $C$. Then we have $\langle [vu] \rangle  + \mathcal{B} = -\langle [uv] \rangle  + \mathcal{B}$, because 
\[\langle [uv] \rangle  + \langle [vu] \rangle  - \langle [uvu] \rangle  \in \mathcal{B} \qtextq{and} \langle [uvu] \rangle = \langle [u] \rangle = \langle [u] \rangle  + \langle [u] \rangle   -\langle [u] \cdot [u]  \rangle  \in \mathcal{B}. \]
We can thus express any element of $\mathcal{C}/\mathcal{B}$ as the image of some formal sum over a basis of oriented edges. Therefore $q^\ast$ is surjective. 

Finally, to show that $q^\ast$ is an isomorphism, we prove that it is also injective. In other words, if a one-chain $c \in C$ is mapped to $q(c) \in \mathcal{B}$, then $c \in B$. We construct a homomorphism $p: \mathcal{B} \to C$ between the free abelian groups by specifying the function between basis elements: for each basis element $[f] \in \mathcal{B}$ corresponding to an edge path, we choose a fixed representative edge path $v_0 \cdots v_m \in [f]$, and send it to $p(\langle [f] \rangle) = v_0v_1 + \cdots v_{m-1}v_m$. Consider then the elements in $\mathcal{B}$.  We have $p(\mathcal{B}) \subseteq B$, because each element $\langle[f] \rangle + \langle [g] \rangle - \langle [f \cdot g] \rangle$ in the spanning set of 
$\mathcal{B}$ is sent to a chain $a_0 a_1  + \cdots +a_{n-1}a_n+ a_{n}a_0$ such that $a_0 a_1 \cdots a_na_0$ is a null-homotopic loop. We recall if an edge loop $a_0 a_1 \cdots a_m$ is null-homotopic, then the chain $a_0a_1 + \cdots + a_{m-1}a_m + a_m a_0$ is in $B$ (see~\cite[p.182]{Armstrong2013BasicTopology}). Hence, $a_0 a_1  + \cdots +a_{n-1}a_n+ a_{n}a_0$ is in $B$ and $p(\mathcal{B}) \subseteq B$.

We now show that $q(c) \in \mathcal{B}$ implies $c \in B$, by showing that $p \circ q$ is the identity on $B$. The homomorphism $p$ depends on the choice of representative of edge paths. For a class of edge paths homotopy equivalent to an edge path $uv$, we always choose the edge $uv$ to be the representative path. Thus, for $c = \sum_i n_i \sigma_i \in \fibre{q}{\mathcal{B}}$,
\[\textstyle p \circ q(c) = \fun{p}{\sum_i n_i \langle [\sigma_i] \rangle} = \sum_i n_i \sigma_i = c \in \imag p \subseteq B.  \]
Hence $q(c) \in \mathcal{B}$ implies $c \in B$. Thus $q^\ast$ is injective and therefore an isomorphism. 

If we consider the zero-chain groups, we have an isomorphism $C_0(K) \cong \ChGrpd{0}{\Egpd{K}}$, as they are both free abelian groups. The right hand side square of \cref{eq:dgm_egrpd_chain} commutes: $\partial_1 \circ q (uv)=  \partial_1 [uv] = u-v$. Because $q^\ast$ is also an isomorphism, the image and kernels of the respective boundary homomorphisms on the top and bottom chains are isomorphic, and thus $q$ induces the isomorphism $\mathscr{s}_k: H_k(K) \xrightarrow{\cong} \HGrpd{k}{\Egpd{K}}$ for $k = 0,1$. 

In the case for arbitrary coefficients, the same homological algebra arguments used in the proof of \Cref{lem:fun_grpd_chains} apply to arrive at the same statements. 
\end{proof}
\begin{proposition}\label{prop:hegrpd_func}
    For $k =0,1$, there is a natural isomorphism between the simplicial homology functor $H_k: \SCpx \to \AbGrp$ and $\mathscr{H}_k \circ \mathcal{E}: \SCpx \xrightarrow{\mathcal{E}} \Grpd \xrightarrow{\mathscr{H}_k} \AbGrp$. Concretely, for a simplicial map $f: K \to L$, the isomorphisms $\mathscr{s}_k$ given in \Cref{lem:egrpd_chains} commute with the induced homomorphisms below:
    \begin{equation}
        \begin{tikzcd}
	{H_k(K;A)} & {H_k(L;A)} \\
	{\HGrpd{k}{\Egpd{K};A}} & {\HGrpd{k}{\Egpd{L};A}}
	\arrow["{H_k(f)}", from=1-1, to=1-2]
	\arrow["\cong", "{\mathscr{s}^X_k}"', from=1-1, to=2-1]
	\arrow["{\mathscr{s}^Y_k}", "\cong"', from=1-2, to=2-2]
	\arrow["\HGrpd{k}{\Egpd{f}}", from=2-1, to=2-2]
\end{tikzcd}.
    \end{equation}
\end{proposition}
\begin{proof}
    This can be verified via explicit computation, as in \Cref{prop:Hfgpd_func}.
\end{proof}

While passing the homology of the fundamental or edge groupoid does not encode more information about the underlying space than that of singular or simplicial homology, we make the equivalence plain as we deal with how groupoid homomorphisms induce homomorphisms on the level of homology. In particular, when we discuss $\Gamma$-coverings (where $\Gamma$ is a discrete group) of a space $X$ in the subsequent section, the monodromy of the covering can be encoded as a groupoid homomorphism $\mu: \Fgpd{X} \to \Gamma$. If the base space $X$ of the covering admits a good cover $\cU$, then the atlas of the bundle can be equivalently encoded as a groupoid homomorphism $t: \Egpd{\nerve{\cU}} \to \Gamma$. Applying the functor $\mathscr{H}_\bullet: \Grpd \to \AbGrp$, we show that the monodromy and atlas homomorphism $\mu$ and $t$ respectively induce homomorphisms on singular homology $\mu^\ast: H_1(X;A) \to \abel{\Gamma} \otimes A$ and simplicial homology $t^\ast: H_1(K;A) \to \abel{\Gamma} \otimes A$ 
\begin{align*}
    \mu^\ast : \sum_i \sigma_i \otimes a_i + B_1(X;A) \mapsto \sum_i \mu([\sigma_i]) \otimes a_i \\
    t^\ast : \sum_i \sigma_i \otimes a_i + B_1(K;A) \mapsto \sum_i t([\sigma_i]) \otimes a_i.
\end{align*}

\section{Covering Spaces}\label{sec:covering}

We refer the reader to~\cite{HatcherAlgebraicTopology} for the details of the facts about covering spaces summarised below. Recall a \emph{covering} of a topological space $X$ is a continuous surjection $p: \tilde{X}  \twoheadrightarrow X$, such that for each neighbourhood of $x \in X$, there is an open neighbourhood $U$ where
\begin{equation}
    \fibre{p}{U} = \bigsqcup_{\tilde{x} \in \fibre{p}{x}} U_{\tilde{x}},
\end{equation}
and the restriction of $p$ to each $U_{\tilde{x}}$ is a homeomorphism onto $U$. 

We focus on the following special case of coverings $p: \tilde{X} \twoheadrightarrow X$.

\begin{definition}\label{def:gamma_covering} A continuous surjection $p: \tilde{X} \twoheadrightarrow X$ is a $\Gamma$-covering, if the following conditions are satisfied:
\begin{enumerate}[label = \textbf{(C\arabic*)}, ref =  \textbf{(C\arabic*)}]
    \item \label{C1} $X$ is locally path connected and semi-locally simply connected topological spaces;
    \item \label{C3} $p$ is the quotient of $\tilde{X}$ by a 
\emph{covering space action} of a discrete group $\Gamma$ on $\tilde{X}$.
\end{enumerate}
\end{definition}

We recall a free group action $\Gamma \times Y \to Y$ is a {covering space action}~\cite{HatcherAlgebraicTopology} \footnote{a.k.a. a properly discontinuous action.}, if for any point $y \in Y$, there is a neighbourhood $V$ such that 
\begin{equation}
    (g \cdot V) \cap V \neq \emptyset \quad \implies \quad  g  = e.
\end{equation}
Since the covering map is the quotient of such an action, $U = p(V)$ is evenly covered, and $\fibre{p}{U}$ is simply the orbit $\Gamma \cdot V$. Thus for $U$ evenly covered, we have a \emph{local trivialisation}
\begin{equation} \label{dgm:local_triv}
\begin{tikzcd}
	{\fibre{p}{U}} & {U \times \Gamma} \\
	{U}
	\arrow["p", from=1-1, to=2-1]
	\arrow["{\varphi_U}", from=1-1, to=1-2]
	\arrow["{\pr_1}", from=1-2, to=2-1]
\end{tikzcd},
\end{equation}
where $\varphi_U$ is a $\Gamma$-equivariant map, which we call a \emph{chart}; if $\varphi_U(z) = (x,g)$, then
\begin{equation}
\varphi_U(h \cdot z) = h \cdot \varphi_U(z) = (x,hg).
\end{equation}
Because $\varphi_U$ is equivariant,  $\varphi_U$ is uniquely defined by where $\inv{\varphi_U}$ sends $U \times \sett{e}$. If $\varphi_U'$ is another local trivialisation, then equivariance implies there is some $\theta \in \Gamma$ such that  $\varphi_U' \circ \inv{\varphi_U}(x,g) = (x,g \cdot  \theta)$ for all $x \in U$. Note that $\theta$ does not depend on $x$ as $\Gamma$ is discrete.

We now consider the group of automorphisms of coverings. If $p,q: \tilde{X}_1 \to X$ are coverings of $X$, then we say a homeomorphism $f: \tilde{X} \to \tilde{X}$ is an isomorphism of coverings $p$ and $q$ if $q = \tilde{f} p$. These homeomorphisms form the group of \emph{deck transformations} $\decktransf{p}$ of the covering $p$. For $p$ a $\Gamma$-covering, it follows from the definition that $\Gamma$ is a subgroup of deck transformations, and furthermore, if the covering is path-connected, then $\Gamma$ is the group of deck transformations~\cite[Prop 1.40]{HatcherAlgebraicTopology}. 

An important class of $\Gamma$-covering spaces are \emph{universal coverings} $p: \tilde{X} \to X$, those where the covering space $\tilde{X}$ is simply connected. A path-connected space $X$ admits a universal covering iff \labelcref{C1} is satisfied. In that case, we have an isomorphism $\fungroupbased{X}{x} \cong \decktransf{p} \cong \Gamma$ mediated by homotopy path lifting, which we discuss below. We give some classical examples of universal coverings of manifolds. 
\begin{example}\label{ex:surface_coverings} We consider four simple examples of $\Gamma$-coverings that are universal. 
\begin{enumerate}[label = (\roman*)] 
    \item The universal covering $p: \R^2 \to \mathbb{T}$ of the flat torus by $\R^2$ is a $\Gamma = \Z \oplus \Z$ covering; the group action is given by 
    \begin{equation*}
        (n,m) \cdot  (x,y) \mapsto (x + n, y+ m).
    \end{equation*}
    \item The universal covering $p: \R^2 \to \mathbb{K}$ of the flat Klein bottle by $\R^2$ is a $\Gamma = \Z \rtimes \Z$ covering; the group action is given by 
    \begin{equation*}
        (n,m) \cdot  (x,y) \mapsto ((-1)^mx + n, y+ m).
    \end{equation*}
    \item The universal covering $p: \Sbb^n \to \RP^n$ of the real projective space by $\Sbb^2$ is a $\Gamma = \Z/2\Z$ covering; the group action is given by 
    \begin{equation*}
        a \cdot x \mapsto (-1)^ax
    \end{equation*}
    where a point $x \in \Sbb^n$ here is regarded as a unit vector in $\R^{n+1}$.
    \item The universal covering $p: \mathbb{D} \to \Sigma_2$ of the compact surface of genus 2 by the hyperbolic plane is a $\Gamma =  \langle g_0, g_1, g_2, g_3 \mid g_{0}g_{1}^{-1}g_{2}g_{3}^{-1}g_{0}^{-1}g_{1}g_{2}^{-1}g_{3}\rangle$ covering; using the Poincar\'e disk model $\mathbb{D} = \settt{ z \in \mathbb{C} }{|z| < 1}$ as the model for the hyperbolic plane,  the group action on $\mathbb{D}$ is a M\"obius transformation on the complex plane, where the individual action of a generator $g_k$ is given in matrix form by
    \begin{equation*}
        g_k \cdot z = \frac{z +\omega^k\sqrt{1-c^2} }{\omega^{-k}\sqrt{1-c^2 } z  + 1},
    \end{equation*}
    where $c = \tan(\pi/8)$, and $\omega = e^{i\pi/4}$.
\end{enumerate}
\end{example}

\subsection{Monodromy}
Since covering spaces are fibre bundles and thus have the homotopy lifting property with respect to discs~\cite[Prop 4.48]{HatcherAlgebraicTopology}, the homotopy groups of $\tilde{X}$ and $X$ are related via that of the fibre of the projection by a long exact sequence~\cite[Theorem 4.41]{HatcherAlgebraicTopology} if $X$ is path-connected. In particular, since the fibres of covering spaces are discrete, the long exact sequence breaks apart into the following exact sequences: for homotopy groups $\pi_n$ with $n \geq 2$, the covering map induces isomorphisms $\pi_n (p) : \pi_n(\tilde{X},\tilde{x}) \xrightarrow{\cong}  \pi_n(X,x)$; and for $n \leq 1$, 
\begin{equation} \label{eq:fibration_LES}
\begin{tikzcd}
	0 & {\pi_1(\tilde{X},\tilde{x})} & {\pi_1(X,x)} & {\pi_0(\fibre{p}{x},\tilde{x})} & {\pi_0(\tilde{X},\tilde{x})}
	\arrow[from=1-1, to=1-2]
	\arrow["{\pi_1(p)}", hook, from=1-2, to=1-3]
	\arrow["\delta", from=1-3, to=1-4]
	\arrow["{\pi_0(\iota)}", two heads, from=1-4, to=1-5]
\end{tikzcd}.
\end{equation}
Since $\fibre{p}{x}$ is discrete, we have a bijection between $\fibre{p}{x}$ and $\pi_0(\fibre{p}{x},\tilde{x})$. Note that since $\pi_0$ is the pointed set of path-connected components rather than a group, `exactness' of this sequence should be understood in terms of maps between pointed sets:
\begin{itemize}
    \item The map $\delta$ takes a class of loops to a point in $\fibre{p}{x}$, that is path-connected with $\tilde{x}$ in $\tilde{X}$;
    \item The injection of loops in $\fungroup{\tilde{X},\tilde{x}}$ into  $\fungroup{X,x}$ are the classes of loops $\fungroup{X,x}$ that are sent to  $\tilde{x}$ by $\delta$;
    \item Furthermore, the path connected components of $\tilde{X}$ are enumerated by points of the fibre $\fibre{p}{x}$, modulo equivalence by being in the image of $\delta$. 
\end{itemize}

Due to the exactness of \cref{eq:fibration_LES}, the homotopy groups of the covering space can be recovered by knowing the map $\delta$; because $\fungroup{p}$ is a monomorphism, $\fungroup{\tilde{X}}$ can be recovered from the kernel of $\delta$, whereas $\pi_0(\tilde{X})$ can be recovered from the image of $\delta$ in $\pi_0(\fibre{p}{x})$. The higher homotopy groups are isomorphic to that of $X$ itself. 

In the case of a $\Gamma$-covering, since $\Gamma$ acts freely on fibres, we have a bijection between $\pi_0(\fibre{p}{x})$ and $\Gamma$ as sets. In fact, the set map $\delta$ can be refined as a group homomorphism $\mu$ between $\pi_1(X)$ and $\Gamma$. We call $\mu$ the \emph{monodromy homomorphism}. We define a generalisation of this group homomorphism to a groupoid homomorphism from $\Fgpd{X}$ to $\Gamma$ in \Cref{def:monodromy}. In \Cref{prop:monodromy_ses} and \Cref{prop:monodromy_cc}, we refine the observations we have surmised from \cref{eq:fibration_LES} in group theoretic terms. We first derive the monodromy homomorphism via homotopy lifting, and demonstrate some relevant properties. 

The monodromy homomorphism is constructed by homotopy lifting of paths. We recall the following elementary fact of covering spaces and path lifting. Each path $\gamma: (I,0) \to (X,x_0)$  has a unique lift $\tilde{\gamma}: (I,0) \to  (\tilde{X}, \tilde{x}_0)$ which commutes with the projection $\gamma = p \circ \tilde{\gamma}$. Furthermore, the lift respects homotopy: the homotopy class of paths $[\gamma]$ based at $x_0$ lifts to a homotopy class of paths $[\tilde{\gamma}]$ based at $\tilde{x}_0$. In particular, if $[\gamma]$ is a class of loops based at $x$, then $[\gamma]$ lifts to a unique homotopy class of paths $[\tilde{\gamma}]$ with fixed end points in $\tilde{x}_0, \tilde{x}_1 \in \fibre{p}{x_0}$. Furthermore, classes of loops that lift to loops (i.e. $\tilde{x}_1 = \tilde{x}_0$) form precisely the subgroup of $\fungroupbased{X}{x_0}$ that is the image of the induced map $p_\ast: \fungroupbased{\tilde{X}}{\tilde{x}_0} \to \fungroupbased{X}{x_0}$.

We can describe path lifting in terms of the fundamental groupoids of the covering spaces. For $x \in X$, let $\mathstar_{X} x$ denote the classes of paths in the fundamental groupoid with $x$ as the source
\begin{equation}
    \mathstar_X x  := \bigcup_{y \in X} \Fgpd{X}(x,y).
\end{equation}
Fixing a point $\tilde{x} \in \fibre{p}{x}$, the homotopy lifting of paths can be regarded as a map $\hslash_{\tilde{x}}:\mathstar_X x \to \mathstar_{\tilde{X}}\tilde{x}$, where $\hslash_{\tilde{x}}([\gamma])$ is the unique element of $\mathstar \tilde{x}$ such that 
\begin{equation}
  \Fgpd{p} \circ \hslash_{\tilde{x}}([\gamma]) =[\gamma]. 
\end{equation}
The lifting map $\hslash_{\tilde{x}}$ is a bijection~\cite[10.2.1]{Brown2006TopologyGroupoids}: any path $\tilde{\gamma}$ in $\tilde{X}$ based at $\tilde{x}$ projects onto a path $\gamma$ that can only lift to $\tilde{\gamma}$ for paths based at $\tilde{\gamma}$; and if two paths with the same end points are lifted to the same class of paths in $\tilde{X}$, then the projection of the homotopy between the two lifted paths rel end points is a homotopy  of the initial two paths in $X$ rel end points. Thus $\hslash_{\tilde{x}}$ is surjective and injective.

For $p$ a $\Gamma$-covering, since the group action on $\tilde{X}$ acts freely on the fibre of points, we can specify $\tilde{\gamma}(1) \in \fibre{p}{y}$ of a lifted path by the unique group element $g$ such that for a given fibre reference point $\lambda(y) \in \fibre{p}{y}$, we have $\tilde{\gamma}(1) = g \cdot \lambda(y)$. Since all paths in $[\gamma] \in \Fgpd{X}(x,y)$ lift to the same class $[\tilde{\gamma}] \in \Fgpd{\tilde{X}}(\lambda(x),g \cdot \lambda(y))$, we have a well-defined map from $\Fgpd{X}(x,y)$ to $\Gamma$. Note that this depends on the choice of fibre references $\lambda(x)$ and $\lambda(y)$.

If we have a choice of reference points in the fibre $\lambda: X \to \tilde{X}$, then we define a map $\mu_\lambda: \Fgpd{X} \to \Gamma$. We refer to this map as the \emph{monodromy homomorphism}. 

\begin{definition} \label{def:monodromy}
    Let $p: \tilde{X} \to X$ be a $\Gamma$-covering. For a given  pseudo-section of the covering $\lambda: X \to \tilde{X}$ (i.e. a not necessarily continuous map satisfying $p \circ \lambda = \iden_X$), The \emph{monodromy homomorphism} is a groupoid homomorphism $\nu_\lambda: \Fgpd{X} \to \Gamma$, where for $[\gamma] \in \Fgpd{X}(x,y)$, the group element $\mu_\lambda([\gamma])$ is the unique element such that 
\begin{equation}\label{eq:deckmapdef}
     \hslash_{\lambda(x)}([\gamma]) \in \Fgpd{\tilde{X}}(\lambda(x),  \mu_\lambda([\gamma]) \cdot \lambda(y)).
\end{equation}
\end{definition}
Because of the uniqueness of lifts of homotopy classes of paths rel end points, $\mu_\lambda$ is a well-defined map. \Cref{lem:concat} shows that $\mu_\lambda$ is indeed a groupoid homomorphism. 

\begin{lemma}\label{lem:concat} The monodromy homomorphism $\mu_\lambda$ of a $\Gamma$-covering is indeed a groupoid homomorphism. In other words, it sends homotopy classes of constant paths to the identity, and respects composition:
\begin{equation}
    \mu_{\lambda}([\gamma_0 \cdot \gamma_1])= \mu_{\lambda}([\gamma_0]) \cdot \mu_{\lambda}([\gamma_1]).
\end{equation}
\end{lemma}
\begin{proof}
    Since the homotopy class of the constant path $[\gamma = x]$ lifts to $[\tilde{\gamma} = \lambda(x)]$, $\mu_{\lambda}([\gamma = x]) = e$.
    
    We now verify that $ \mu_{\lambda}$ respects composition. Let $\tilde{\gamma}_i$ be the lifts of $\gamma_i$, and $\tilde{\gamma}$ be the lift of $\gamma$. For convenience of notation, let us denote $\tilde{x}_i = \lambda(\gamma_i(0))$, $g_i = \mu_{\lambda}([\gamma_i])$ for $i=0,1$, and $g = \mu_{\lambda}([\gamma])$. We align the end point $g_0 \cdot \tilde{x}_1$ of $\tilde{\gamma}_0$ with the starting point $\tilde{x}_1$ of $\tilde{\gamma_1}$ by considering the shifted path $g_0 \cdot \tilde{\gamma}_1$. The concatenation $\hat{\gamma} = \tilde{\gamma}_0 \cdot (g_0 \cdot \tilde{\gamma}_1)$ gives us a lift of $\gamma$ with base point $\tilde{x}_0$, and end point $\hat{\gamma}(1) = g_0 \cdot \tilde{\gamma}_1(1)= (g_0 \cdot g_1) \cdot \tilde{x}_2$. By uniqueness of lifts, we have $\hat{\gamma} = \tilde{\gamma}$. In particular, $\tilde{\gamma}(1) = \hat{\gamma}(1)$ implies $(g_0 \cdot g_1) \cdot \tilde{x}_2 = g \cdot \tilde{x}_2$. Since the action is free, we thus conclude that $g = g_0 \cdot g_1$.
\end{proof}

The following lemma describes the effect of the choice of pseudo-section on the the monodromy homomorphism.

\begin{lemma}\label{lem:gauge} Let $p: \tilde{X} \to X$ be a $\Gamma$-covering,and consider two pseudo-sections  $\lambda,\lambda': X \to \tilde{X}$ of the covering. Let $\theta: X \to \Gamma$ be the assignment such that $\lambda'(x) = \inv{\theta(x)} \cdot \lambda(x)$. Then for $[\gamma] \in \Fgpd{X}(x,y)$
\begin{equation}\label{eq:changerep}
    \mu_{\lambda'}([\gamma]) = \inv{\theta(x)} \cdot \mu_{\lambda}([\gamma])\cdot {\theta(y)}.
\end{equation}
In other words, $\theta$ defines a natural transformation $\theta: \mu_{\lambda} \Rightarrow \mu_{\lambda'}$ in the category of groupoids. 
\end{lemma}
\begin{proof} For convenience of notation denote $x = \gamma(0)$, $y = \gamma(1)$, $\tilde{x} = \lambda (\gamma(0))$ and $\tilde{y} = \lambda(\gamma(1))$. Let $\tilde{\gamma}: (I,0) \to (\tilde{X}, \tilde{x})$ be the $\tilde{\gamma}': (I,0) \to (\tilde{X}, \theta_x \cdot \tilde{x})$ lifts of $\gamma$ to the two different base points. Due to the uniqueness of lifts of paths, $\theta_x \cdot \tilde{\gamma} = \tilde{\gamma}'$. In particular, $\inv{\theta_x} \cdot \tilde{\gamma}(1) = \tilde{\gamma}'(1)$. Substituting \cref{eq:deckmapdef}, $\inv{\theta_x} \cdot \mu_{\lambda}([\gamma]) \cdot \tilde{y} = \mu_{\lambda'}([\gamma]) \cdot \tilde{y}' = \mu_{\lambda'}([\gamma]) \cdot \inv{\theta_y} \cdot \tilde{y}$. Because the action is free, we obtain \cref{eq:changerep}.    
\end{proof}

Consider the restriction $\mu_{\lambda\rvert x} : \fungroupbased{X}{x} \to \Gamma$ of the monodromy homomorphism to the fundamental group. Since $\mu_\lambda$ is a groupoid homomorphism (\Cref{lem:concat}),  the restriction to a vertex group of the domain $\mu_{\lambda\rvert x} $ is a group homomorphism. In topological terms, $\mu_{\lambda\rvert x}$ sends a class of loop $[\gamma]$ based at $x \in X$ to the unique deck transformation such that 
\begin{equation}
    \mu_{\lambda\rvert x} ([\gamma])  \cdot \hslash_{\lambda(x)}([\gamma])(1) = \lambda({x}).
\end{equation}
We note that $\mu_{\lambda\rvert x} $ only depends on the choice of reference point $\lambda(x)$ of $x$ in the fibre, and not any other point. If we choose a different choice of reference point in the fibre $\lambda'(x)$, \Cref{lem:gauge} implies we induce an inner automorphism:
\begin{equation}\label{eq:monodromy_change_lift}
   \mu_{\theta \cdot \lambda\rvert x}  = \inv{\theta} \cdot  \mu_{\lambda\rvert x} \cdot  {\theta},
\end{equation}
where $\theta \in \Gamma$ is the unique element such that $\lambda(x) = \theta \cdot \lambda'(x)$.

The monodromy group homomorphism $\mu_{\lambda\rvert x}$  encodes topological information about the $\Gamma$-covering. We first consider a statement in terms of path-connected coverings.
\begin{proposition}[{\cite[Proposition 1.31,39-40]{HatcherAlgebraicTopology}}]\label{prop:monodromy_ses}
    Suppose $p: \tilde{X} \to X$ is a path-connected $\Gamma$-covering. We have a short exact sequence
\begin{equation} \label{eq:monodromy_ses}
    \begin{tikzcd}
        0 & \fungroup{\tilde{X},\lambda(x)} & \fungroup{X,x} & \Gamma & 0 
    \arrow["\fungroup{p}",hook, from=1-2, to=1-3]
	\arrow[from=1-4, to=1-5]
	\arrow["{\mu_{\lambda\rvert x} }", two heads, from=1-3, to=1-4]
	\arrow[from=1-1, to=1-2]
    \end{tikzcd}.
\end{equation}
Letting $T: \Gamma \times \tilde{X} \to \tilde{X}$ denote the covering space action on $\tilde{X}$, the group of deck transformations $\decktransf{p}$ is isomorphic to $\Gamma$:
\begin{equation}
    \decktransf{p} = \settt{T(g,\cdot): \tilde{X} \to \tilde{X}}{ g \in \Gamma}.
\end{equation}
\end{proposition}

In fact, the monodromy homomorphism (up to isomorphism) classifies path-connected $\Gamma$-covering spaces.

\begin{theorem}[{\cite[Theorem 1.38-40]{HatcherAlgebraicTopology}}] There is a bijection between isomorphism classes of path-connected $\Gamma$-coverings, and normal subgroups of $\fungroup{X,x}$. Two $\Gamma$-coverings have the same monodromy homomorphism kernel iff they are isomorphic coverings; and every normal subgroup $H \triangleleft \fungroup{X,x}$ has a path-connected $\Gamma$-covering with monodromy homomorphism kernel $H$. 
\end{theorem}

If $\tilde{X}$ is not path-connected, $\mu_{\lambda \rvert x}$ is no longer surjective, though $\fungroup{p}$ is still a monomorphism whose image is the kernel of $\mu_{\lambda \rvert x}$. The  image of $\mu_{\lambda \rvert x}$in $\Gamma$ encodes information about the connected components of $\tilde{X}$.

\begin{proposition}
\label{prop:monodromy_cc} Consider $p: \tilde{X} \twoheadrightarrow X$ be a $\Gamma$-covering, where $X$ is path-connected. Let $\{\tilde{X}_i\}_{i \in \cI}$ be the set of connected components of $\tilde{X}$, and $p_i: \tilde{X}_i \to X$ be the restriction of $p$ to $\tilde{X}_i$. Let $\mu_{\lambda\rvert x} : \fungroup{X,x} \to \Gamma$ be a monodromy homomorphism, and $\Gamma'$ be the image of $\mu_{\lambda\rvert x}$. Then 
\begin{itemize}
    \item All restrictions $p_i$ are isomorphic $\Gamma'$-coverings of $X$;
    \item We have a bijection between left cosets of $\Gamma'$ in $\Gamma$ and connected components of $\tilde{X}$.
\end{itemize}
\end{proposition}

\begin{proof}
    Consider a connected component $\tilde{X}_i$ of $\tilde{X}$, and let $p_i$ be the restriction of $p$ to $\tilde{X}_i$ . We first show that $p_i$ is a surjection onto $X$. Consider some $x$ in the image of $p_i$. Since $X$ is path-connected, there is a path $\gamma$ from $x$ to any $y \in X$. Because $p$ is a covering, $\gamma$ lifts to a path $\tilde{\gamma}$ based at $\tilde{x}_i \in \fibre{p}{x} \cap \tilde{X}_i$, satisfying $p \circ \tilde{\gamma} = \gamma$. Evaluating this at $t= 1$, we thus have a point path connected to $\tilde{x}_i$ that is in $\fibre{p}{y}$. Thus, for any $y \in X$, there is an element in $\tilde{X}_i$ that projects onto $y$, i.e. $p_i$ is a surjection onto $X$.

    Recall $p$ is the quotient of $\tilde{X}$ by a $\Gamma$ action that satisfies the covering space action criterion \labelcref{C3}. The subgroup $\Gamma_i \subset \Gamma$ that sends $\tilde{X}_i$ to itself also satisfies \labelcref{C3}, and thus $p_i$ is a path-connected $\Gamma_i$-covering of $X$. Note that since $\mu_{\lambda \rvert x}$ is obtained by querying lifts of paths, $\mu_{\lambda \rvert x}$ of $p$ is the monodromy homomorphism of $p_i$ if $\lambda(x) = \tilde{x}_i$. Applying \Cref{prop:monodromy_ses}, we have $\Gamma_i \cong \imag \mu_{\lambda \rvert x} = \Gamma'$. Furthermore, the second half of \Cref{prop:monodromy_ses} implies the set of automorphisms of $p_i$  is given by $\Gamma'$
    \begin{equation*}
       \decktransf{p_i} =  \settt{T(g, \cdot): \tilde{X}_i \to \tilde{X}_i}{g \in \Gamma'}.
    \end{equation*}
    
    Consider some other connected component $\tilde{X}_j$ of $\tilde{X}$. Since the covering space action is free and transitive on fibres of $p$, if we pick some $x \in X$, and $\tilde{x}_i \in \fibre{p}{x} \cap \tilde{X}$, there a unique element $g_{ij} \in \Gamma$, such that $\tilde{x}_j = g_{ij} \cdot \tilde{x}_i$. Since the action of $g_{ij}$ on $\tilde{X}$ is a homeomorphism and must send each connected component to a homeomorphic connected components, the action of $g_{ij}$ restricted to $\tilde{X}_i$ is a homeomorphism sending $\tilde{X}_i$ to $\tilde{X}_j$. Furthermore, since $p$ is the quotient of the $\Gamma$ action, $p_j = p_i \circ T(g_{ij}, \cdot)$, where $T: \Gamma \times \tilde{X} \to \tilde{X}$ denotes the $\Gamma$-action on $\tilde{X}$. Thus, $T(g_{ij}, \cdot)$ is a deck transformation between $p_i$ and $p_j$. This implies $\tilde{X}$ consists of isomorphic path-connected coverings over $X$. $\Gamma_i \cong \Gamma_j \cong \Gamma'$.

    The $\Gamma$-group action on $\tilde{X}$ descends to a $\Gamma$-group action on its space of connected components $\pi_0 \tilde{X}$. Because the action of $\Gamma$ on $\tilde{X}$ is a covering space action, $\Gamma$ descends to a transitive action on $\pi_0 \tilde{X}$; in particular, we have $\Gamma \cdot \tilde{X}_i = \pi_0 \tilde{X}$. Thus the orbit stabiliser theorem implies there is a bijection between $ \pi_0 \tilde{X}$ and cosets of the stabiliser of $\tilde{X}_i$. However, the stabiliser of $\tilde{X}_i$ is precisely the group actions that send points on $\tilde{X}_i$ to points on $\tilde{X}_i$, i.e. $\imag \mu_{\lambda \rvert x}$. Thus, we have a bijection between $\pi_0 \tilde{X}$ and cosets of $\imag \mu_{\lambda \rvert x}$.
\end{proof}

\subsection{Transition Maps of the Covering}\label{ssec:transition}

We consider $\Gamma$-coverings that satisfy a more stringent topological condition than \labelcref{C1}: that is, the base space $X$ admits a good cover $\cU$. Since each $U_i \in \cU$ is contractible, the homotopy lifting of $U_i$ to $\tilde{X}$ implies there is a chart $\varphi_i: \fibre{p}{U_i} \to U_i \times \Gamma$. We call the set of local trivialisations $\sett{\varphi_i}_{i \in \cI}$ the atlas corresponding to the covering $\cU$. 

Given the atlas $\sett{\varphi_i}_{i \in \cI}$, the system of transition maps is the collection of maps $\psi_{ji} = \varphi_j \circ \inv{\varphi_i} : (U_i \cap U_j) \times \Gamma \to (U_i \cap U_j) \times  \Gamma$ where $U_i \cap U_j \neq \emptyset$. Because $\cU$ is a good cover, $U_i \cap U_j$ is connected. Furthemore, since $\Gamma$ is discrete, there is a unique $t_{ij}\in\Gamma$ such that transition maps correspond to a local group action $\psi_{ji}(x,g) = (x,g \cdot t_{ij}) $. Due to $\psi_{ki} = \psi_{kj}\psi_{ji}$, the equivariance of $\varphi_i$ implies $t_{ij}t_{jk} = t_{ik}$ for $U_i \cap U_j \cap U_k \neq \emptyset$. By~\Cref{lem:cocycle_egpd}, $t$ extends to an edge groupoid homomorphism $t: \Egpd{\nerve{\cU}} \to \Gamma$, whose evaluation on edges of the nerve yields the group elements $t_{ij}$ that specify $\psi_{ji}$. We can $t$ the \emph{transition homomorphism} of the atlas.

Given a transition homomorphism associated to an atlas, we can reconstruct the covering up to isomorphism by gluing together local models of the covering space $U_i \times \Gamma$ on overlaps. We recall the following facts about transition maps of principal $\Gamma$-bundles, specialised to $\Gamma$-coverings and good covers as in our setting; see for example~\cite{Husemoller1966FibreBundles} for a detailed exposition of on the subject.

\begin{proposition} \label{prop:transition_homomorphism}
    Let $X$ be a space with a good cover $\cU = \sett{U_i}_{i \in \cI}$. 
    \begin{itemize}
        \item For any groupoid homomorphism $t: \Egpd{\nerve{\cU}} \to \Gamma$, there is a $\Gamma$-covering of $p: \tilde{X} \twoheadrightarrow X$ and an atlas $\sett{\varphi_i: \fibre{p}{U_i} \to U_i \times \Gamma}_{i \in \cI}$ of the covering, such that its transition homomorphism is precisely $t$~\cite[\S5 Theorem 3.2]{Husemoller1966FibreBundles}.
    \item  Consider two $\Gamma$-coverings $p: \tilde{X} \to X$ and $p':  \tilde{X}' \to X$ of $X$. There is a $\Gamma$-equivariant automorphism $\alpha \in \decktransf{p}$ relating $p' = p \circ \alpha$, iff for any arbitrary choices of atlases for $p$ and $p'$, the transition homomorphisms $t,t': \Egpd{\nerve{\cU}} \to \Gamma$ are related by a natural isomorphism $\theta: t \Rightarrow t'$~\cite[\S5 Prop 2.5]{Husemoller1966FibreBundles}.
    \end{itemize}
\end{proposition}

We now relate the transition homomorphism with the monodromy homomorphism. Since the transition homomorphism is defined on the edge groupoid of the nerve of a good cover, while monodromy is defined on the fundamental groupoid of the underlying space, we relate the two homomorphisms via the realisation and snapping groupoid homomorphisms $R$ and $S$. We show that by choosing compatible choices of pseudo-sections and atlases, monodromy -- an attachment of $\Gamma$ to any class of paths in $X$ -- is an extension of the data encoded by the transition homomorphism on the nerve of a good cover of the space.

\begin{proposition} \label{prop:monodromy_compression} Consider $p: \tilde{X} \to X$ a $\Gamma$-covering of a space with good cover $\cU = \sett{U_i}_{i\in \cI}$. Let $S: \Fgpd{X} \leftrightarrows \Egpd{\nerve{\cU}}: R$ be the equivalence of categories as specified in \Cref{prop:equivalence_groupoids}, such that $SR(i)= i$ for $i \in \cI$. Then given a pseudo-section $\lambda: X \to \tilde{X}$:
\begin{itemize}
    \item There is a choice of atlas $\sett{\varphi_i}_{i \in \cI}$, such that the transition homomorphism $t: \Egpd{\nerve{\cU}} \to \Gamma$ is the pullback of the monodromy homomorphism along the realisation homomorphism: $t = \mu_\lambda \circ R$;
    \item The following diagram commutes iff  the pseudo-section is given by $\lambda(x) = \varphi_{S(x)}(x,e)$:
    \begin{equation}
    \begin{tikzcd}[cramped]
	\Egpd{\nerve{\cU}} && \Fgpd{X} \\
	& \Gamma
	\arrow["R", curve={height=-6pt}, from=1-1, to=1-3]
	\arrow["t"', from=1-1, to=2-2]
	\arrow["S", curve={height=-6pt}, from=1-3, to=1-1]
	\arrow["{\mu_{\lambda}}", from=1-3, to=2-2]
\end{tikzcd}.
\end{equation}
\end{itemize}
\end{proposition}

\begin{proof}
To prove the first point, we show that 
for any pseudo-section $\lambda$ of the covering, and any choice of transition homomorphism $t$, there is a unique natural isomorphism $\theta: \mu_\lambda \circ R \Rightarrow t$. Thus by \Cref{prop:transition_homomorphism} there is a choice of atlas such that $\theta$ is the identity and the transition homomorphism $t$ is a sample of the monodromy $\mu_\lambda$. Since the covering group action is free and transitive on fibres of $p$, for any pseudo-section $\lambda$, we can find a unique assignment $\theta: \cI \to \Gamma$, such that 
    \begin{equation*}
        \lambda(R(i)) =  \inv{\varphi}_{i}(R(i), \theta(i)).
    \end{equation*}
 Consider an edge $ij$ in $\nerve{\cU}$ and the realisation of edge path $R(ij) \in \Fgpd{X}(R(i),R(j))$. By definition there is a representative path $\gamma$ of the class $R(ij)$ that is carried by $ij$. In other words $\gamma \subset U_i \cup U_j$ with $\gamma(0) \in U_i$ and $\gamma(1) \in U_j$. Since $\cU$ is a good cover, $U_i \cup U_j$ is simply connected, due to homotopy lifting properties, we can choose a section $s_{ij}: U_i \cup U_j \to \tilde{X}$ such that $s_{ij}(R(i)) = \lambda(R(i)) = \inv{\varphi}_{i}(R(i),\theta_i)$. Since this is a section of the covering, $s_{ij}$ is uniquely specified by where it sends $R(i)$. In particular, since $\inv{\varphi_i}(\cdot, \theta_i)$ is a section, uniqueness of lifts implies $s_{ij}\rvert_{U_i} = \inv{\varphi_i}(\cdot, \theta_i)$. Consider  $Y_{ij} = s_{ij}(U_i \cap U_j)$, and the map $\varphi_j \circ s_{ij}\rvert_{U_i \cap U_j}: Y_{ij} \to U_j \times \Gamma$; since $s_{ij}\rvert_{U_i \cap U_j} = \inv{\varphi_i}(\cdot, \theta_i)\rvert_{U_i \cap U_j}$, we have  $\varphi_j \circ s_{ij}\rvert_{U_i \cap U_j}(x) = \psi_{ji}(x, \theta_i) = (x, \theta_i t_{ij})$. Because $\varphi_j$ is a homeomorphism, $s_{ij}\rvert_{U_i \cap U_j} = \inv{\varphi_j}(\cdot, \theta_i t_{ij})\rvert_{U_i \cap U_j}$. We now note that $\inv{\varphi_j}(\cdot, \theta_i t_{ij})$ and $s_{ij}\rvert_{U_j}$ are lifts of $U_j$ to the covering space, and lifts are uniquely determined by where they send a single point. Since the two lifts coincide on the overlap $U_i \cap U_j$, we deduce that $s_{ij}\rvert_{U_j} = \inv{\varphi_j}(\cdot, \theta_i t_{ij})\rvert_{U_j}$.

Due to uniqueness of lifts, the lift of $\gamma$ to $\tilde{X}$ such that $\tilde{\gamma}(0) = \lambda(R(i))$ is $\tilde{\gamma} = s_{ij} \circ \gamma$. Because $\gamma(1) = R(j)$, we have $\tilde{\gamma}(1) = \inv{\varphi_j}(R(j), \theta_i t_{ij}) $. Because $\inv{\varphi_j}$ is $\Gamma$-equivariant, $\tilde{\gamma}(1) = \theta_i t_{ij}\inv{\theta_j}  \cdot \inv{\varphi_j}(R(j), \theta_j) = \theta_i t_{ij}\inv{\theta_j} \cdot \lambda(R(j))$. Therefore, the definition of $\mu_\lambda$ implies $\mu_\lambda
    \circ R([ij]) = \theta_i t_{ij} \inv{\theta_j}$. 
Extending to all classes of edge paths by composition of edges, the assignment $\theta: \cI \to \Gamma$ specifies a natural isomorphism $\theta: \mu_\lambda \circ R \Rightarrow t$.

Recall we have a natural isomorphism $\varsigma_\bullet: \iden_{\Fgpd{X}} \Rightarrow RS$, where $\varsigma_x$ is the unique class of paths in $\Fgpd{X}(x,RS(x))$ with a representative  path $s_x$ contained in $U_{S(x)}$. Thus $\mu_\lambda \circ RS([\gamma]) = \mu_\lambda(\varsigma_x) \mu_\lambda([\gamma]) \inv{\mu_\lambda}(\varsigma_y)$. We choose a section $\sigma: U_{S(x)} \to \fibre{p}{U_{S(x)}}$, such that $\sigma(\cdot) = \varphi_{S(x)}( \cdot , \theta(S(x)))$. In particular, on the end points of $s_x$, we have  $\sigma(x) = \varphi_{S(x)}(x, \theta(S(x)))$ and $\sigma(RS(x)) = \varphi_{S(x)}( RS(x), \theta(S(x)))$. The section lifts the path $s_x$ to a path between $\sigma(x)$ and $\sigma(RS(x))$. 

For monodromy, we choose a pseudo-section $\lambda_S(\cdot) = \varphi_{S(\cdot)}(\cdot,\theta(S(\cdot)))$ determined purely by $S$, the choise of $\theta: \mathcal{I} \to \Gamma$, and the atlas $\varphi_\cdot$ . For the path $s_x$,  the unique lifted path based at  $\lambda_S(x) =  \sigma(x)$ has the end point 
$$\sigma(RS(x)) = \varphi_{S(x)}( RS(x), \theta(S(x)))  = \varphi_{SRS(x)}( RS(x), \theta(SRS(x))) = \lambda_S(x),$$  
because we have chosen $R,S$ such that $SR(i) = i$ for all indices of cover elements. Thus, if we choose $\lambda = \lambda_S$, then $\mu_\lambda([s_x]) = \mu_\lambda(\varsigma_x) = e$. Therefore we obtain $\mu_\lambda \circ RS = \mu_\lambda$. Conversely, if $\mu_\lambda = \mu_\lambda \circ RS$ for all classes of paths, then for any $x \in U_{i}$ such that $i  = S(x)$, and $\inv{\varsigma_x} \in \Fgpd{X}(R(i),x)$, we have $\mu_\lambda(\inv{\varsigma_x}) = e$. In other words, the choice of pseudo-section for the monodromy homomorphism is the lift along end paths of points $\lambda(x) = \lambda_{S}(x)$. Repeat for all points $x \in \fibre{S}{i}$, and all $i \in \cI$, we thus obtain $\mu_\lambda = \mu_\lambda \circ RS$ implying $\lambda = \lambda_S$.

    If we set $\theta = e$, then we have $t = \mu_\lambda \circ R$. Choosing $\lambda = \lambda_S$, we have $\mu_\lambda = \mu_\lambda \circ RS$, thus implying $\mu_\lambda = t \circ S$. 
\end{proof}

\subsection{Induced Coverings}

Given a $\Gamma$-covering $p: \tilde{Y} \twoheadrightarrow Y$, a map $f: X \to Y$ from a space $X$ satisfying the topological criterion \labelcref{C1}, consider the pullback of the pair of maps $X \xrightarrow{f} Y \xleftarrow{p} \tilde{Y}$:
\begin{equation} \label{dgm:induced_covering}
    \begin{tikzcd}
	{\tilde{X}} & {\tilde{Y}} \\
	X & Y
	\arrow["f^\ast p"',  two heads, dashed, from=1-1, to=2-1]
	\arrow["{f_\sharp}", two heads, dashed, from=1-1, to=1-2]
	\arrow["\lrcorner"{anchor=center, pos=0.125}, draw=none, from=1-1, to=2-2]
	\arrow["p",  two heads, from=1-2, to=2-2]
	\arrow["f", from=2-1, to=2-2]
\end{tikzcd}.
\end{equation}
Explicitly, the pullback space $\tilde{X}$ is given by
\begin{equation}
    \tilde{X} = X \times_Y \tilde{Y} = \settt{(x,\tilde{y}) \in X \times \tilde{Y} }{f(x) = p(\tilde{y})},
\end{equation}
and $f^\ast p$ is the restriction of the projection $X \times \tilde{Y} \to X$ on $\tilde{X}$. Similarly, the map $f_\sharp$ is the restriction of the projection $X \times \tilde{Y} \to \tilde{Y}$ on $\tilde{X}$. We recall that $p: \tilde{Y} \twoheadrightarrow Y$ is a $\Gamma$-covering, and $f: X \to Y$ is a map between spaces satisfying \labelcref{C1}, then:
\begin{enumerate}
    \item The map $f^\ast p :\tilde{X} \to X$ in the pullback \cref{dgm:induced_covering} is a $\Gamma$-covering;
    \item The map $f_\sharp$ in the pullback \cref{dgm:induced_covering} is equivariant w.r.t. the $\Gamma$-covering space actions on $\tilde{X}$ and $\tilde{Y}$ respectively.
\end{enumerate}

\begin{example} \label{ex:lift_inconsistent_eg}
    Let $f: I \to S^1$ be the map that identifies the end points of the interval to form a circle. Let $p: \R \to S^1$ be the universal covering $p(t) = e^{2\pi i t}$ of $S^1$. Then the induced covering space is given by $\tilde{X} = \settt{(t, t+n)}{t \in [0,1],\ n \in \Z}$, the covering map $f^\ast p$ and $f_\sharp$ are the projections onto the first and second factors respectively. Note that the projection map $f_\sharp$  is a surjection onto $\R$ as it glues together end points of disjoint intervals in $\tilde{X}$.
\end{example}

Let us consider then the monodromy homomorphism of the induced covering. We show that if the pseudo-section $\lambda: X \to \tilde{X}$ is chosen carefully, the monodromy homomorphism $\mu^\ast_\lambda: \Fgpd{X} \to \Gamma$ of the induced covering $f^\ast p$  is the pullback of a monodromy homomorphism $\mu_{\lambda_Y}: \Fgpd{Y} \to \Gamma$ of $p$ along $f$. The choices $\lambda$ and $\lambda_Y$ need to be compatible; in other words, we choose $\lambda(x) = (x, \lambda_Y \circ f(x))$. This is equivalent to forcing $\lambda$ to satisfy $f(a) = f(b) \iff f_\sharp \lambda(a) = f_\sharp \lambda(b)$. If this is not satisfied, then there are cases, such as \Cref{ex:lift_inconsistent_eg}, where  $\mu_\lambda$ cannot be expressed as a pullback of $\mu_{\lambda_Y}$.

\begin{lemma}\label{prop:pullback_monodromy_full} Let $f^\ast p: \tilde{X} \to X$ be the pullback of a $\Gamma$-covering $p: \tilde{Y} \to Y$ along a continuous map $f: X \to Y$, where $X$ is path-connected. Suppose for some point $x \in X$, the pseudo-sections $\lambda_Y: Y \to \tilde{Y}$ and $\lambda: X \to \tilde{X}$ satisfy $f_\sharp \circ \lambda(x) = \lambda_Y \circ f(x)$. Then the pesudo-sections are compatible for all points in $X$ iff $\mu_\lambda$ is the pullback of $
\mu_{\lambda_Y}$:  
    \begin{align} \label{eq:induced_monodromy_full}
    \mu_{\lambda}^\ast &= \mu_{\lambda_Y} \circ \Fgpd{f} \iff f_\sharp \circ \lambda = \lambda_Y \circ f. 
\end{align}
\end{lemma}

\begin{proof} 
We first proof a slightly weaker result first. If the lifts satisfy $f_\sharp \circ \lambda(x_0) = \lambda_Y \circ f(x_0)$, then for classes of paths between a pair of points $x_0, x_1 \in X$, we have 
\begin{align} \label{eq:induced_monodromy_full_lem}
    \forall [\gamma] \in \Fgpd{X}(x_0,x_1),\quad \mu_{\lambda}([\gamma]) = \mu^\ast_{\lambda_Y} \circ \Fgpd{f}([\gamma]) \iff f_\sharp \circ \lambda(x_1) = \lambda_Y \circ f(x_1). 
\end{align}
We show this by considering path lifts in the induced covering. Let us consider a path $\gamma: I \to X$, from $\gamma(0) = x_0$ to $\gamma(1) = x_1$, and $f\circ \gamma: I \to Y$. The former path admits a lift $\gamma^\ast$ to $\tilde{X}$ and the latter a lift $(f \circ \gamma)^\ast$ to $\tilde{Y}$. If we choose base points of the lifts consistently, such that $(f \circ \gamma)^\ast(0) = f_\sharp \gamma^\ast(0) = \lambda_Y \circ f(x_0)$, then by uniqueness of lifts, we have $(f\circ \gamma)^\ast = f_\sharp \circ \gamma^\ast$. This statement is expressed in the following commutative diagram:
\begin{equation}\label{dgm:induced_path_lift}
    \begin{tikzcd}
	{(I,0)} & {(\tilde{X},\tilde{x})} & {(\tilde{Y},f_\sharp(\tilde{x}))} \\
	& {(X,x)} & {(Y,f(x))}
	\arrow["{f^\ast p}"', from=1-2, to=2-2]
	\arrow["{f_\sharp}", from=1-2, to=1-3]
	\arrow["p", from=1-3, to=2-3]
	\arrow["f", from=2-2, to=2-3]
	\arrow["\lrcorner"{anchor=center, pos=0.125}, draw=none, from=1-2, to=2-3]
	\arrow["\gamma"', from=1-1, to=2-2]
	\arrow["{\gamma^\ast}", from=1-1, to=1-2]
	\arrow["{(f\circ\gamma)^\ast}", curve={height=-18pt}, from=1-1, to=1-3]
\end{tikzcd}
\end{equation}
Consider then the monodromy homomorphisms $\mu_{\lambda_Y}: \Fgpd{Y} \to \Gamma$, and $\mu^\ast_{\lambda}: \Fgpd{X} \to \Gamma$ defined by respective pseudo-sections $\lambda_Y$ and $\lambda$. For brevity of notation, let $\tilde{y}_i = \lambda_Y \circ f(x_i)$, and $\tilde{x}_i = \lambda(x_i)$. From the definition of the monodromy homomorphism, we must have
\begin{align*}
    f_\sharp \circ \gamma^\ast (1) &= \fun{f_\sharp}{\mu^\ast_{\lambda}([\gamma]) \cdot \tilde{x}_1 }= \mu^\ast_{\lambda}([\gamma]) \cdot f_\sharp(\tilde{x}_1) \\
    (f \circ \gamma)^\ast (1)  &= \mu_{\lambda}([f \circ \gamma])  \cdot \tilde{y}_1,
\end{align*}
where in the first line we use the fact that $f_\sharp$ is an equivariant map. Because the upper triangle commutes in \cref{dgm:induced_path_lift}, $(f \circ \gamma)^\ast = f_\sharp \circ \gamma^\ast$ implies for any $[\gamma] \in \Fgpd{X}(x_0,x_1)$,
\begin{equation}\label{eq:dummy0}
    \mu_{\lambda_Y}([f\circ\gamma]) \cdot \tilde{y}_1  = \mu^\ast_{\lambda}([\gamma]) \cdot f_\sharp(\tilde{x}_1).
\end{equation}
Since the covering $p$ satisfies \labelcref{C3}, the action of $\Gamma$ on $\tilde{y}$ is free. 

If we have $\mu_{\lambda_Y}([f\circ\gamma]) = \mu^\ast_{\lambda}([\gamma]) $ for some $[\gamma]$, then \cref{eq:dummy0} implies $\lambda_Y \circ f(x_1) = f_\sharp \circ \lambda(x_1)$ due to the free group action. Conversely. If we have $f_\sharp \circ \lambda(x_1) = \lambda_Y \circ f(x_1)$, then due to the free group action, this implies $\mu_{\lambda_Y}([f\circ\gamma]) = \mu^\ast_{\lambda}([\gamma])$ for any $[\gamma] \in \Fgpd{X}(x_0,x_1)$.

We have thus proved \cref{eq:induced_monodromy_full_lem}.
If $f_\sharp \circ \lambda = \lambda_Y \circ f$ for all points, then $\mu^\ast_{\lambda} = \mu_{\lambda_Y} \circ \Fgpd{f}$ for all homotopy classes of paths. Conversely, if for all homotopy classes of paths $\mu_{\lambda} = \mu_{\lambda_Y} \circ \Fgpd{f}$, and the lifts are synchronised $f_\sharp \circ \lambda(x) = \lambda_Y \circ f(x)$ at some point $x$, then since $X$ is path-connected, for any point $x'$ we can consider a path from $x$ to $x'$; then the fact that for $[\gamma] \in \Fgpd{X}(x,x')$ we have  $\mu_{\lambda}([\gamma]) = \mu_{\lambda_Y} \circ \Fgpd{f}([\gamma])$  implies $f_\sharp \circ \lambda(x') = \lambda_Y \circ f(x')$.
\end{proof}

Since the pullback of a monodromy homomorphism is a monodromy homomorphism for some choice of lift, we can choose a compatible atlas whose transition homomorphism is a sample of the pulled back monodromy homomorphism. 

\begin{proposition} \label{prop:monodromy_groupoid_big}
    Let $f^\ast p: \tilde{X} \to X$ be the pullback of a $\Gamma$-covering $p: \tilde{Y} \to Y$ along a continuous map $f: X \to Y$. If $X$ admits a good cover $\cU$, then for any pseudo-section $\lambda: Y \to \tilde{Y}$, and associated monodromy homomorphism $\mu_\lambda: \Fgpd{Y} \to \Gamma$, there is some choice of atlas of the pullback covering $f^\ast p$, such that the transition homomorphism $t^\ast : \Egpd{\nerve{\cU}} \to \Gamma$ commutes with the following homomorphisms:
    \begin{equation}\label{dgm:monodromy_groupoid_big}
        \begin{tikzcd}[cramped]
	{\Pi X} & {\Pi Y} \\
	{\mathscr{E}\mathsf{N}(\mathscr{U})} & \Gamma
	\arrow["{\Pi f}", from=1-1, to=1-2]
	\arrow["{\mu_\lambda}", from=1-2, to=2-2]
	\arrow["R", from=2-1, to=1-1]
	\arrow["{t^\ast }", from=2-1, to=2-2]
\end{tikzcd}.
    \end{equation}
\end{proposition}
\begin{proof}
    We choose the lift of points in $X$ to the induced covering space $\tilde{X}$ satisfying the conditions of \Cref{prop:pullback_monodromy_full}. Since $\mu_\lambda \circ \Fgpd{f}$ is equal to the monodromy homomorphism of such a lift, then we can choose the atlas as specified in \Cref{prop:monodromy_compression} such that its transition homomorphism is a sample of the monodromy. 
\end{proof}
\begin{remark}
    While we can sample from a pullback monodromy to recover a transition homomorphism, in general we cannot extend a transition homomorphism of a pullback covering to recover the pullback monodromy. That is because the extension requires a type of lift of points in the induced covering specified in \Cref{prop:monodromy_compression}, that might not be compatible with the requirements on the lift as specified in \Cref{prop:pullback_monodromy_full}.
\end{remark}
 
\begin{remark} Consider a simplicial map $L \xrightarrow{\psi} K$. If we have a map $f: |K| \to X$, because the realisation functor $R$ is a natural transformation $R: \cE \Rightarrow \Pi$, the combination of \Cref{prop:monodromy_groupoid_big} and \Cref{prop:gr_functorial} implies if $f^\ast t: \Egpd{K} \to \Gamma$ and $(f \circ |\Psi|)^\ast t: \Egpd{L} \to \Gamma$ are transition homomorphism derived from the same monodromy homomorphism $\mu_\lambda$, then we can write $(f \circ |\Psi|)^\ast t = f^\ast t \circ \Egpd{\Psi}$.
\end{remark}
Equivalence up to natural isomorphism allows us to make some statements about the fundamental groups of $X$ and $Y$. In particular, if we choose the covering $p$ on $Y$ to be the universal covering of $Y$, then the transition homomorphism of the pullback covering recovers the homomorphism on the fundamental group induced by the map $f: X \to Y$ up to isomorphism. 

By restricting~\Cref{prop:monodromy_groupoid_big} from groupoids to vertex groups of the groupoids, we immediately obtain the following. 
\begin{corollary} \label{prop:monodromy_group_big}
     Let $f^\ast p: \tilde{X} \to X$ be the pullback of the universal covering $p: \tilde{Y} \to Y$ of $Y$ along a continuous map $f: X \to Y$. Assume $X$ admits a good cover $\cU$, and let $t^\ast: \Egpd{\nerve{\cU}} \to \Gamma$ be the transition homomorphism for some choice of atlas on $f^\ast p$. Then the group homomorphism $\Egroup{\nerve{\cU},u} \xrightarrow{\imath_u} \Egpd{\nerve{\cU}} \xrightarrow{t^\ast} \Gamma$ recovers the induced homomorphism of $f$ on the fundamental group up to isomorphism:
     \begin{equation} \label{dgm:th_recovers_induced_f}
         \begin{tikzcd}[cramped]
         \fungroup{X,R(u)} & \fungroup{Y,f\circ R(u)} \\
	\Egroup{\nerve{\cU},u} & \Gamma 
	\arrow["\fungroup{f}", from=1-1, to=1-2]
	\arrow["\cong","R"', from=2-1, to=1-1]
	\arrow["\cong", "\mu_{\lambda \rvert f \circ R(u)}"',from=1-2, to=2-2]
	\arrow["t^\ast \circ \imath_u", from=2-1, to=2-2]
     \end{tikzcd}.
     \end{equation}
\end{corollary}
\Cref{prop:monodromy_group_big} immediately suggests that $\fungroup{f}$ is encoded on finite data $t^\ast$ specified on the edges of the nerve, one that can be obtained by realising those edges as explicit paths, and lifting the images of the paths in the universal cover of $X$. We emphasise that $t^\ast$ is dependent on a choice of atlas in the pullback covering $f^\ast p$, as stated in \Cref{prop:monodromy_groupoid_big}. As a change of atlas -- that is, a gauge transformation -- induces a natural transformation between the groupoid homomorphisms (\Cref{prop:transition_homomorphism}), the induced homomorphisms on the edge group is uniquely defined up to conjugation. Thus, if we have different spaces mapped into $Y$, we cannot meaningfully compare the images of their fundamental/edge groups in $\Gamma$ due to dependency on such choices. However, we can mod out the effect of the gauge symmetry by comparing the conjugacy classes of the image of loops. Since $\Gamma$ is isomorphic to the fundamental group of $Y$, the set of conjugacy classes of $\Gamma$ are in bijective correspondence with $[S^1, Y]$, the homotopy types of maps from $S^1$ into $Y$.

An even simpler approach to mod out the conjugacy would be to work with homology groups instead. If we apply the groupoid homology functor $\mathcal{H}_1: \Grpd \to \AbGrp$ to~\cref{dgm:monodromy_groupoid_big} in \Cref{prop:monodromy_groupoid_big}, the equivalence between the first groupoid homology of the fundamental group and first singular homology, as well as that between the first groupoid homology of the edge groupoid of the nerve and first simplicial homology, yields a homomorphism $\tau: \homol{1}{\nerve{\cU}} \to \abel{\Gamma}$ that recovers the induced homomorphism of $f$ on singular homology groups.

\begin{proposition} \label{prop:monodromy_homology_big}
    Consider a space $X$ admitting a good cover $\cU$. Let $f: X \to Y$ be a continuous map, and $f^\ast p: \tilde{X} \to X$ be the covering map induced from the universal covering of $Y$. Let $t: \Egpd{\nerve{\cU}} \to \Gamma$ be a transition homomorphism of the covering $f^\ast p$. Then for any abelian group $A$, we have a well-defined group homomorphism $\tau: H_1(\nerve{\cU};A) \to \abel{\Gamma} \otimes A$ given by,
     \begin{equation}
         \tau: \left[ \sum_i \sigma_i \otimes a_i \right]  \mapsto \sum_i \abel{t}([\sigma_i]) \otimes a_i,
     \end{equation}
     which recovers the induced homomorphism of $f$ on homology up to isomorphism
     \begin{equation} \label{dgm:th_recovers_induced_hf}
         \begin{tikzcd}[cramped]
         H_1(X;A) & H_1(Y;A) \\
	H_1(\nerve{\cU};A) & \abel{\Gamma} \otimes A 
	\arrow["H_1(f)", from=1-1, to=1-2]
	\arrow["\cong", from=2-1, to=1-1]
	\arrow["\cong", from=1-2, to=2-2]
	\arrow["\tau", from=2-1, to=2-2]
     \end{tikzcd}.
     \end{equation}
\end{proposition}

\begin{proof} 
We assemble $\tau$ by considering the following commutative diagram 
    \[\begin{tikzcd}[column sep=large]
	& {H_1(X;A)} & {H_1(Y;A)} \\
	& \HGrpd{1}{\Fgpd{X}} &\HGrpd{1}{\Fgpd{Y}} \\
	{H_1(\nerve{\cU};A)} & {\HGrpd{1}{\Egpd{\nerve{\cU}};A}} & {\HGrpd{1}{\Gamma;A}} & {\abel{\Gamma} \otimes A}
	\arrow["{H_1(f)}", from=1-2, to=1-3]
	\arrow["\cong", from=1-2, to=2-2]
	\arrow["\cong"', from=1-3, to=2-3]
	\arrow["{\HGrpd{1}{\Pi f}}", from=2-2, to=2-3]
	\arrow["{\HGrpd{1}{\mu_\lambda}}", "\cong"', from=2-3, to=3-3]
	\arrow["\cong", from=3-1, to=3-2]
	\arrow["\HGrpd{1}{R}","\cong"', from=3-2, to=2-2]
	\arrow["{\HGrpd{1}{t}}", from=3-2, to=3-3]
	\arrow["\cong", from=3-3, to=3-4]
\end{tikzcd}\]
The top square commutes due to the natural isomorphism between the first singular homology and the first groupoid homology of the fundamental groupoid \Cref{prop:Hfgpd_func}; the bottom square commutes due to applying the groupoid homology functor $\HGrpd{1}{-;A}$ to \cref{dgm:monodromy_groupoid_big}. The isomorphisms on the bottom row are due to \Cref{prop:hegrpd_func} and \Cref{prop:grp_homology}. Composing the explicit homomorphisms described in those propositions recovers the explicit expression for $\tau$. 
\end{proof}

\section{Recovering Induced Coverings on Data from Geodesics}
We focus on the case where the covering space in question is a Riemannian universal covering of a compact manifold $M$, and our data consists of a finite point cloud $\X \subset M$. The typical task of topological analysis is to summarise the `shape' of the $\X$, in relation to its inclusion into the ambient space. In this setting, we consider the question of whether $\X$ `captures' some non-trivial first homology of the ambient manifold. This is of course dependent on a construction of choice on $\X$ to simulate its shape. We consider two such constructions on $\X$. The first is the \v{C}ech complex on the point cloud, where we restrict the thickening radius $\epsilon$ to be less than the convexity radius of the manifold. The latter consists of graphs $G$ obtained by interpolating minimising geodesics between points in $\X$; we show in \Cref{sec:generic} that for generic point clouds, there is a unique choice of interpolating minimising geodesics given a point cloud. In the first setting, we aim to recover the homomorphism on first homology induced by the inclusion $\X^\epsilon \hookrightarrow M$. In the latter setting, we consider the homomorphism induced by the map $f: |G| \to M$ with the interpolated minimising geodesics as its image.   We show that in both cases, we only need the distances between points on $\X$, the distance between points in its fibre $\fibre{p}{\X}$ in the universal covering space, and the group action on the universal covering space to recover the induced homomorphisms. This is described in \Cref{prop:basic_infer_t}. 

As a simple application, we apply our framework to analyse cycles formed by four points randomly sampled on the collection of surfaces described in~\Cref{ex:surface_coverings}. The distribution of the persistent homology of such cycles is called the first \emph{principal persistence measure} of the surface from which they are sampled~\cite{Gomez2024CurvatureDiagrams}. We show that our framework can be applied to detect the homology classes of such four point cycles. By doing so we can  decompose the persistence measure by homology classes of the manifold. Since principal persistence measures are used as geometric features of spaces, such decomposition cna help us better interpret the geometric signatures detected by the persistence measure. 

\subsection{Geometry Preliminaries}
We first review some facts about smooth coverings of complete, Riemannian manifolds. We refer the reader to~\cite{Gallot2004RiemannianGeometry,LeeIntroductionManifolds,Gallier2020DifferentialGroups} for a comprehensive account of the subject. A \emph{smooth covering} is a smooth topological covering $p: \tilde{M} \twoheadrightarrow M$ between smooth manifolds, such that for each point $x \in M$, there is a neighbourhood $U$ where each component of $\fibre{p}{U}$ maps diffeomorphically onto $U$. For a smooth and connected manifold $M$, there exists a smooth universal covering~\cite[Corollary 4.43]{LeeIntroductionManifolds} i.e. a smooth covering where $\tilde{M}$ is simply connected. A \emph{Riemannian} covering is smooth covering between Riemannian manifolds which is also a local isometry. Conversely, a local isometry smooth covering $p$ is a Riemannian covering if $\tilde{M}$ is complete and $M$ is connected~\cite[Prop 18.7]{Gallier2020DifferentialGroups}. Any smooth covering $p: \tilde{M} \twoheadrightarrow M$ of a Riemmanian manifold $M$ can be made Riemannian: if $g$ is a Riemannian metric on $M$, then there is a unique pullback metric $\tilde{g} = p^\ast g$ on $\tilde{M}$ such that $p$ is a Riemannian covering~\cite[Prop 18.5]{Gallier2020DifferentialGroups}.

From henceforth we restrict ourselves to the following setting: $\tilde{M}$ and $M$ are both connected, complete Riemannian manfiolds, and $p: (\tilde{M}, \tilde{g}) \twoheadrightarrow (M,g)$ is a Riemannian $\Gamma$-coverings. We summarise some facts about smooth $\Gamma$-coverings of manifolds in the lemma below:
\begin{lemma} \label{lem:manifold_covering_facts}
    Let $p: \tilde{M} \to M$ be a path-connected smooth $\Gamma$-covering. Then:
    \begin{itemize}
        \item The $\Gamma$-actions on $\tilde{M}$ are also smooth and proper~\cite[Prop 21.12]{LeeIntroductionManifolds};
        \item If $p$ is also Riemannian -- i.e. it is a local isometry --- then $\Gamma$-actions are isometries of $\tilde{M}$~\cite[P.27]{Lee2018IntroductionManifolds}~\cite[P.548]{Gallier2020DifferentialGroups}.
    \end{itemize}
\end{lemma}
\begin{example}\label{ex:surface_coverings_geom}
    The examples of universal coverings in~\Cref{ex:surface_coverings} are all Riemannian coverings. As the group actions on the covering spaces are all isometries, the base spaces inherit the local metric of the covering space. In the coverings of the torus and Klein bottle by $\R^2$, both the base space and the covering space have the flat metric; $\RP^2$ inherits the local metric of the sphere $ds^2 = d\theta^2 + \sin^2(\theta) d\varphi^2$; and the genus two surface $\Sigma_2$ inherits the hyperbolic metric $ds^2 = 4 \frac{dz d\bar{z}}{(1 - |z|^2)^2}$ of the Poincar\'e disk. In particular, since the covering spaces $\R^2$, $\Sbb^2$, and $\mathbb{D}$ have constant scalar and sectional  curvatures $0, +1, -1$ respectively, the base spaces inherit the same constant scalar curvatures. 
\end{example}

Furthermore, we recall the following facts about complete Riemannian manifolds. For $M$ be a complete connected Riemannian manifold, the Hopf-Rinow theorem implies the existence of minimising geodesics for any pair of points. Recall a geodesic $\gamma: [0, s] \to M$ with unit speed parametrisation is a \emph{minimising geodesic} if the length of the curve is the distance between the two points on the manifold $|\fun{\gamma_i}{[0,s]} | = s = \dmet{\gamma(0)}{\gamma(s)}$. When we lift paths in Riemannian coverings, geodesics lift to geodesics in $\tilde{M}$~\cite[Prop 2.81]{Gallot2004RiemannianGeometry}; and in particular, minimising geodesics lift to minimising geodesics.

\begin{restatable}{proposition}{LemQuotientGeometry}
\label{lem:quotient_geometry}
Consider a Riemannian $\Gamma$-covering of  path-connected complete manifolds $p: \tilde{M} \to M$. Then:
\begin{enumerate}[label=(\roman*)]
    \item The distance $\dm(x,y)$ between points on $M$ is given by the Hausdorff distance between the fibres of points in $\tilde{M}$; 
    \begin{equation}
        \dm(x,y) = \tilde{\dm}_H\left(\fibre{p}{x},  \fibre{p}{y}
\right)
    \end{equation}
    \item In particular, for any $\tilde{x} \in \fibre{p}{x}$ and $\tilde{y} \in \fibre{p}{y}$,
    \begin{equation}
     \tilde{\dm}_H\left(\fibre{p}{x},  \fibre{p}{y}
\right) = \tilde{\dm}(\tilde{x}, \fibre{p}{y}) := \inf_{g \in \Gamma} \tilde{\dm}\left(\tilde{x}, g \cdot \tilde{y}\right),  \label{eq:metric_on_base}
\end{equation}
and the infimum on the right hand side is attained by finitely many $g \in \Gamma$;
\item A minimising geodesic $\gamma$ from $x$ to $y$ lifts to a minimising geodesic $\tilde{\gamma}$ from $\tilde{x}$ to $g \cdot \tilde{y}$ with the same arc length iff $g$ attains the right hand side of \cref{eq:metric_on_base}. 
\item Consequently, if $\gamma$ is the \emph{unique} minimising geodesic, then $\tilde{\gamma}$ is also the unique minimising geodesic from $\tilde{x}$ to $g \cdot \tilde{y}$.
\end{enumerate}

\end{restatable}
\begin{proof}For simplicity of notation we let $[\tilde{x}] = \fibre{p}{x}$. We first show that $ \tilde{\dm}_H\left([\tilde{x}], [\tilde{y}]\right) = \tilde{\dm}(\tilde{x},[\tilde{y}])$. Recall the Hausdorff distance is defined by
\begin{equation*} \textstyle
    \tilde{\dm}_H\left([x], [y]\right) = \max\left\{\sup_{g \in \Gamma} \tilde{\dm}(g \cdot x, [y]),\sup_{g \in \Gamma} \tilde{\dm}([x],  g\cdot y)
     \right\}.
\end{equation*}
We take apart the definition term by term. Since $\Gamma$-actions are isometries (\Cref{lem:manifold_covering_facts}) and trivially transitive on orbits, we have $\sup_{g \in \Gamma} \tilde{\dm}(g \cdot \tilde{x}, [\tilde{y}]) = \tilde{\dm}(\tilde{x},  [\tilde{y}])$, as the assignment $g \mapsto \tilde{\dm}(g \cdot \tilde{x},  [\tilde{y}])$ is constant over $\Gamma$:
\begin{equation*}
   \textstyle  \tilde{\dm}(g \cdot \tilde{x},  [\tilde{y}]) = \tilde{\dm}(\tilde{x}, g^{-1} \cdot  [\tilde{y}]) = \tilde{\dm}(\tilde{x}, [\tilde{y}]).
\end{equation*}
Furthermore, the isometry condition also implies the two terms over which the max is taken are equivalent:
\begin{equation*}
   \textstyle  \tilde{\dm}(\tilde{x},[\tilde{y}]) = \inf_{g \in \Gamma} \tilde{\dm}\left(\tilde{x}, g \cdot \tilde{y}\right) = \inf_{g \in \Gamma} \tilde{\dm}\left(\inv{g} \cdot \tilde{x}, \tilde{y}\right) =  \inf_{h \in \Gamma} \tilde{\dm}\left(h \cdot \tilde{x}, \tilde{y}\right) = \tilde{\dm}([\tilde{x}],\tilde{y}). 
\end{equation*}
Thus, the Hausdorff distance between the orbits $[x]$ and $[y]$ are given by 
\begin{equation*}
  \textstyle   \tilde{\dm}_H\left([\tilde{x}], [\tilde{y}]\right) = \max\left\{\sup_{g \in \Gamma} \tilde{\dm}(g \cdot \tilde{x}, [\tilde{y}]),\sup_{g \in \Gamma} \tilde{\dm}([\tilde{x}],  g\cdot \tilde{y})
     \right\} = \max\left\{ \tilde{\dm}(\tilde{x}, [\tilde{y}]), \tilde{\dm}([\tilde{x}], \tilde{y})\right\} = \tilde{\dm}(\tilde{x}, [\tilde{y}]).
\end{equation*}
We now show that $L =  \inf_{g \in \Gamma} \tilde{\dm}\left(\tilde{x}, g \cdot \tilde{y}\right)$ is attained by finitely many $g \in \Gamma$. Recall the $\Gamma$-action being proper  (\Cref{lem:manifold_covering_facts}) implies for $K,K'$ compact subsets of $M$, the set of $g$ such that $(g \cdot K) \cap K' \neq \emptyset$ is compact. Since $\Gamma$ has the discrete topology, this implies that set of $g$ is finite. We now consider $r = \tilde{\dm}(\tilde{x},\tilde{y})$ and the closed $r$-ball centred at $\tilde{y}$. Since $\tilde{M}$ is complete this is a compact subset. By definition $r \geq L$, and the intersection of the orbit $[\tilde{x}]$ with the closed $r$-ball is non-empty and at least contains $\tilde{x}$. Now since the action is proper, there are finitely many elements of $g$ such that $g \cdot \tilde{x}$ is contained in the closed $r$-ball. Thus, the infimum of $\tilde{\dm}(g \cdot \tilde{x}, \tilde{y})$ is attained by some $g$ such that $g \cdot \tilde{x}$ is contained in the closed $r$-ball, and there can only be finitely many such $g$'s. 

We defer proving $\dm(x,y) = \inf_{g \in \Gamma} \tilde{\dm}(\tilde{x}, g \cdot \tilde{y}) = \tilde{d}_H([\tilde{x}], [\tilde{y}])$ to the end as we need the following observations about geodesics. Recall then that for Riemannian coverings, geodesics of the base space are projections of geodesics in the covering space, and geodesics in the covering space are lifting of geodesics in the base space~\cite[Proposition 18.6]{Gallier2020DifferentialGroups}. Moreover, since since $p$ is a local isometry, the arc lengths are preserved when the geodesics are projected and lifted. 

Because $M$ is complete, there is at least one geodesic $\gamma$ from $x$ to $y$ whose arc length attains $\dm(x,y)$ (that is the infimum of arc lengths over the geodesics between $x$ and $y$). When we lift any such geodesic to a geodesic $\tilde{\gamma}$ based at $\tilde{x}$, the lifted geodesic $\tilde{\gamma}$ must also have arc length $\dm(x,y)$, which must be the least out of all geodesics between not only $\tilde{x}$ and $\tilde{y}$, but also any geodesic between $\tilde{x}$ and $g \cdot \tilde{y}$, as these geodesics also project to geodesics between $x$ and $y$, and must thus have arc length at least $\dm(x,y)$ (due to local isometry preserving arc lengths). Thus, $\tilde{\gamma}$ is a geodesic attaining $\inf_{g \in \Gamma} \tilde{\dm}(\tilde{x}, g \cdot \tilde{y})$. Because $\tilde{\gamma}$ has arc length $\dm(x,y)$, we thus have $\dm(x,y) = \inf_{g \in \Gamma} \tilde{\dm}(\tilde{x}, g \cdot \tilde{y}) = \tilde{d}_H([\tilde{x}], [\tilde{y}])$. If $\gamma$ is the unique minimising geodesic between $x$ and $y$, then any minimising geodesic between $\tilde{x}$ and $g \cdot \tilde{y}$ other than $\tilde{\gamma}$ must have greater arc length than $\tilde{\gamma}$, thus $\tilde{\gamma}$ is the unique minimising geodesic between $\tilde{x}$ and $g \cdot \tilde{y}$.

Conversely, suppose $g_\ast$ attains $L = \inf_{g \in \Gamma} \tilde{\dm}(\tilde{x}, g \cdot \tilde{y})$, and let $\tilde{\gamma}$ be the minimising geodesic from $\tilde{x}$ to $g_\ast \cdot \tilde{y}$ with arc length $L$. Because we have established that $L = \tilde{\dm}(\bar{x}, [\bar{y}]) = \tilde{\dm}([\bar{x}], \bar{y})$ for \emph{any} $\bar{x} \in \fibre{p}{x}$ and $\bar{y} \in \fibre{p}{y}$, any geodesic between any $\bar{x}$ and $\bar{y}$ must have arc length at least $L$. Because $p$ is a local isometry, any geodesic between $x$ and $y$ must have length at least $L$. Thus the geodesic $p \circ \tilde{\gamma}$ is a minimising geodesic from $x$ to $y$. Therefore, we have a minimising geodesic $\gamma = p\circ \tilde{\gamma}$ that lifts to $\tilde{\gamma}$.

\end{proof}

\begin{corollary}\label{prop:geodesic_groupoid_assignment} Let $p: \tilde{M} \to M$ be a Riemannian $\Gamma$-covering of path-connected complete manifolds. Given any pseudo-section $\lambda: M \to \tilde{M}$, the monodromy homomorphism $\mu_\lambda: \Fgpd{M} \to \Gamma$ on minimising geodesics $\gamma$ from $x$ to $y$ satisfies 
\begin{equation} \label{eq:geodesic_groupoid_assignment}
    \dm(x,y) = \tilde{\dm}(\lambda(x), \mu_\lambda([\gamma]) \cdot \lambda(y)). 
\end{equation}
If all the minimising geodesics belong to a unique class $[\gamma]$, then $\mu_\lambda([\gamma])$ is the only element of $\Gamma$ satisfying \cref{eq:geodesic_groupoid_assignment}.
\end{corollary}
\begin{proof}
    By \Cref{lem:quotient_geometry}, every minimising geodesic $\gamma$ from $x$ to $y$ lifts to a unique minimising geodesic $\tilde{\gamma}$ from $
    \lambda(x)$ to $g \cdot \lambda(y)$ for some $g \in \Gamma$, iff they satisfying $\dm(x,y) = \tilde{\dm}(\lambda(x), g \cdot \lambda(y))$. From the definition of the monodromy homomorphism \Cref{def:monodromy}, this $g$ must be $\mu_\lambda([\gamma])$. Thus, if every minimising geodesic between $x$ and $y$ belong to the same homotopy class , then $\mu_\lambda([\gamma])$ is the unique element of $\Gamma$ satisfying \cref{eq:geodesic_groupoid_assignment}.
\end{proof}

In the context of induced coverings along maps $f: X \to M$, \Cref{prop:geodesic_groupoid_assignment} implies we can construct a transition homomorphism for the induced covering by solving a minimisation problem (\cref{eq:t_infer}), given the assumptions below.

\begin{proposition} \label{prop:basic_infer_t}
    Consider a Riemannian $\Gamma$-covering of  path-connected complete manifolds $p: \tilde{M} \to M$. Let $f: X \to M$ be a continuous map and $\cU = \sett{U_i}_{i \in \cI}$ be a good cover of $X$. Assume:
    \begin{enumerate}[label=(S\arabic*), ref=(S\arabic*)]
        \item \label{item:S1} The data $(\cU, f)$ admit a realisation functor $R: \Egpd{\nerve{\cU}} \xrightarrow{\simeq} \Fgpd{X}$, such that $\Fgpd{f}\circ R([ij]) \in \Fgpd{M}(y_i, y_j)$ is the unique homotopy class of paths containing a minimising geodesic between $y_i = f \circ R(i)$ and  $y_j = f \circ R(j)$.
    \end{enumerate}
    Then a transition homomorphism $t: \Egpd{\nerve{\cU}} \to \Gamma$ of the induced covering on $X$ is prescribed, by the following assignment to edges $ij$:
    \begin{equation} \label{eq:t_infer}
        t(ij) = \arg \inf_{g \in \Gamma} \tilde{\dm}\left(\tilde{y}_i, g \cdot \tilde{y}_j\right).
    \end{equation}
\end{proposition}
\begin{proof}
    Recall \Cref{prop:monodromy_groupoid_big}, which allows us to infer a transition homomorphism $t: \Egpd{\nerve{\cU}} \to \Gamma$ for the induced covering from a monodromy homomorphism $\mu_\lambda: \Fgpd{X} \to \Gamma$; this is given by $t = \mu_\lambda \circ \Fgpd{f} \circ R$. If \labelcref{item:S1} holds, then for every $ij \in \Egpd{\nerve{\cU}}$, we apply \Cref{prop:geodesic_groupoid_assignment} to obtain 
    \begin{equation}
        t(ij) = \mu_\lambda \circ \Fgpd{f} \circ R([ij]) = \inf_{g \in \Gamma} \tilde{\dm}\left(\lambda(y_i), g \cdot \lambda(y_j)\right).
    \end{equation}
\end{proof}
While the assumption~\labelcref{item:S1} might seem unnatural, we discuss two constructions in the subsection below which satisfies it. 

\subsection{Inference of Ambient Cycles}
We explore the consequences of \Cref{prop:basic_infer_t} to inferring whether `cycles' of a point cloud $\X \subset M$ correspond to underlying cycles of $M$. As usual by cycles of a point cloud, we mean the cycles of some simplicial complex $K$ built on top of the point cloud, in a way that reflects its geometry within $M$. 

The first simplicial complex we consider is the \v Cech complex, a standard construction in topological data analysis. Given a scale parameter $\epsilon > 0$, the $\epsilon$-\v Cech complex $\mathrm{\check Cech}_\epsilon(\X)$ is the nerve of $\cU_\epsilon = \settt{\openball{\epsilon}{x}}{x \in \X}$ which covers the $\epsilon$-thickening of the point set $\X$. If $M$ is a Euclidean space, then for any $\epsilon$, the cover $\cU_\epsilon$ is a good cover of $\X^\epsilon$, and thus we can leverage the nerve lemma use the homotopy equivalence between $\X^\epsilon$ and $\nerve{\cU_\epsilon}$ to work with the simplicial complex $\nerve{\cU_\epsilon}$ as a topological and geometric descriptor of $\X^\epsilon$. However, relaxing the ambient space to be any Riemannian manifold,  even individual geodesic balls $\openball{\epsilon}{x}$ might not be contractible for some sufficiently large $\epsilon$, so $\cU_\epsilon$ ceases to be a good cover of $\X^\epsilon$. This illustrates the difficulty in using the \v Cech complex, or indeed its combinatorially simpler cousin the Vietoris-Rips complex, to model the shape of a point cloud in $M$ at large spatial scales. 

Nonetheless, if we limit our thickening scale $\epsilon$ to be less than the \emph{convexity radius} $\conv{M}$of the manifold~\cite{Carmo1992RiemannianGeometry}, then we can still guarantee $\cU_\epsilon$ to be a good cover of $\X^\epsilon$. Recall the convexity radius $\conv{p}$ of a point $p \in M$ is the supremum over all radii of geodesic balls $\openball{r}{p}$, such that $\openball{r}{p}$ is strongly geodesically convex. We also note that $\conv{p} > 0$ for any $p$~\cite{Whitehead1962CONVEXPATHS}. The convexity radius of the manifold is then the infimum over the convexity radius of all points in the manifold. Consider then the induced cover of $\X^\epsilon$ by the inclusion into $M$ and the Riemannian universal covering of $M$.  We show that for $\epsilon < \conv{M}$, we can recover the induced covering on $\X^\epsilon$ by solving the minimisation problem~\cref{eq:t_infer} described in~\Cref{prop:basic_infer_t}, as we can construct a realisation functor $R: \Egpd{\nerve{\cU_\epsilon}} \xrightarrow{\simeq} \Fgpd{\X^\epsilon}$ 
 that satisfies \labelcref{item:S1}. This is the content of \Cref{prop:pointcloud_convx_radius}. Using the transition homomorphism of the induced covering thus obtained,  we can infer whether the point cloud `captures' any ambient cycles at scale $\epsilon < \conv{M}$.

\begin{proposition} \label{prop:pointcloud_convx_radius}
    Consider a Riemannian $\Gamma$-covering of  path-connected complete manifolds $p: \tilde{M} \to M$. Let $\X \subset M$ be a finite point set and consider $\X^\epsilon$. Suppose $\epsilon < \conv{M}$; then:
    \begin{enumerate}[label = (\roman*)]
        \item $\cU =\settt{\openball{\epsilon}{x}}{x \in \X} $ is a good cover of $\X^\epsilon$~\cite[Prop A.1]{Fiorenza2012CechConstruction};
        \item Furthermore, we can infer a transition homomorphism $t: \Egpd{\mathrm{\check  Cech}_\epsilon(\X)} \to \Gamma$ of the induced covering on $\X^\epsilon$ in the manner of~\cref{eq:t_infer}, where for edge $ij$ in $\mathrm{\check  Cech}_\epsilon(\X)$,
    \begin{equation}
         t(ij) = \arg \inf_{g \in \Gamma} \tilde{\dm}\left(\tilde{x}_i, g \cdot \tilde{x}_j\right),
    \end{equation}
    for arbitrary choices $\tilde{x} \in \fibre{p}{x}$ of representatives of $x \in \X$ in their fibres. 
    \end{enumerate}
\end{proposition}

\begin{proof}
    The first statement is proved in~\cite[Prop A.1]{Fiorenza2012CechConstruction}.  We then specify a functor $R: \Egpd{\mathrm{\check Cech}_\epsilon(\X)} \to \Fgpd{\X^\epsilon}$. On objects of $\Egpd{\mathrm{\check Cech}_\epsilon(\X)}$, we set $R$ to send each vertex of $\mathrm{\check Cech}_\epsilon(\X)$, corresponding to cover element $\openball{\epsilon}{x}$, to $x \in \X \subset M$. Consider $x_i$ and $x_j$ in $\X$ where $\openball{\epsilon}{x_i} \cap \openball{\epsilon}{x_j} \neq \emptyset$. Because $\{\openball{\epsilon}{x_i},\openball{\epsilon}{x_j}\}$ is a good cover of $\openball{\epsilon}{x_i} \cup \openball{\epsilon}{x_j}$, the nerve lemma implies $\openball{\epsilon}{x_i} \cup \openball{\epsilon}{x_j}$ is simply connected, and thus paths between $x_i$ and $x_j$ are homotopy equivalent. This is also the class containing any minimising geodesics between $x_i$ and $x_j$; we do so by showing that such minimising geodesics must be contained in $\openball{\epsilon}{x_i} \cup \openball{\epsilon}{x_j}$. Suppose a minimising geodesic is not fully contained in $\openball{\epsilon}{x_i} \cup \openball{\epsilon}{x_j}$; then there is a point along the geodesic where the distance to each $x_i$ and $x_j$ is at least $\epsilon$. Since the geodesic is minimising this implies the distance between $x_i$ and $x_j$ is at least $2\epsilon$, which contradicts the fact that $\openball{\epsilon}{x_i} \cap \openball{\epsilon}{x_j}\neq \emptyset$. Thus, any minimising geodesic between $x_i$ and $x_j$ represents the unique class of paths between $x_i$ and $x_j$ in $\openball{\epsilon}{x_i} \cup \openball{\epsilon}{x_j}$. 
    
    We have thus constructed a realisation functor $R$ to satisfy~\labelcref{item:S1} of~\Cref{prop:basic_infer_t}. Applying~\Cref{prop:basic_infer_t} completes the proof. 
\end{proof}
 
In an alternative approach to compensate for the lack of good covers for $\X^\epsilon$ for large $\epsilon$, we can simplify the problem by only considering the one-skeleton of the complex, an associated edges in the one-skeleton with a minimising geodesic between the points represented by the vertices of the complex. We define the construction formally below. 

\begin{definition} \label{def:min_geodesic} Consider a finite simple graph $G$ along with a map $f: |G| \to M$ from its geometric realisation to a complete Riemannian manifold $M$. We say $f$ is a \emph{min-geodesic map} if $f$ is injective on the vertices, and for an edge $uv \in G$, $f\rvert_{uv} : I \to M$ is a reparametrisation of a minimising geodesic between $f(u)$ and $f(v)$. Collective, we say $(G,f)$ is a min-geodesic graph
\end{definition}

Given a graph $G$ built on a point cloud as its vertex set, the associated min-geodesic graph interpolates between pairs of points  by minimising geodesics wherever the pair of points are joined by an edge in $G$. 

If we choose vertex stars as a cover of $|G|$, then a min-geodesic graph satisfies~\labelcref{item:S1} in~\Cref{prop:basic_infer_t} by construction. Thus, we can once again reduce the problem of inferring the induced covering on $|G|$ to solving the minimisation problem~\cref{eq:t_infer}.

\begin{proposition} \label{prop:graph_convx_radius}
    Consider a Riemannian $\Gamma$-covering of  path-connected complete manifolds $p: \tilde{M} \to M$. Consider a graph $G$ equipped with a min-geodesic embedding $f: |G| \to M$. Then we can infer a transition homomorphism $t: \Egpd{G} \to \Gamma$ of the induced $\Gamma$-covering on $G$ in the manner of~\cref{eq:t_infer}, where for any edge $uv$ in $G$,
    \begin{equation*}
         t(uv) = \arg \inf_{g \in \Gamma} \tilde{\dm}\left(\tilde{u}, g \cdot \tilde{v}\right),
    \end{equation*}
    for arbitrary choices $\tilde{u} \in \fibre{p}{f(u)}$  and $\tilde{v} \in \fibre{p}{f(v)}$  of representatives of vertices in their fibres.  Conversely, any groupoid homomorphism $t: \Egpd{G} \to \Gamma$ where $t(uv) \in  \arg \inf_{g \in \Gamma} \tilde{\dm}\left(\tilde{u}, g \cdot \tilde{v}\right)$ is a transition homomorphism for some min-geodesic embedding.
\end{proposition}
\begin{proof}
    Recall the cover of $|G|$ by open vertex stars $\cU = \settt{|\opstar{u}|}{ u \in V(G)}$ is a good cover of $|G|$. Thus we consider the edge groupoid on $\nerve{\cU} = G$. Since we have assumed every edge is mapped by $f$ to the unique minimising geodesic between the image of its boundary vertices, we trivially satisfy~\labelcref{item:S1} in~\Cref{prop:basic_infer_t}. Thus,~\Cref{prop:basic_infer_t} implies the result of this proposition.

    For the converse, we recall~\Cref{lem:quotient_geometry}(iii); for any edge $uv \in G$, if $t(u_i u_j)$ is a solution to the minimisation problem $\inf_{g \in \Gamma} \tilde{\dm}\left(\tilde{u}_i, g \cdot \tilde{u}_j\right)$, then there is a minimising geodesic $\gamma$ from $f(u_i)$ to $f(u_j)$ that lifts to a minimising geodesic $\tilde{\gamma}$ from $\tilde{u}_i$ to $t(u_i u_j)\cdot \tilde{u}_j$. Let $f$ be the min-geodesic embedding of $G$ that is compatible with such minimising geodesics. Then, $t(u_iu_j)$ is the evaluation of a monodromy homomorphism $\mu_\lambda: \Fgpd{M} \to \Gamma$ on $f\rvert_{ij}: I \hookrightarrow |G| \to M$, for a choice of pseudo-section $\lambda$ where $\lambda(f(u_i)) = \tilde{u}_i$. We can then apply~\Cref{prop:monodromy_groupoid_big} which states that the monodromy data $t(u_iu_j)$ recovers a transition homomorphism $t: \Egpd{G} \to \Gamma$ of the induced $\Gamma$-covering on $G$.
\end{proof}
 Naturally, minimising geodesics may not be unique between all pairs of points in the image of the vertices of $G$, but we show in~\Cref{prop:NC_generic} and~\Cref{cor:UMG_generic} of the following section that for a finite point cloud $\X \subset M$, there is an open dense subset of the  $\# \X$-point configuration space, such that any pair of points in $\X$ has a unique minimising geodesic.

\subsection{Genericity of Unique Minimising Geodesics}\label{sec:generic}
We first show that finite point clouds on manifolds are generically equipped with unique minimising geodesics. We consider the unordered configuration space $\uconf{n}{M}$ of $n$ distinct points on a manifold $M$; that is the configuration space of $n$ ordered points, 
\begin{equation*}
    \conf{n}{M} = \settt{(x_0, \ldots, x_{n-1}) \in M^n}{x_i \neq x_j,\ \text{if } i \neq j},
\end{equation*}
modulo the action of the symmetry group $S_n$ on the points. The Hausdorff distance between subsets of $M$ descends to a metric between point sets in $\uconf{n}{M}$.

We show that with respect to this metric on $\uconf{n}{M}$, the set of point clouds in $\uconf{n}{M}$ that are equipped with unique minimising geodesics contain an open dense subset of $\uconf{n}{M}$. These are point sets where no point in the set is in the \emph{cut locus} another; we call them $\emph{NC}$ point sets. Any $NC$ point set is equipped with unique minimising geodesics between each pair of points~\cite[Ch. 13, cor. 28]{Carmo1992RiemannianGeometry}. We show that NC point sets are open and dense in $\uconf{n}{M}$. 

We first review the notion of the \emph{cut locus} for a complete Riemannian manfiold $M$ equipped with a Levi-Civita connection. Fixing a point $p$, consider $v \in S_pM \subset T_pM$, where $S_pM$ is the unit sphere in the tangent space; we define $\rho_p: S_pM \to [0,\infty]$ to be the function 
\begin{equation}
    \rho_p(v) = \sup \settt{t > 0}{\dm(p,\exp_p(tv)) = t}.
\end{equation}
That is, $\rho(v)$ measures the maximal extent along which the geodesic $\gamma$ (satisfying $\gamma(0) = p$ and $\gamma'(0) = v$)  is minimal. The cut locus $L_p \subset M$ of the point $p$ is defined to be the subset 
\begin{equation}
    L_p = \settt{\exp_p(\rho(v)v)}{v \in S_pM,\ \rho(v) < \infty }.
\end{equation}
The cut locus $L_p$ of any point is a closed, nowhere dense, measure zero subset of $M$~\cite{Grove1993CriticalFunctions}, and $L_p$ is the boundary of $M \setminus L_p$~\cite[III Lemma 4.4]{Carmo1992RiemannianGeometry}, with Hausdorff dimension at most $(d-1)$, where $d$ is the dimension of $M$~\cite{Mantegazza2002HamiltonJacobiManifolds}. We note that every point $q$ in $M \setminus L_p$ has a unique minimal geodesic $\gamma$ from $p$ to $q$; if there is another geodesic of the same length $|\gamma| = s$ between them, then the extension $\exp_p(tv)$ of $\gamma$ for $t \in [0,s +\epsilon]$ is no longer minimal for any $\epsilon > 0$~\cite[Prop 16.18]{Gallier2020DifferentialGroups}, thus placing $q = \exp_p(sv)$ in the cut locus.

Recently,~\cite{Eltzner2021StabilityManifolds} showed that the cut locus varies in a stable way w.r.t. $p$ for compact manifolds. 
\begin{proposition}[{\cite[Cor 3.8, Prop 3.11]{Eltzner2021StabilityManifolds}}] \label{prop:cutlocus_stability} Let $M$ be a compact Riemannian manifold. Then for any $p \in M$, and any $\epsilon > 0$, there exists $\delta > 0$ s.t.
\begin{equation}
    \bigcup_{q: \dm(p,q) < \delta} L_q \subseteq (L_p)^\epsilon := \settt{x \in M}{\dm(L_p, x) < \epsilon}.
\end{equation}
\end{proposition}
That is, an $\epsilon$-thickening of $L_p$ contains the cut locus of any $\delta$-perturbed point away from $p$ for $\delta$ sufficiently small. 

We now turn our attention to finite point clouds on compact manifolds $M$. Our aim is to show that the following desirable property is generic in $\uconf{n}{M}$; in other words, we wish the property to holds true on an open dense subset w.r.t. to the topology induced by the Hausdorff distance. 

\begin{definition}
    We say a finite set of $n$ distinct points $\X \subset M$ on a complete Riemannian manfiold is $\emph{NC}$ if for any pair of points in $\X$
    \begin{equation}
        q \notin  L_p, \qtextq{and}  p \notin L_q.
    \end{equation}
\end{definition}
Since $q \in L_p$ iff $p \in L_q$~\cite[Ch. 13, cor. 2.7]{Carmo1992RiemannianGeometry}, we can apply an inductive argument to show that it is NC holds true on a dense subset. 

\begin{proposition}
     NC point sets in $\uconf{n}{M}$ are dense on a complete Riemannian manifold for all $n \in \mathbb{N}$.
\end{proposition}

\begin{proof}
    We prove this inductively. Throughout this proof for $\X_n = \{p_0, \ldots, p_{n-1}\} \in \uconf{n}{M}$, we let $L_i = L_{p_i}$ denote the cut locus of point $p_i$.

    We first show that NC is dense for $n=2$. For any given $p_0$, $M \setminus L_0$ is open (since $L_0$ is closed), and dense (since $L_0$ is the boundary of $M\setminus L_0$). Thus, for any two point configuration $\X_2 = \sett{p_0, p_1}$, we can find $p_1' \neq p_0$ that is arbitrarily close to $p_1$  while $\X_2' = \{p_0, p_1'\}$ is NC. Thus, NC pairs are dense in $M_2$: for any $\epsilon>0$ and $\X_2 \in M_2$, there is some NC pair of points $\X'_2 \in M_2$ such that $d_H(\X_2', \X_2) < \epsilon$.

    Having established the base case, we prove the inductive step. Suppose for all $n \leq k$, the set of NC point clouds is dense in $\uconf{n}{M}$. Now consider $\X_{k+1} = \{x_0, \ldots, x_k\} \in M_{k+1}$, and let us denote $\X_{k} = \{x_0, \ldots, x_{k-1}\}$ for some arbitrary choice of $k$ points from $\X_{k+1}$. For any $\epsilon > 0$, we can find $\X_k'$ that is $\epsilon$-close to $\X_k$ and NC. Suppose $x_k = \X_{k+1} \setminus \X_k$ is in some of the cut loci of the other points in $\X_k'$, i.e. $x_k \in \cup_{i=0}^{k-1} L_i$. Since $L_i$ are nowhere dense, the finite union $\cup_{i=0}^{k-1} L_i$ is also nowhere dense, and thus we can find a point $x_k'$ that is $\epsilon$-close to $x_k$, such that $x_k' \notin \cup_{i=0}^{k-1} L_i'$. Since $q \in L_p$ iff $p \in L_q$~\cite[Ch. 13, cor. 2.8]{Carmo1992RiemannianGeometry}, none of the points $\X_k'$ are in the cut locus of $L_n'$. We can furthermore choose $x_k'$ such that it is distinct from points in $\X_k'$ as the latter is only a finite point set. Thus $\X_{k+1}' = \X_k' \cup \{x_k'\}$ is an NC point set in $M_{k+1}$. Since these are finite point sets in $M$ and thus are closed subsets, $\dm_H(\X'_k, \X_k) < \epsilon$ iff the $\epsilon$-thickening $(\X_k)^\epsilon$ of $\X_k$ contains $\X_k'$, and vice versa. Because $\dm(x_k, x_k')< \epsilon$, we thus have $(\X_{k+1})^\epsilon \supseteq \X_{k+1}'$ and $(\X_{k+1}')^\epsilon \supseteq \X_{k+1}$, i.e. we have found for any $\X_{k+1} \in M_{k+1}$, and every $\epsilon > 0$, an NC point set $\X_{k+1}'$ that is $\epsilon$-close to $\X_{k+1}$ in Hausdorff distance. In other words, the set of NC point sets is dense in $M_{k+1}$. We have thus shown inductively that NC point sets are dense in $\uconf{n}{M}$ for all finite $n$.
\end{proof}

To show that NC point sets are open in $\uconf{n}{M}$, we require the stability result of \Cref{prop:cutlocus_stability}.

\begin{proposition} \label{prop:NC_generic}
    For a compact Riemannian manifold and all $n \in \mathbb{N}$, NC point sets form an open dense subset of $\uconf{n}{M}$.
\end{proposition}

\begin{proof}
    As we have shown that NC points are dense it only remains to show that they are open in $\uconf{n}{M}$. 

    Consider an NC point set $\X_n = \{x_0, \ldots, x_{n-1}\}$. In other words, $x_i \notin L_j$ for any $i,j$. Because $L_j$ are nowhere dense in $M$, for some sufficiently small $\epsilon >0$, the thickening $\cup_{i = 0}^{n-1} (L_i)^{2\epsilon}$ does not intersect $\X$. Because $M$ is compact Riemannian, we can now invoke \Cref{prop:cutlocus_stability} and choose the smallest $\delta \in (0, \epsilon)$ such that for all $x_i$, and $p_i \in B_\delta(x_i)$, the cut locus $L_p$ is contained within $(L_i)^\epsilon$. Now that implies for any choice of $x_i' \in B_\delta(x_i)$, the finite point set $\X' = \{x_0', \ldots, x_{n-1}'\} \in M_{n}$ does not intersect any of $L_{p_i}$, and thus $\X'$ is also NC. 
    
    If we choose $0 < 2\delta < \min_{i\neq j} \dm(x_i, x_j)$ (the right hand side is positive since minimum taken over a finite set of distinct points), then $\dm(x_i, x_i') < \delta$ for all $i = 0, \ldots, n-1$ iff $\dm_H(\X,\X') < \delta$. Thus, for any NC point set $\X$, there is an open $\delta$-ball around $\X$ in $\uconf{n}{M}$, such that all point sets in the ball are also NC. Thus, the set of NC point sets are open in $\uconf{n}{M}$.
\end{proof}

Because pairs of points that are not in mutual cut loci are joined by a unique minimising geodesic~\cite[Ch. 13, cor. 2.8]{Carmo1992RiemannianGeometry}, we obtain the following observation. 

\begin{restatable}{corollary}{UMG_generic}
\label{cor:UMG_generic}
Let $M$ be a compact Riemannian manifold. For all $n \in \mathbb{N}$, there is an open dense subset of $\uconf{n}{M}$ in which every point cloud $\X $ is equipped with unique minimising geodesics. 
\end{restatable}

\begin{remark}
    While \Cref{cor:UMG_generic} implies point sets that are equipped with unique minimising geodesics are dense in $\uconf{n}{M}$, they may not form an open subset. That is because the cut locus of $p\in M$ can contain points $q$ that have a unique geodesic to $p$ (conjugate points); thus there can be point clouds that are equipped with unique minimising geodesics that are not NC. Furthermore, within any neighbourhood of $q$, there can be points on the cut locus of $p$ with more than one minimising geodesic to $p$. 
\end{remark}

\subsection{Application: Analysing Principal Persistence Measures}\label{subsec:4pt}

We consider a setting where we have an empirical sample of points $\X$ on a compact Riemannian manifold $M$. The principal Vietoris Rips persistence set of $\X$, for first homology, is a collection of points on a persistence diagram, where each point corresponds to a 1-cycle with non-trivial persistence formed over four points in $\X$. We show how we can associate a min-geodesic graph (\Cref{def:min_geodesic}) to each of these points, and doing so infer the homology class of each point in $M$ using the transition homomorphism induced by the universal cover. We demonstrate this on empirical uniform samples on compact surfaces, and show across several examples that cycles belonging to different classes can have different support on the persistence diagram.

\subsubsection{Characterising one-cycles on four points}

We recall the $\epsilon$-Vietoris-Rips complex of a point cloud  $\X$ in a metric metric space $X$ is a flag complex on the $\epsilon$-neighbourhood graph of $\X$. We refer to the persistent homology of the filtration of the $\epsilon$-Vietoris-Rips complex over a range of $\epsilon$ parameters as the Vietoris-Rips persistent homology.  In~\cite{Gomez2024CurvatureDiagrams}, they show that the $k$-dimensional Vietoris-Rips persistent homology of a $\X$ with $n$-points is trivial if $n < 2(k+1)$; if $n = 2(k+1)$, there can be at most one non-trivial generator of the persistent homology. Thus, we can send an $n$-tuple of points in $X$ (where $n = 2k+2$) to either the diagonal of the persistence diagram (when the persistent homology is trivial), or  a point $(b,d)$ the persistence diagram representing the persitence of the only cycle it generates. In~\cite{Gomez2024CurvatureDiagrams}, they define the \emph{Vietoris-Rips principal persistence set} of a metric space $X$ to be the set of points on the persistence diagram corresponding to the $k$-dimensional persistent homology generator of some set of $(2k+2)$-points on $X$. If $X$ admits a measure $\mu$, then this map from $X^n$ to persistence diagrams pushes forward the product measure $\mu^{\otimes n}$ on $X^n$ to a persistence measure on the diagram. This is the \emph{Vietoris-Rips principal persistence measure} of a metric measure space $(X, \mu)$. 

Focussing on the case where we are interested in one-dimensional homology, the empirical first Vietoris-Rips principal persistence measure of a point cloud $\X$ can be estimated by uniformly sampling $4$ points on $\X$, and computing the birth and death $[b,d)$ of its unique first Vietoris-Rips persistent homology generator, if any. The following lemma which is a special case of~\cite[Theorem 4.4, Ex. 4.5]{Gomez2024CurvatureDiagrams} summarises the geometric conditions where we obtain a non-trivial generator of first Vietoris-Rips persistent homology.

\begin{lemma}[{\cite[Theorem 4.4, Ex. 4.5]{Gomez2024CurvatureDiagrams}}]\label{lem:4ptcycle_rips} A Vietoris-Rips complex on four points from a metric space $(X, \dm)$ has at most one 1-cycle; furthermore, such a 1-cycle has non-trivial persistence $[b,d)$, if and only if we have a labelling of vertices $x_0, \ldots , x_3$ such that 
\begin{equation}
    \dm(x_0,x_1) \geq \dm(x_2,x_3) = d > b = \max_{(i,j) \in \{0,1\} \times \{2, 3\}}\dm(x_i, x_j). \label{eq:4ptcycle_rips}
\end{equation}
\end{lemma}

In other words, we see a non-trivial generator of one-dimensional Vietoris-Rips persistence when the $\epsilon$-neighbourhood graph $G_\epsilon$ on the four points is a cycle graph, for precisely the scales $\epsilon \in [b,d)$. For $x_0, \ldots, x_3 \subset M$ in complete Riemannian manifold with non-trivial dimension one persistence, we can explicitly relate the point on the persistence diagram to a first homology class or a conjugacy class of the fundamental group of $M$ using a min-geodesic map $f: |G_\epsilon| \to M$ (see~\Cref{def:min_geodesic}). We show that the induced homomorphism of $f$ on homology or fundamental group is no different to that induced by the inclusion of its image $f(|G|)$ into $M$: for any min-geodesic map $f$ associated to four points with non-trivial persistence, $f$ is  an embedding of $S^1 \simeq |G|$ into $M$. To facilitate our proof we adopt the following terminology.

\begin{definition} Let $\{x_1, \ldots, x_n\} = \X$ be a finite point set in a metric space $(X, \dm)$ with $n \geq 3$. Consider a cycle graph $G$ with $\X$ as its vertex set. We say $G$ is a \emph{minimal $\epsilon$-cycle} when $\dm(x_i, x_j) \leq \epsilon$ iff $(i,j) \in G$. 
\end{definition}

To show that $f$ is an embedding of $|G| \simeq S^1$, we show that no pair of edges incident on the same point can have intersecting images, unless their images are contained in one another. 

\begin{lemma}[Intersection Lemma] \label{lem:intersection} Let $\gamma: [0,s] \to M$ and $\tilde{\gamma}: [0,\tilde{s}] \to M$ be minimising geodesics parametrised by arc length. Suppose they have the same base point $\gamma(0) = \tilde{\gamma}(0) = x$ and denote their respective end points by $\gamma(s) = y$ and $\tilde{\gamma}(\tilde{s}) = \tilde{y}$. Suppose w.l.o.g. that $s \leq \tilde{s}$. Then either $\gamma \subseteq \tilde{\gamma}$ or  $\gamma \cap \tilde{\gamma} \subseteq \sett{x,y}$, where in the latter case we have $y = \tilde{y}$.
\end{lemma}

\begin{proof}
    If $\dot{\gamma}\rvert_0 = \dot{\tilde{\gamma}}\rvert_0$, then the geodesics $\gamma$ and $\tilde{\gamma}$ are solutions to the geodesic equation with identical initial conditions; thus if we assume $s \leq \tilde{s}$, we have $\gamma \subseteq \tilde{\gamma}$. 

    For the case $\dot{\gamma}\rvert_0 \neq \dot{\tilde{\gamma}}\rvert_0$, assume there is some point $p \in \gamma \cap \tilde{\gamma}$ in the intersection that is distinct from $x$. Let $\gamma(l) = p = \tilde{\gamma}(\tilde{l})$. Since $\gamma$ and $\tilde{\gamma}$ are minimising geodesics, the restrictions of them to segments ending at $x$ and $p$ yield two minimising geodesics between $x$ and $p$, and thus $l = \tilde{l} = \dm(x, p) \leq s \leq \tilde{s}$. Let us consider the following piecewise geodesic curve from $x$ to $\tilde{y}$, where the segment between $x$ and $p$ in $\tilde{\gamma}$ is replaced by the corresponding segment in ${\gamma}$:
\begin{equation*}
    \sigma(t) = 
    \begin{cases}
    {\gamma}(t) & t \in [0, l] \\
    \tilde{\gamma}(t) & t \in [l, \tilde{s}]
    \end{cases}.
\end{equation*}
 As $\sigma$ is composed of two pieces of minimising geodesics with arc length parametrisation,  it has length $|\sigma| = \tilde{s} = \dm(x, \tilde{y})$. If $l < \tilde{s}$, and $\dot{\gamma}\rvert_l \neq \dot{\tilde{\gamma}}\rvert_l$ --- i.e. either $\gamma$ and $\tilde{\gamma}$ intersect transversely or are anti-parallel at $p$ --- the velocity $\dot{\sigma}$ is discontinuous at $p$. Thus $\sigma$ is not geodesic, and therefore cannot be a distance minimising curve. In other words, $|\sigma| > \dm(x,\tilde{y}) = \tilde{s}$. We thus reach a contradiction. Thus  $\dot{\gamma}\rvert_l = \dot{\tilde{\gamma}}\rvert_l$, or we conclude that $l = \tilde{s}$. In the former case, the uniqueness of solutions of the geodesic equation with specified initial position and velocity implies $\gamma \subseteq \tilde{\gamma}$. In the latter case where $l = \tilde{s}$, the fact that $l \leq s \leq \tilde{s}$  implies $l = s = \tilde{s}$. Thus we conclude that $p = y = \tilde{y}$.
\end{proof}

\begin{proposition}\label{prop:cycle_embed}
    Let $G$ be a minimal $\epsilon$-cycle and $f: |G| \to M$ be a minimal geodesic map. The map $f$ is not an embedding iff $\# V(G) = 3$ and one of the three vertices lie on the minimising geodesic between the other two.
\end{proposition}

\begin{proof}
Since $f$ is a continuous map from a compact space to a Hausdorff space, it is a closed map.  In addition, if $f$ is injective, then $f$ is an embedding. We now show that this is the case. 

The map $f$ is not injective if the geodesic segments of $f$ intersect off any shared end points. Let $\gamma_i$ denote the minimising geodesic from $x_i$ to $x_{i+1}$ in the image of $f$. Suppose for some $i \neq j$, we have $p \in \imag \gamma_i \cap \imag \gamma_j$ for some choice of minimising geodesics $\gamma_i$ and $\gamma_j$. Let $s_k = \dmet{x_k}{x_{k+1}}$ and define $\mu_k = \frac{\dmet{x_k}{p}}{s_k}$ for $k = i,j$. Since $p$ lies on a minimising geodesic between $x_k$ and $x_{k+1}$, we have $\mu_k \in [0,1]$. W.l.o.g. index the indices along the cycle graph $G$ such that $\mu_i \leq \min(\mu_j, 1-\mu_j) \leq \frac{1}{2}$. Then 
\begin{align*}
    \dmet{x_i}{x_j} &\leq \dmet{x_i}{p} +  \dmet{p}{x_j} =  \mu_i s_i + \mu_j s_j \leq (\min(\mu_j, 1-\mu_j)  + \mu_j) \epsilon \leq \epsilon \\
    \dmet{x_i}{x_{j+1}} &\leq \dmet{x_i}{p} +  \dmet{p}{x_{j+1}} = \mu_i s_i + (1-\mu_j) s_j \leq (\min(\mu_j, 1-\mu_j)  + (1-\mu_j)) \epsilon  \leq \epsilon.
\end{align*}
Since $G$ is a minimal $\epsilon$-cycle, this implies $i$ is adjacent to both $j$ and $j+1$ in the minimal $\epsilon$-cycle. In other words, either $\# V(G) = 3$, or $\# V(G)  > 3$ and $i \in \{j, j+1\}$.

In either case, $f$ can only fail to be injective if the minimising geodesics $\gamma_i$ and $\gamma_{i+1}$ of adjacent vertices $i, i+1 \in G$ intersect at some point $p \neq x_{i+1}$. 
Without loss of generality we index the vertices such that 
\begin{equation*}
    s_i := |\gamma_i| = \dmet{x_i}{x_{i+1}} \leq  \dmet{x_{i+1}}{x_{i+2}}  = |\gamma_{i+1}| =: s_{i+1}.
\end{equation*}
Then \Cref{lem:intersection} implies that either $\imag \gamma_i \subseteq \imag \gamma_{i+1}$, or $p = x_i= x_{i+2}$, since we have assume $p \neq x_{i+1}$. However, since $f$ is injective on the vertices of $G$, we cannot have $x_i = x_{i+2}$, thus the only possibility is the former. In particular we have $x_i \in \imag \gamma_{i+1}$. Because $\gamma_{i+1}$ is a minimising geodesic between $x_{i+1}$ and $x_{i+2}$, we thus have $\dmet{x_{i+1}}{x_{i+2}} = \dmet{x_i}{x_{i+1}} + \dmet{x_i}{x_{i+2}}$; in particular because $\dmet{x_{i+1}}{x_{i+2}}  \leq \epsilon$, we have $\dmet{x_i}{x_{i+2}} \leq  \epsilon$. Because $G$ is a minimal $\epsilon$-cycle  this implies there's an edge between $i$ and $i+2$, i.e. the number of vertices on $G$ must be three. 

Hence we have shown that $f$ fails to be injective iff $\# V(G) = 3$ and one of the three vertices lie on the minimising geodesic between the other two. 
\end{proof}
Specialising to the four point case, the combination of~\Cref{lem:intersection} and~\Cref{prop:cycle_embed} yields the following observation.
\begin{corollary}
    Consider four points $\X = \{x_0, \ldots, x_3\} \subset M$ in a complete Riemannian manifold. The Vietoris-Rips persistence diagram of $\X$ contains at most one point. Such a point has non-trivial persistence $[b,d)$, iff the $\epsilon$-neighbourhood graph on $\X$ is a minimal $\epsilon$-cycle precisely when $\epsilon \in [b,d)$, and such a graph always admits an min-geodesic map that is an embedding $f: S^1 \simeq |G| \to M$.
\end{corollary}

\subsubsection{Empirical Observations on Compact Surfaces}
We consider the principal Vietoris-Rips persistence measures in dimension one for the compact surfaces with constant scalar curvature $M$ as described in~\Cref{ex:surface_coverings}. These are the torus, Klein bottle, the real projective plane, and the surface with genus two. We denote $\Gamma \cong \fungroup{M}$ to be the group that generates the covering space action of the universal covering $p: \tilde{M} \twoheadrightarrow M$. For each of these surfaces, we uniformly sample four points $\X = \sett{x_0, \ldots, x_3} \subset M$ independently and compute its Vietoris-Rips persistent homology in dimension one; if we have a generator with non-trivial persistence $[b,d)$, we then consider the $\epsilon$-nearest neighbour graph $G$ , for any $\epsilon \in [b,d)$, where the graph is a minimal $\epsilon$-cycle. For each point $x_i \in \X \subset M$, we find a point $\tilde{x}_i \in \fibre{p}{x_i}$ in the fibre of the universal covering of $M$, and compute for $ij \in G$ the element $t(ij) = \arg \inf_{g \in \Gamma} \tilde{\dmet{\tilde{x}_i}{\tilde{x}_j}}$. We recall~\Cref{prop:graph_convx_radius} which states the such data recovers a transition homomorphism on the induced covering of $|G|$ by some min-geodesic embedding, and that this min-geodesic embedding is generically uniquely determined by the position of the points~(\Cref{cor:UMG_generic}). Since the induced covering of $|G| \simeq S^1$ is obtained by pulling back the unviersal covering of $M$, we can recover the induced homomorphism on first homology $\tau: \Z \cong \homol{1}{G} \to \homol{1}{M} \cong \abel{\Gamma}$ by the folowing sum, due to~\Cref{prop:monodromy_homology_big}:
\begin{equation*}
    1 \mapsto \abel{t}(01) + \abel{t}(12) + \abel{t}(23) +\abel{t}(30) \in \abel{\Gamma}.
\end{equation*}
Note that while the product $t(01)t(12)t(23)t(30)$ recovers the monodromy homomorphism on the loop traced by the cycle graph in $M$, and that it is equivalent to the induced homomorphism $\fungroup{f}: \fungroup{G} \to \fungroup{M}$ up to isomorphism (\Cref{prop:monodromy_group_big}), the induced homomorphism depends on a choice of base point of the loop, and its lift to the covering space. We thus choose homology as the invariant to avoid such dependencies. This is especially relevant in the case where $M$ is the Klein bottle or the surface with genus two, whose fundamental groups are not abelian.

In these empirical experiments, we have performed one million independent four point samples on such surfaces, and labelled the homology class of the manifold that each persistent four point sample corresponds to. We visualise the empirical principal Vietoris-Rips persistence measures on the torus, klein bottle, real projective plane, and the surface with genus two in \Cref{fig:torus}, where we colour the diagram by a kernel density estimation on the empiricla measure to indicate the relevant abundance of points in each region of the persistence diagram. We display the full meausure, as well as component measures consisting of points belong to different the homology classes. In the diagram where the whole empirical measure is displayed, the empirical fraction of four point samples that have non-trivial persistence is labelled by $\phi_{\text{bar}}$ on the diagram where all the points are shown; on the individual diagrams depicting points belonging to different homology classes, the percentage of points belong to each class among all those with non-trivial persistence is also shown.

We recall from~\Cref{ex:surface_coverings_geom} that such surfaces have curvature inherited from their covering spaces; the torus and Klein bottle have zero scalar and sectional curvature; $\RP^2$ has $+1$ scalar and sectional curvature; and $\Sigma_2$ has $-1$ scalar and sectional curvature. In~\cite[Theorem 5.19]{Gomez2024CurvatureDiagrams}, they have proved that on simply connected manifolds with constant sectional curvature, the support of the principal Vietoris-Rips persistence set is bounded by an explicit formula that depends on curvature. We illustrate these bounds on the diagrams as orange dashed lines. In addition, we also mark the diameter of the surfaces on the diagrams. As expected, the persistent four point loops which are trivial in $M$ lie within the bounds, as they lift to their universal covering space and thus subject to the persistence bounds of their covering space. However, the four point cycles that belong to non-trivial classes can exceed those bounds, and can have far greater persistence. Furthermore, such cycles in non-trivial classes cannot have arbitrarily small birth death values, while the trivial cycles can approach the origin of the persistence diagram. 

While four point samples with non-trivial persistence are generally uncommon, among those that do have non-trivial persistence, most of them form non-trivial cycles. Non-trivial cycles tend to be generated by one or few generators. We conjecture this is due to the constrained geometry on the four point cycles due to the positive persistence as described in~\Cref{lem:4ptcycle_rips}.

\begin{figure}[h]
    \centering
    \includegraphics[width = 0.75\textwidth]{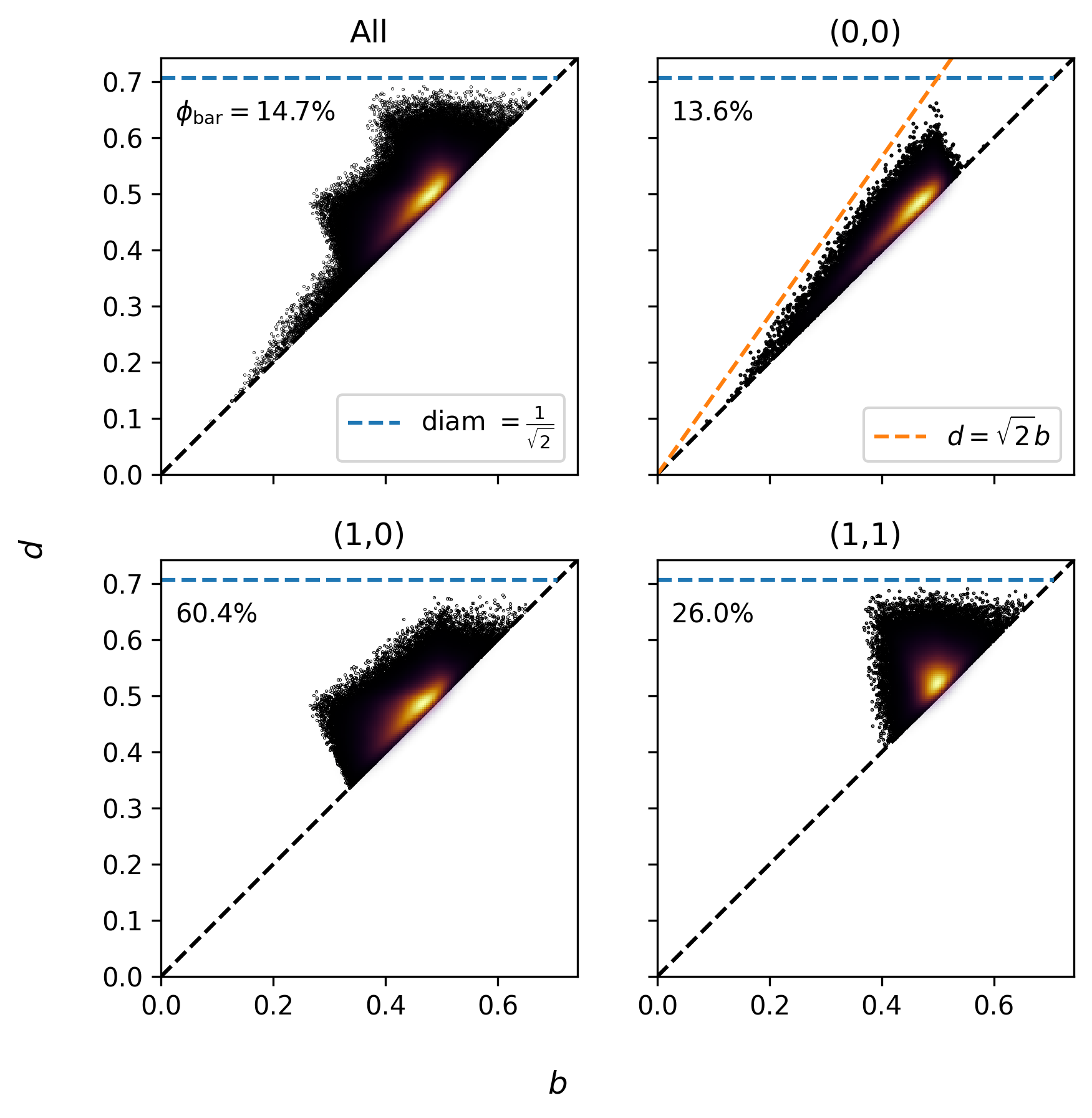}
    \caption{Empirical samples of the first principal Vietoris-Rips persistence measure on the flat torus. The homology classes of the four point cycles are represented by $(n,m) \in \Z \oplus \Z \cong \homol{1}{M}$. We group those $(1,0)$ and $(0,1)$ into one diagram $(1,0)$ due to the symmetry between the generators. }
    \label{fig:torus}
\end{figure}

\begin{figure}[h]
    \centering
    \includegraphics[width = 0.95\textwidth]{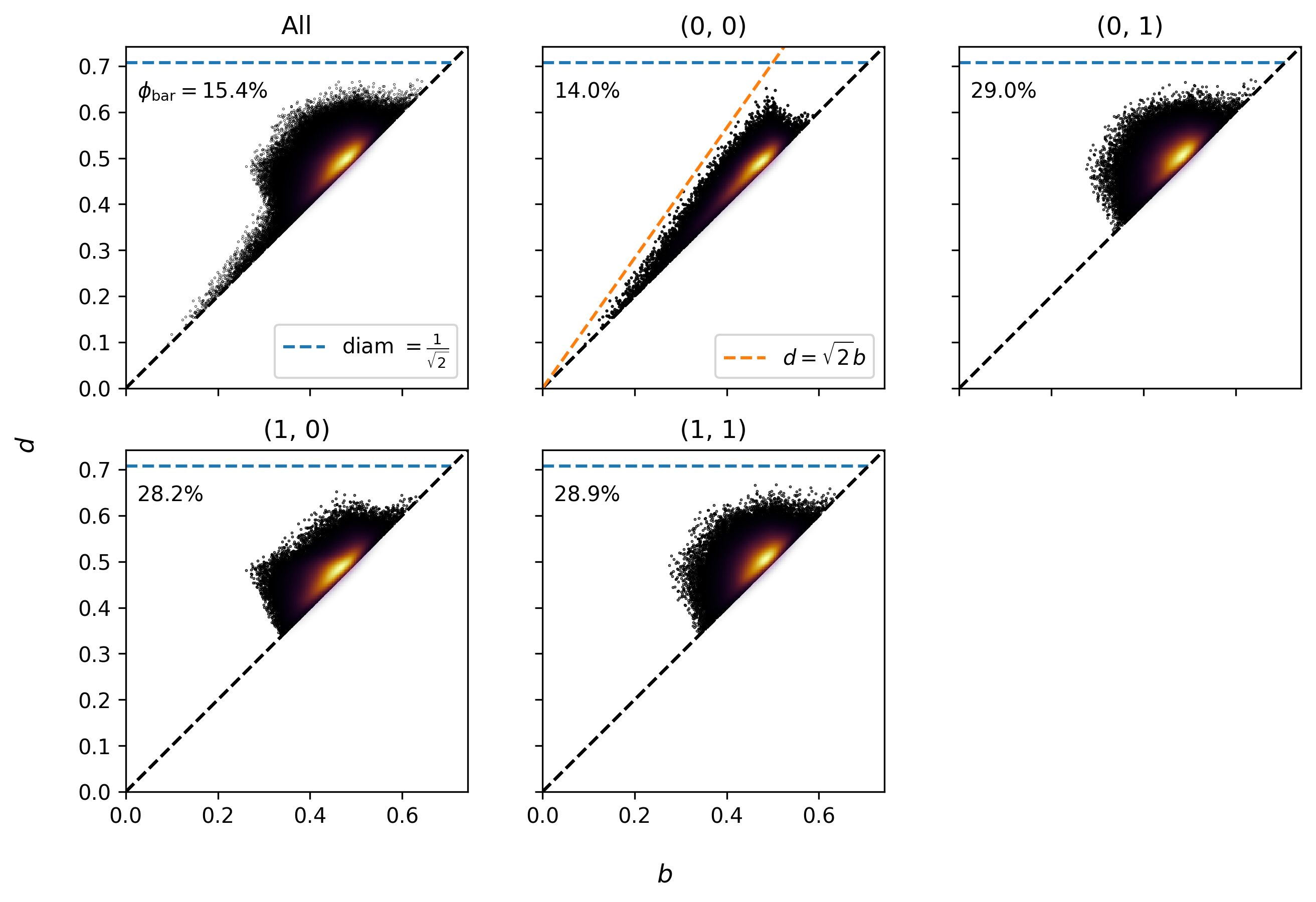}
    \caption{Empirical samples of the first principal Vietoris-Rips persistence measure on the flat Klein bottle. The homology classes of the four point cycles are represented by $(n,m) \in \Z \oplus \Z_2 \cong \homol{1}{M}$. }
    \label{fig:klein}
\end{figure}

\begin{figure}[h]
    \centering
    \includegraphics[width = 0.95\textwidth]{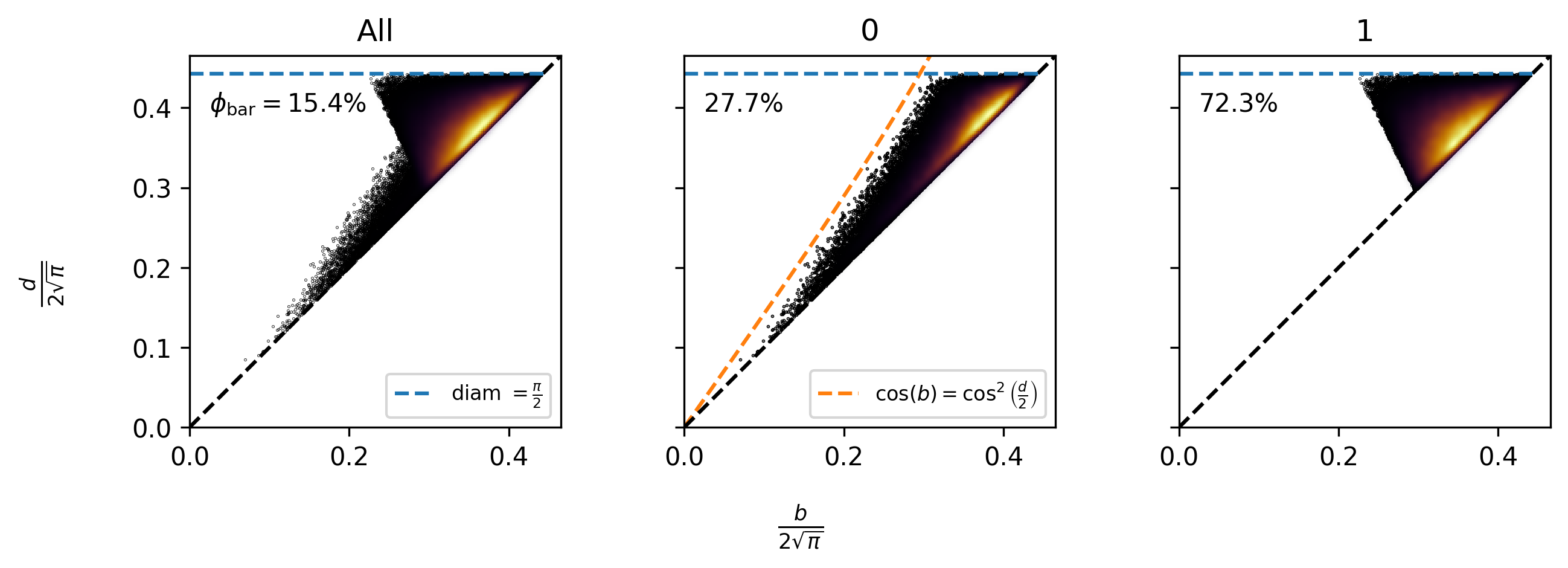}
    \caption{Empirical samples of the first principal Vietoris-Rips persistence measure on $\RP^2$. The homology classes of the four point cycles are represented by $a \in \Z_2 \cong \homol{1}{M}$. }
    \label{fig:rp2}
\end{figure}

\begin{figure}[h]
    \centering
    \includegraphics[width = 0.95\textwidth]{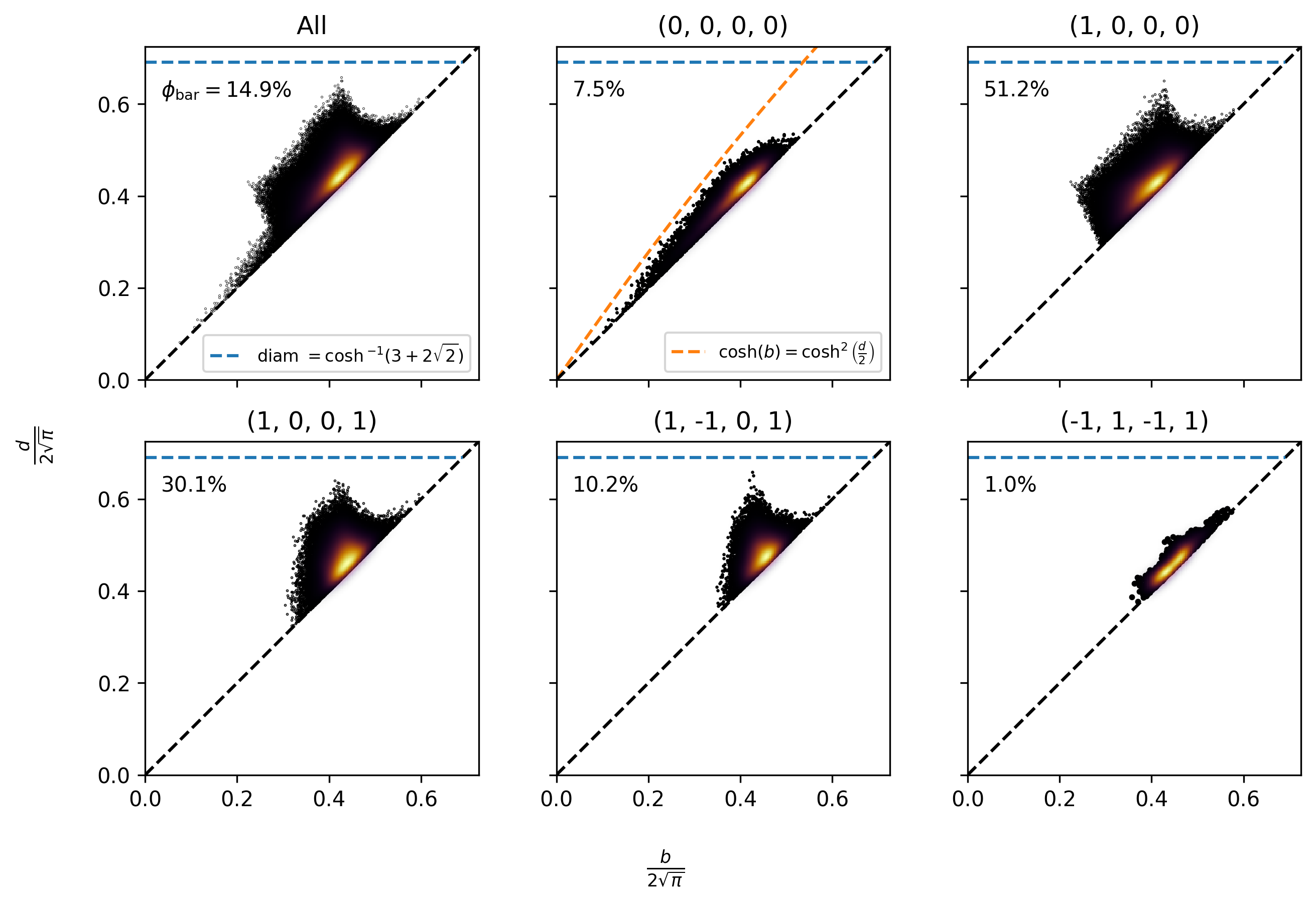}
    \caption{Empirical samples of the first principal Vietoris-Rips persistence measure on the orientable surface with genus two. The homology classes of the four point cycles are represented by $(n_0,\ldots, n_3) \in \Z^4  \cong \homol{1}{M}$. Due to symmetries betwen the the four generators, the homology class indicated in the diagrams represent classes up to permutation of $(n_0,\ldots, n_3)$. }
    \label{fig:g2}
\end{figure}

\subsection*{Acknowledgements}
The author would like to thank John Harvey for many fruitful and productive discussions throughout his work on this subject. The author is supported by a UKRI Future Leaders Fellowship [grant number MR/W01176X/1; PI: J. Harvey].

\bibliographystyle{plain}
\bibliography{references}

\begin{thebibliography}{10}

\bibitem{Armstrong2013BasicTopology}
Mark~Anthony Armstrong.
\newblock {\em {Basic topology}}.
\newblock Springer Science {\&} Business Media, 2013.

\bibitem{Bauer2023AVariations}
Ulrich Bauer, Michael Kerber, Fabian Roll, and Alexander Rolle.
\newblock {A unified view on the functorial nerve theorem and its variations}.
\newblock {\em Expositiones Mathematicae}, 41(4):125503, 12 2023.

\bibitem{Bjorner2003NervesGroups}
Anders Bj{\"{o}}rner.
\newblock {Nerves, fibers and homotopy groups}.
\newblock {\em Journal of Combinatorial Theory, Series A}, 102(1):88--93, 2003.

\bibitem{Brown2006TopologyGroupoids}
Ronald Brown.
\newblock {Topology and groupoids}, 2006.

\bibitem{Carmo1992RiemannianGeometry}
Manfredo Perdigão~do. Carmo.
\newblock {\em {Riemannian geometry}}.
\newblock Mathematics. Theory {\&} applications. Birkh{\"{a}}user, Boston, MA, 1992.

\bibitem{Cohen1989CombinatorialApproach}
Daniel~E Cohen.
\newblock {\em {Combinatorial group theory: a topological approach}}.
\newblock Number~14. Cambridge University Press, 1989.

\bibitem{Eltzner2021StabilityManifolds}
Benjamin Eltzner, Fernando Galaz-Garc{\textbackslash}'{\textbackslash}ia, Stephan Huckemann, and Wilderich Tuschmann.
\newblock {Stability of the cut locus and a central limit theorem for Fr{\'{e}}chet means of Riemannian manifolds}.
\newblock {\em Proceedings of the American Mathematical Society}, 149(9):3947--3963, 2021.

\bibitem{Fiorenza2012CechConstruction}
Domenico Fiorenza, Urs Schreiber, and Jim Stasheff.
\newblock {{\v{C}}ech cocycles for differential characteristic classes: an infinity-Lie theoretic construction}.
\newblock {\em Advances in Theoretical and Mathematical Physics}, 16(1):149 – 250, 2012.

\bibitem{Gallier2020DifferentialGroups}
Jean Gallier and Jocelyn Quaintance.
\newblock {\em {Differential Geometry and Lie Groups}}, volume~12.
\newblock Springer International Publishing, 2020.

\bibitem{Gallot2004RiemannianGeometry}
Sylvestre Gallot, Dominique Hulin, and Jacques Lafontaine.
\newblock {\em {Riemannian Geometry}}.
\newblock Springer Berlin Heidelberg, 2004.

\bibitem{Gomez2024CurvatureDiagrams}
Mario G{\'{o}}mez and Facundo M{\'{e}}moli.
\newblock {Curvature Sets Over Persistence Diagrams}.
\newblock {\em Discrete {\&} Computational Geometry}, 72(1):91--180, 7 2024.

\bibitem{Grove1993CriticalFunctions}
Karsten Grove.
\newblock {Critical point theory for distance functions}.
\newblock In {\em Proceedings of Symposia in Pure Mathematics}, volume~54, pages 357--385, 1993.

\bibitem{HatcherAlgebraicTopology}
Allen Hatcher.
\newblock {\em {Algebraic Topology}}.
\newblock Cambridge University Press, 1 edition edition, 2005.

\bibitem{Higgins1971CategoriesGroupoids}
Philip~J Higgins.
\newblock {\em {Categories and groupoids}}.
\newblock Citeseer, 1971.

\bibitem{Husemoller1966FibreBundles}
Dale Husem{\"{o}}ller.
\newblock {\em {Fibre bundles}}, volume~5.
\newblock Springer, 1966.

\bibitem{LeeIntroductionManifolds}
John~M Lee.
\newblock {\em {Introduction to Smooth Manifolds}}, volume 218 of {\em Graduate Texts in Mathematics}.
\newblock Springer New York, 2012.

\bibitem{Lee2018IntroductionManifolds}
John~M Lee.
\newblock {\em {Introduction to Riemannian manifolds}}, volume~2.
\newblock Springer, 2018.

\bibitem{Leray1945SurRepresentations}
Jean Leray.
\newblock {Sur la forme des espaces topologiques et sur les points fixes des repr{\'{e}}sentations}.
\newblock {\em Journal de Math{\'{e}}matiques Pures et Appliqu{\'{e}}es}, 24:95--167, 1945.

\bibitem{Mantegazza2002HamiltonJacobiManifolds}
Carlo Mantegazza and Andrea~Carlo Mennucci.
\newblock {Hamilton—Jacobi Equations and Distance Functions on Riemannian Manifolds}.
\newblock {\em Applied Mathematics {\&} Optimization}, 47(1):1--25, 2002.

\bibitem{Munkres1974TopologyMunkres}
James~R Munkres.
\newblock {\em {Topology; a first course [by] james r. munkres}}.
\newblock Prentice-Hall, 1974.

\bibitem{Onus2022QuantifyingComplexes}
Adam Onus and Vanessa Robins.
\newblock {Quantifying the homology of periodic cell complexes}.
\newblock {\em arXiv preprint arXiv:2208.09223}, 2022.

\bibitem{Onus2023ComputingWindows}
Adam Onus and Primoz Skraba.
\newblock {Computing 1-Periodic Persistent Homology with Finite Windows}.
\newblock {\em arXiv preprint arXiv:2312.00709}, 2023.

\bibitem{Riehl2017CategoryContext}
Emily Riehl.
\newblock {\em {Category theory in context}}.
\newblock Courier Dover Publications, 2017.

\bibitem{Stillwell2012ClassicalTheory}
John Stillwell.
\newblock {\em {Classical topology and combinatorial group theory}}, volume~72.
\newblock Springer Science {\&} Business Media, 2012.

\bibitem{Weibel1994AnAlgebra}
Charles~A Weibel.
\newblock {\em {An introduction to homological algebra}}.
\newblock Number~38. Cambridge university press, 1994.

\bibitem{Whitehead1962CONVEXPATHS}
J.~H.~C. Whitehead.
\newblock {Convex Regions in the Geometry of Paths}.
\newblock In {\em Differential Geometry}, pages 223--232. Elsevier, 1962.

\end{thebibliography}
\appendix
\section{Carriers}\label{app:carrier}
Let $K$ be a simplicial complex, and $X$ be a topological space. An assignment $\sigma \mapsto C(\sigma) \subset X$ of simplices to subsets of $X$ is a \emph{carrier} if $C(\sigma) \subseteq C(\tau)$ for all $\sigma \subseteq \tau$ in $K$. That is, if a simplex $\sigma$ is in the face of a simplex $\tau$, then the subset assigned to $\tau$ by a carrier $C$ must contain that of $\tau$. Of particular interest are continuous maps $f: |K| \to X$ carried by $C$. We say a map $f$ is carried by $C$ if $f(|\sigma|) \subseteq C(\sigma)$. 

In particular, if we have a cover $\cU$ of space $X$, then the assignment of $\sigma \in \nerve{\cU}$ to the subset  $\sigma \mapsto \bigcup_{i \in \sigma} U_i$ is a carrier of the nerve in $X$.

Recall we say $\gamma: I\to X$ is carried by an edge path $\upsilon: I_m \to \nerve{\cU}$, if the assignment of vertex element $k$ to $U_{\nu(k)}$ carries $\gamma$. 

\pathcarrier*
\begin{proof}
    This follows from Lebesgue's number lemma. Since $\cV = \sett{\fibre{\gamma}{\gamma \cap U_i}}_{i \in \cI}$ is an open cover of the interval, there exists a small enough $\delta > 0$ such that intervals of length $\delta$ are contained in some member of $\cV$. Choose positive natural number $m > 2/\delta$ and consider open ball of radii $\delta/2$ around evenly spaced points $0, 1/m, \ldots, 1-1/m, 1$. Since we have chosen $m$ to be sufficiently large, the collection of balls cover the interval, and each ball is point is contained in one element of $\cV$. Let $\upsilon: \{0,\ldots, m\} \to \cI$ be a map such that $\upsilon(k)$ is the index of a cover element such that $\fibre{\gamma}{\gamma \cap U_{\upsilon(k)}}$ contains the open ball around $k/m$. Since $\delta/2$ is greater than the separation $1/m$ between the successive centres of the balls, $U_{\upsilon(k)} \cap U_{\upsilon(k+1)}$ is non-empty, and thus $\upsilon(k)\upsilon(k+1)$ is an edge in $\nerve{\cU}$ and $\upsilon(0)\cdots \upsilon(m)$ is an edge path on $\nerve{\cU}$. Moreover, by construction, $\gamma((\frac{k}{m}, \frac{k+1}{m}))$ is contained in $U_{\upsilon(k)}$ and $U_{\upsilon(k+1)}$, and $ \fun{\gamma}{\frac{k}{m}} \in U_{\upsilon(k)}$. Thus the edge path  $\upsilon(0)\cdots \upsilon(m)$ carries $\gamma$. If $\upsilon(0) \neq i$ or $\upsilon(m) \neq j$, consider the sequence of vertices  $i_0 \cdots i_{m+2}$ where $i_0 = i$, $i_{m+2} =j$, and $i_k = \upsilon(k-1)$ for $k = 1,\ldots, m+1$. Since $U_{\upsilon(0)} \cap U_i \ni \gamma(0)$ and $U_{\upsilon(m)} \cap U_j \ni \gamma(1)$, the sequence $i_0 \cdots i_{m+2}$ is an edge path. If we reparametrise the path $\gamma$ by a piecewise linear function 
    \begin{equation}
    a: t \mapsto 
        \begin{cases}
            0               & t \in [0, \frac{1}{m+2}] \\
            \frac{m+2}{m} t-\frac{1}{m} & t \in [\frac{1}{m+2},\frac{m+1}{m+2}] \\
            1               & t \in [\frac{m+1}{m+2}, 1]
        \end{cases},
    \end{equation}
    then the edge path satisfy the conditions for the edge path to carry $\gamma$ as stipulated in \Cref{def:pathcarrier}. By construction, for $k = 1,\ldots, m+1$ $\gamma \circ a(\frac{k}{m+2}) = \gamma(\frac{k-1}{m}) \in U_{\upsilon(k-1)} = U_{i_k}$ , and  $\gamma\circ a((\frac{k}{m+2}, \frac{k+1}{m+2})) = \gamma((\frac{k-1}{m}, \frac{k}{m})) \subset U_{\upsilon(k-1)} \cup U_{\upsilon(k)}   = U_{i_k} \cup U_{i_{k+1}}$; while for $k=0$ or $k = m+2$, we have $\gamma \circ a([0, \frac{1}{m+2})) = \gamma(0) \in U_i = U_{i_0}$, and $\gamma \circ a((\frac{m+1}{m+2},1]) = \gamma(1) \in U_j = U_{i_{m+2}}$.
\end{proof}

\begin{lemma} \label{lem:hmtpy_implies_ehmtpy}
    Let $\cU = \{ U_i \}_{i \in \cI}$ be an open cover of a space $X$. Suppose two paths $\gamma, \gamma' : I \to X$ from $x$ to $y$ are carried by edge paths $\eta, \eta'$ on the nerve of the cover from $i$ to $j$ respectively. If $[\gamma]
    = [\gamma'] \in \Fgpd{X}(x,y)$, then $[\eta] = [\eta'] \in \Egpd{\nerve{\cU}}(i,j)$.
\end{lemma}
\begin{proof}
    Let $F: I \times I \to X$ be the homotopy from $\gamma(t) = F(t,0)$ to $\gamma'(t) = F(t,1)$. We consider the pullback of the open cover $\cV = \fibre{F}{\cU}$ on $X$. Consider the Freudenthal triangulation $L$ of $I \times I$ on an $n \times n$ square lattice of vertices on $I \times I$. In the proof below, we show that for a sufficiently large $n$, there exists an assignment $\upsilon: L_0 \to \cI$, such that:
    \begin{itemize}
        \item For every vertex $u \in L_0$, the image of its star is contained in some element of the cover $F(\opstar{u}) \subset U_{\upsilon(u)}$; in particular,
        \item For the vertices along the sides of $\{0\} \times I$ and $\{1\} \times I$ of the square where $F(0,\bullet) = x$ and $F(1,\bullet) = y$, $\upsilon$ assigns them to $i$ and $j$ respectively; and 
        \item The assignment $\upsilon$ restricted to the top and bottom edges are simplicial maps from $I_n \to \nerve{\cU}$ whose images are $\eta$ and $\eta'$ respectively.
    \end{itemize}
    We first explore the consequences of the existence of $L$, $\upsilon$ satisfying the conditions above before proving that it indeed exists. If the first condition is satisfied, then $\upsilon$ defines a simplicial map $\upsilon: L \to \nerve{\cU}$. Consider some simplex $\sigma \in L$. Because  $F(\opstar{u}) \subset U_{\upsilon(u)}$, the image of the simplex $F(|\sigma|) \in \bigcap_{u \in \sigma} U_{\upsilon(u)}$ is in the intersections of the cover elements assigned to contain the stars. Thus $\bigcap_{u \in \sigma} U_{\upsilon(u)} \neq \emptyset$, and $\upsilon$ sends subsets of vertices spanning a simplex in $L$ to subsets of vertices spanning a simplex in $\cN(\cU)$. In other words, $\upsilon$ defines a simplicial map.

    Furthermore, if the latter two conditions are satisfied, then the edge homotopy between the edge paths from $(0,0)$ to $(1,0)$ along the bottom edge $I \times \{0\}$ and the edge path along the other side of the boundary $\{0\} \times I$,$I \times \{1\}$,$\{1\} \times I$, in the Freudenthal triangulation is mapped to an edge homotopy between $\eta$ and $\eta'$ in $\nerve{\cU}$, as simplicial maps take edge homotopies to edge homotopies. This proves the claim of this lemma. 

    We now show that $\upsilon, L$ satisfying the conditions above exists for sufficiently large resolution $n$ of the Freudenthal triangulation. This is essentially a variation of Lebesgue's lemma to cater for the specific boundary conditions we face in tour special case. We recycle and adapt elements of the proof of Lebesgue's number theorem as presented in \cite[Lemma 27.5]{Munkres1974TopologyMunkres}. Applying Lebesgue lemma there is a $\delta > 0$ such that each subset of $I\times I$ with diameter less than $\delta$ is contained in some cover element of $\cV$. We now require the Freudenthal triangulation to be sufficiently fine (i.e. $n$ sufficiently large), such that diameters of stars are less than such a $\delta$, and also satisfy the size constraints set by the boundary conditions listed above.
    
    We now consider the conditions on vertex stars on the sides of $\{0\} \times I$ and $\{1\} \times I$ of the square. We consider the left hand side; without loss of generality the same argument applies for the right hand side as well. Recall  $F(0, \bullet) = x$, and $x \in U_i$ since an edge path starting from $i$ carries $F(t,0)$. Let $V_i = \fibre{F}{U_i}$, which contains $\{0\} \times I$. Consider the distance function to the complement of $V_i$ in $I \times I$, and restrict it to $\{0\} \times I$. Since $V_i$ is an open neighbourhood of $\{0\} \times I$, such a function $d$ is a positive function on $I$ and attains a positive minimum $\delta_0 > 0$. Let $\delta_1 > 0$ be similarly defined for $\{1\} \times I$, where it is the minimum of the distance of $\{1\} \times I$ to the complement of $V_j = \fibre{F}{U_j}$. If we take $n$ to be sufficiently large so that $n > \max 1+(1/\delta_0,1/\delta_1)$, then the stars of vertices along the sides $\{0\} \times I$ and $\{1\} \times I$ are respectively contained in $V_i$ and $V_j$. In other words, $F$ sends the stars of vertices along the left and right hand sides of $I \times I$ to $U_i$ and $U_j$ respectively. We can set $\upsilon$ on such vertices on the left and right hand side to $i$ and $j$ respectively. 

    Now we show that with a sufficiently large $n$ we can have an assignment $\upsilon$ to vertices on the top and bottom sides of $I \times I$, such that $\upsilon$ maps those paths onto $\eta'$ and $\eta$, while having stars satisfying the containment condition $F(\opstar{u}) \subset U_{\upsilon(u)}$. For simplicity we focus on $I \times \{0\}$ as the argument for the other side is the same. Recall the definition of $\eta$ carrying $F(\bullet, 0)$, which implies we can express $I \times \{0\}$ is the geometric realisation of a path graph with $m+1$ vertices, such that there is some sequence of cover elements $i_0, \ldots, i_{m}$ such that $F([k/m, (k+1)/m],0)$ is covered by $\{U_{\eta(k)}, U_{\eta(k+1)}\}$, and $F(k/m,0) \in U_{i_k}$. Applying Lebesgue's lemma, each interval $[k/{m}, (k+1)/m]$ can be further triangulated by a finite path graph with sufficiently many $n_k+1$ vertices, such that the star of each vertex is mapped within $U_{\eta(k)}$ or $U_{\eta(k+1)}$ by $F(\cdot, 0)$. Furthermore, from the definition of $\eta$ carrying $F(\cdot, 0)$, the star of the vertices at each end point $k/m$ and $(k+1)/m$ is mapped with  $U_{\eta(k)}$ and $U_{\eta(k+1)}$ respectively. Taking $l$ to be the maximum over all $n_k$'s for $k= 0,\ldots, m$, we can triangulate $I \times \{0\}$ by a path graph with $ml + 1$ vertices such that every $kl$\textsuperscript{th} vertex has the image of its star contained in $U_{\upsilon(k)}$; while vertices between the $kl$\textsuperscript{th} and $(k+1)l$\textsuperscript{th} has the image of their stars contained in either $U_{\upsilon(k)}$ or $U_{\upsilon(k+1)}$. Thus, for any $l$ sufficiently large, we have a simplicial map from the path graph with $ml + 1$ vertices to $\nerve{\cU}$, such that the image of this map is edge homotopic to $\eta$ itself. 

    Similarly, for any sufficiently large $l'$ , the path $F(\cdot, 1)$ is realised by a finite path graph with $m'l'+1$ vertices, and there is a simplicial map from the path graph to $\nerve{\cU}$ whose image is edge homotopic to $\eta'$. If we choose a sufficiently large $l'$ as a multiple of $m$, and $l$ as a multiple of $m'$, we can set $n-1 = m'l' = ml$ for the grid size of the Freudenthal triangulation $L$, such that for vertices $u \in L$, we have an assignment $\upsilon: L_0 \to \cI$ satisfying the bullet pointed conditions above. 
\end{proof}

\begin{proposition}\label{prop:def_Sgh}
    Let $X$ be a space with an open cover $\cU = \{U_i \}_{i \in \cI}$. Given a map $s: X \to \cI$ satisfying $x \in U_{s(x)}$, we have a unique, well-defined \emph{snapping} groupoid homomorphism $S: \Fgpd{X} \to \Egpd{\nerve{\cU}}$, where $[\gamma] \in \Fgpd{X}(x,y)$ is sent to a class of edge paths $[\eta] \in \Egpd{\nerve{\cU}}(s(x), s(y))$, represented by some edge path $\eta$ carrying $\gamma$.
\end{proposition}

\begin{proof}
    \Cref{lem:pathcarrier} shows that every path from $x$ to $y$ is carried by some edge path on the nerve from $s(x)$ to $s(y)$. \Cref{lem:hmtpy_implies_ehmtpy} shows that an assignment of paths between $x$ and $y$ to edge paths between $s(x)$ and $s(y)$ that carry them defines a unique map of homotopy classes of paths in $\Fgpd{X}(x,y)$ to homotopy classes of edge paths in $\Egpd{\nerve{\cU}}(s(x), s(y))$. We verify the conditions for this assignment to be a groupoid homomorphism. First, the constant path at any point $x$ is carried by a constant edge path at $s(x)$, thus the identity element of the fundamental group is sent to the edge group. Consider then the concatenation of two paths $\gamma$ and $\gamma'$ from $x$ to $y$ and $y$ to $z$ respectively, where they are carried  by edge paths $\eta$ and $\eta'$ from $r(x)$ to $r(y)$ and $r(y)$ to $r(z)$ respectively. Since the composition of edge paths $\eta$ and $\eta'$ is itself an edge path that carries that concatenation of $\gamma$ and $\gamma'$, the assignment respects the composition of morphisms in the respective groupoids. Thus, $S$ is a groupoid homomorphism uniquely determined by $s$.
\end{proof}

We show that there is a groupoid homomorphism $R: \Egpd{\nerve{\cU}} \to \Fgpd{X}$ -- the realisation of homotopy classes of edge paths as homotopy classes of continuous paths in $X$ -- that is adjoint to the snapping 
$S$, if we assume $\cU$ is a good cover. We recall the carrier lemma as it appears in Vidit Nanda's lecture notes (?).

\begin{lemma}[Carrier]\label{lem:carrier} Let $K$ be a simplicial complex. Suppose $C$ is a carrier of $K$. If $C(\sigma)$ is contractible for each $\sigma \in K$, then:
\begin{enumerate}[label=(\alph*)]
    \item \label{lem:carrier_b} There exists a continuous map $f: |K| \to X$ carried by $C$;
    \item \label{lem:carrier_a} Any two maps $f,g: |K| \to X$ that are both carried by $C$ are homotopic;
    \item We can choose a homotopy $F: |K| \times [0,1] \to X$ between $f$ and $g$ such that $F_t = F(\bullet, t): |K| \to X$ is also carried by $C$ for all $t \in [0,1]$
\end{enumerate}
\end{lemma}
For the carrier of the nerve $\nerve{\cU}$ given by $C(\sigma) = \bigcup_{i \in \sigma} U_i$, the nerve lemma that is a consequence of the carrier lemma ensures that $C$ is contractible. 

\begin{lemma}[Nerve lemma] If $\cU$ is a good cover of a space $X$, then $|\nerve{\cU}| \simeq X$.\end{lemma}

Since $\{U_i\}_{i \in \sigma}$ is a good cover of $C(\sigma)$, and simplices are contractible, $C(\sigma)$ is also contractible. Thus the assumptions of the carrier lemma holds if we choose $\cU$ to be a good cover. We first show that any edge path on the nerve has a representative path which it carries on $X$, if $U_i \cap U_j \neq \emptyset$ implies $U_i \cup U_j$ is path-connected.  

\begin{lemma}\label{lem:existence_epath_rep} Let $\cU = \{U_i\}_{i \in \cI}$ be a good cover of $X$. For any given choice of $r: \cI \to X$ satisfying $r(i) \in U_i$, each edge path $\eta: I_m \to \nerve{\cU}$ carries some path $\gamma_\eta : I \to X$, parametrised such that 
\begin{align}
    \textstyle \fun{\gamma_\eta}{\frac{k}{m}} &= r(\eta(k)) & k = 0,\ldots,m \\
    \textstyle  \fun{\gamma_\eta}{\left(\frac{k}{m},\frac{k+1}{m} \right)} &\subset U_{\eta(k)} \cup U_{\eta(k+1)} & k = 0,\ldots,m-1.
\end{align}
\end{lemma}
\begin{proof}
    Consider an edge $ij \in \nerve{\cU}$, and the subset $U_i \cup U_j \subset X$. Due to the nerve lemma, $U_i \cup U_j$ is path-connected and contractible. Thus, there is a path from $r(i)$ to $r(j)$ that stays within $U_i \cup U_j$. Concatenating such paths for each successive edge in the edge path $\eta$ yields the path $\gamma_\eta$ as described.
\end{proof}

We now show that representatives of edge homotopic paths are homotopic. 

\begin{lemma}\label{lem:epathrep_homotopy} Assume the conditions of \Cref{lem:existence_epath_rep}. Let $\eta$ and $\eta'$ be edge homotopic paths on the nerve $\nerve{\cU}$ from $i$ to $j$. Then any representative path $\gamma_\eta$ and $\gamma_{\eta'}$ of $\eta$ and $\eta'$ respectively are homotopic rel end points. 
\end{lemma}
\begin{proof} For an edge-homotopic pair of edge paths $\eta$ and $\eta'$, consider the finite sequence of edge paths $\eta = \eta_0, \ldots, \eta_m = \eta'$ such that each successive pair of edge paths $\eta_k$ and $\eta_{k+1}$ by an elementary contraction one way or another. Consider a pair of edge paths $a: I_m \to \nerve{\cU}$ and $b: I_{m+1} \to \nerve{\cU}$ where $a$ is an elementary contraction of $b$; that is, there is some $0 \leq i < m$, such that 
    \begin{itemize}
        \item $a(j) = b(j)$ for $j \leq i$;
        \item $a(j) = b(j+1)$ for $j \geq i+1$; and
        \item $b(i)b(i+1)b(i+2)$ spans a simplex in $K$.
    \end{itemize}
Consider representative paths $\gamma_a$ and $\gamma_b$ of edge paths $a$ and $b$. Let $\gamma_{a|[j,k]}$ denote the restriction of the the representative path to the representative sub-edge path from $j$ to $k$. Consider restrictions to edges $\gamma_{a|[j,j+1]}$. For $0 \leq j < i$, since $\gamma_{a | [j,j+1]}, \gamma_{b | [j,j+1]} \subset U_{a(j)} \cup U_{a(j+1)}$, and $\subset U_{a(j)} \cup U_{a(j+1)}$ is contractible, $[\gamma_{a | [j,j+1]}]= [\gamma_{b | [j,j+1]}]$. For identical reasons, for $j > i$, we have $[\gamma_{a | [j,j+1]}] = [\gamma_{b | [j+1,j+2]}]$. 

Because $\cU$ is a good cover, and $b(i)b(i+1)b(i+2)$ spans a simplex in the nerve, the nerve lemma implies $\bigcup_{j = i}^{i+2}U_{b(i)}$ is contractible. Thus, the restriction of representative paths $\gamma_{a \rvert [i,i+1]}$ and $\gamma_{b | [i,i+2]}$ to those subsequences are homotopic rel end points. Thus $[\gamma_{a \rvert [i,i+1]}]  = [\gamma_{b | [i,i+2]}]$. Concatenating all sub edge paths, we thus deduce that 
$$[\gamma_a] = [\gamma_{a|[0,1]}] \cdots [\gamma_{a|[m-1, m]} = [\gamma_{b|[0,1]}] \cdots [\gamma_{b|[m, m+1]} = [\gamma_b].$$ 
\end{proof}

\begin{lemma}\label{lem:epathrep_good} Assume the conditions of \Cref{lem:existence_epath_rep}.
\begin{itemize}
    \item Any path $\gamma: I\to X$ carried by $\eta$ is homotopic to $\gamma_\eta$;
    \item Any path $\gamma$ from $r(i)$ to $r(j)$ carried by $\eta$ is homotopic to $\gamma_\eta$, \emph{rel end points}.
\end{itemize}
\end{lemma}
\begin{proof}
    Since $\gamma_\eta$ is carried by $\eta$, the first statement of the lemma is a consequence of the carrier lemma. 

    For the second statement, we need to show that we can choose a homotopy such that the end points of the paths are fixed. Consider a homotopy $F: I \times I \to X$ between $\gamma$ and $\gamma_\eta$ carried by $\eta$. Consider the paths $\alpha(t) = F(0,t)$ and $\omega(t)= F(1,t)$.   Since both $\gamma$ and $\gamma_\eta$ have end points $r(i)$ and $r(j)$, paths $\alpha(t)$ and $\omega(t)$ loops based at $r(\eta(0))$ and $r(\eta(m))$ respectively. Because $F(\cdot, t): I \to X$ is carried by $\eta$ for all $t$, the loops $\alpha$ and $\omega$ are respectively contained in $U_{\eta(0)}$ and $U_{\eta(m)}$, which are both contractible. Thus $[\alpha] = e$ and $[\omega] = e$ in $\fungroup{X, r\eta(0)}$ and $\fungroup{X, r\eta(m)}$ respectively. Since the homotopy $F$ implies $[\gamma] = [\alpha][\gamma_\eta] \inv{[\omega]} \in \Fgpd{X}(r(i), r(j))$, we thus have $[\gamma] = [\gamma_\eta]$.
    \end{proof}

\begin{proposition} \label{prop:def_Rgh}
    Let $X$ be a space with a good cover $\cU = \{U_i\}_{i \in \cI}$. For any map $r: \cI \to X$ such that $r(i) \in U_i$, we have a unique groupoid homomorphism $R: \Egpd{\nerve{\cU}} \to \Fgpd{X}$, where $R(i) = r(i)$, and a class of edge paths $[\eta] \in \Egpd{\nerve{\cU}}(i,j)$ is assigned to the class of paths $[\gamma] \in \Fgpd{X}(r(i),r(j))$ homotopic rel end points to ones that $\eta$ carries. 
\end{proposition}
\begin{proof}
    We first show that the assignment of an edge path $\eta$ from $i$ to $j$ to a path $\gamma$ from $r(i)$ to $r(j)$ which it carries induces a unique well-defined map from classes of edge paths in $\Egpd{\nerve{\cU}}(i,j)$ to classes of paths in $\Fgpd{X}(r(i), r(j))$.

    If we assign $\eta$ to any path $\gamma$ that it carries, then $\gamma$ and any choice of representative path $\gamma_\eta$ as described in \Cref{lem:existence_epath_rep} are homotopic rel end point, due to \Cref{lem:epathrep_good}. Since any two edge paths that are edge homotopic have homotopic representative paths rel end points (\Cref{lem:epathrep_homotopy}, paths between $r(i)$ and $r(j)$ that are carried by edge-homotopic edge paths from $i$ to $j$ are homotopic. Thus, we have a well-defined map from $\Egpd{\nerve{\cU}}(i,j)$ to $\Fgpd{X}(r(i), r(j))$.

    We now check that this assignment defines a groupoid homomorphism. We first check that it sends the identity element of vertex groups to vertex groups. Since the constant edge path at $i\in \nerve{\cU}$ is sent to some loop based at $r(i)$ contained in $U_i$, the good cover condition ensures that this loop is trivial. 
    We then check that the assignment respects composition. Since the concatenation of paths carried by edge paths is carried by the concatenation of edge paths, this assignment thus takes $R([\eta_1])R([\eta_2])  = R([\eta_1\eta_2]) = R([\eta_1][\eta_2])$. Thus $R$ is a groupoid homomorphism. 
\end{proof}

We now show that  $R: \Egpd{\nerve{\cU}} \leftrightarrows \Fgpd{X} : S$ is an equivalence of categories. 

\equivalenceEFG*

\begin{proof}
    Let $R,S$ be the groupoid homomorphisms as given in \Cref{prop:def_Rgh,prop:def_Sgh}.
    
    Consider $RS([\gamma])$; the class of paths $[\gamma]$ is first sent to the unique class of edge paths represented by an edge path $\eta \in \Egpd{\nerve{\cU}}(s(x), s(y))$ that carries $\gamma$. Applying $R$ then sends that edge class to the class of paths in $\Fgpd{X}(rs(x), rs(y))$ carried by $\eta$. Since $\gamma$ and a representative path $\gamma_\eta$ of $RS([\gamma])$ are both carried by $\eta$, \Cref{lem:carrier} implies there is a homotopy $F: I \times I \to X$ from $F(\cdot, 0)= \gamma$ to $F(\cdot, 1) = \gamma_\eta$, where each $F(\cdot, t)$ is also carried by $\eta$. Thus, we have two paths $\alpha =  F(0, \cdot) \subset U_{\eta(0)}$ and $\omega =  F(1, \cdot) \subset U_{\eta(m)}$. Because $\cU$ is a good cover, $\alpha$ belongs to the unique class of paths carried by $U_{\eta(0)}$ from $x$ to $rs(x)$. Similarly, $\omega$ belongs to the unique class of paths carried by $U_{\eta(m)}$ from $y$ to $rs(y)$.
    Defining $\varsigma_\cdot \in \Fgpd{X}(\cdot, rs(\cdot))$ to be the unique homotopy class of paths carried by $U_{s(\cdot)}$, because $F$ is a homotopy, we have   
    \begin{equation*}
        [\gamma] = [\alpha][\gamma_\eta]\inv{[\omega]}\implies [\gamma] = \varsigma_x RS([\gamma]) \inv{\varsigma_y}.
    \end{equation*}

    Consider then $SR([\eta])$. In the first step, $R$ sends $[\eta] \in \Egpd{\nerve{\cU}}(\eta(0),\eta(m))$ to the class of paths $[\gamma] \in \Fgpd{X}(r(\eta(0)), r(\eta(m)))$, represented by path $\gamma$ which $\eta$ carries; and in the second step, $S$ sends $[\gamma]$ to a class of edge paths from $sr(\eta(0))$ to $sr(\eta(m))$ that carries $\gamma$. Since $r(\cdot)
    \in U_\cdot$ and $r(\cdot) \in U_{sr(\cdot)}$, we have a pair of edges $sr(\eta(0))\eta(0)$ and $sr(\eta(m))\eta(m)$ in the nerve $\nerve{\cU}$. Moreover, since $\gamma(0) \in U_{sr(\eta(0))} \cap U_{sr(\eta(0))}$ and  $\gamma(1) \in U_{sr(\eta(m))} \cap U_{sr(\eta(m))}$, the edge path $sr(\eta(0))\eta(0)\cdots \eta(m) sr(\eta(m))$ carries a reparametrisation of $\gamma$, while $SR([\eta])$ carries $\gamma$. Since homotopy is invariant w.r.t. reparametrisation, \Cref{lem:hmtpy_implies_ehmtpy} implies $SR[\eta] = [sr(\eta(0)) \eta(0)][\eta][\eta(m)sr(\eta(m))]$. Thus we have a natural isomorphism $\varrho: SR \Rightarrow \iden_{\Egpd{\nerve{\cU}}}$, where $\varrho(i) = [sr(i)i]$.
\end{proof}

\section{Proof of \Cref{prop:grp_homology}}
\label{app:grouphomology}
\grouphomology*
\begin{proof}
    For the zeroth homology group, since all morphisms of a groupoid with one element has the same source and target, with image of $\partial_1$ is zero; thus $ \HGrpd{0}{G; A} = (\ChGrpd{0}{G} \otimes A)/(\imag (\partial_1 \otimes \iden_A) \cong \Z \otimes A \cong A$.

    For the first homology group, let us consider $A = \Z$ first. Recall $\ChGrpd{1}{G} = \Z[G]$ is the free abelian group on elements of $G$ as a basis set. We use the notation $\sum_i n_i \langle g_i \rangle$ to denote elements of $G$ as a basis in a formal sum.  Denote $\partial_2 \ChGrpd{2}{G} =\mathcal{B}_1(G)$. Since $\partial_1$ is the zero homomorphism, we have an exact sequence 
    \begin{equation*}
         0 \to \mathcal{B}_1(G) \hookrightarrow \ChGrpd{1}{G} \twoheadrightarrow \HGrpd{1}{G} \to 0.
    \end{equation*}
    We now note that $\mathcal{B}_1(G)$ is precisely the relations on symbols $\langle g \rangle $ such that the quotient $\HGrpd{1}{G}:= \ChGrpd{1}{G}/\mathcal{B}_1(G) $ is isomorphic to $\abel{G}$.

    To show this, this let us consider the epimorphism $\eta: \ChGrpd{1}{G} \twoheadrightarrow \abel{G}$, where $\sum_i n_i \langle g_i \rangle  \mapsto \sum_i n_i g_i $. If $\ker \eta = \mathcal{B}_1(G)$, then we have an isomorphism $\HGrpd{1}{G}:= \ChGrpd{1}{G}/\mathcal{B}_1(G) = \ChGrpd{1}{G}/\ker \eta  \cong \abel{G}$. The boundary subgroup $\mathcal{B}_1(G)$ is generated by formal sums over elements of $G$, of the form $\langle f \rangle  + \langle g \rangle - \langle f \cdot g \rangle$.  Direct evaluation shows that $\mathcal{B}_1(G) \subseteq \ker \eta$:
    \begin{align*}
        \eta( \langle f \rangle + \langle g \rangle - \langle f \cdot g \rangle ) &= \eta(\langle  f \rangle + \langle g \rangle ) - \eta(\langle f \cdot g \rangle) = f + g - (f+g) = 0.
    \end{align*}
    To show $\mathcal{B}_1(G) \supseteq \ker \eta$, we note that elements $\sum \langle g \rangle $ of $\ker \eta$ are those that can be indexed such that $g_1 \cdots g_m= e$. If we rewrite $\sum \langle g \rangle $ accordingly as thus,
    \begin{align*}
        \langle g_1 \rangle  + \cdots + \langle g_m \rangle
        &= \left(\langle g_1 \rangle + \langle g_2 \rangle  - \langle g_1 \cdot g_2 \rangle \right) + \left(\langle g_1 \cdot g_2 \rangle +  \langle g_3 \rangle - \langle (g_1 \cdot g_2)  \cdot g_3 \rangle \right) + \\
        &\cdots + \left(\langle  g_1\cdots g_{m-1}\rangle + \langle g_m \rangle - \langle (g_1 \cdots g_{m-1})\cdot  g_m \rangle \right) + \cancelto{\langle e \rangle}{ \langle  g_1 \cdots g_m \rangle },
    \end{align*}
    we have written $\sum \langle g \rangle $ as an element of $\partial_2$  (since $\langle e \rangle = \langle e \rangle  + \langle e \rangle - \langle e \cdot e\rangle)$. we thus have $\ker \eta \subseteq \mathcal{B}_1(G)$. Hence $\eta$ induces an isomorphism $\HGrpd{1}{G} := \ChGrpd{1}{G} \diagup \mathcal{B}_1(G) \xrightarrow{\cong} \abel{G}$.

    For homology with $A$ coefficients, because $(-) \otimes A$ is a right exact functor, we obtain the following exact sequence:
     \begin{equation*}
         \mathcal{B}_1(G) \otimes A \hookrightarrow \ChGrpd{1}{G;A} \overset{\eta \otimes \iden_A}{\twoheadrightarrow} \abel{G} \otimes A \to 0.
    \end{equation*}
   Furthermore, as $ \mathcal{B}_1(G) \otimes A =  \imag \partial_2 \otimes A = \imag (\partial_2 \otimes 1_A)$, we have an isomorphism  $\HGrpd{1}{G;A} := \ChGrpd{1}{G;A}/ \imag (\partial_2 \otimes 1_A) \xrightarrow{\cong} \abel{G}\otimes A$, where $[\sum_i \langle g_i \rangle \otimes a_i] \mapsto \sum_i g_i \otimes a_i.$
\end{proof}

\end{document}